\RequirePackage[l2tabu, orthodox]{nag}
\documentclass[a4paper,reqno]{amsart}

\usepackage[T1]{fontenc}
\usepackage[tt/osf=false]{cfr-lm}

\usepackage{amsmath}
\usepackage{amsthm}
\usepackage{amssymb}
\usepackage[foot]{amsaddr}
\usepackage{mathtools}
\usepackage{enumitem}

\DeclareMathOperator*{\bigast}{\mathop{\raisebox{-0.6ex}{\scalebox{2.5}{$\ast$}}}}

\usepackage{tikz}
\usetikzlibrary{arrows}
\usetikzlibrary{cd}
\usepackage{xcolor}
\usepackage{graphicx}
\usepackage{pinlabel}

\usepackage{subcaption}

\DeclareCaptionLabelFormat{xx}{#2}
\usepackage[backend=bibtex, style=numeric, maxnames=8, isbn=false, sortcites=true, backref]{biblatex}
\AtEveryBibitem{\clearlist{language}}
\usepackage{hyperref}
\hypersetup{
    colorlinks,
    linkcolor={red!50!black},
    citecolor={blue!50!black},
    urlcolor={blue!80!black}
}

\usepackage{zref-clever}
\usepackage{microtype}

\zcsetup{countertype={thm=theorem}}
\newtheorem{thm}{Theorem}
\numberwithin{thm}{section}

\AddToHook{env/lem/begin}{\zcsetup{countertype={thm=lemma}}}
\newtheorem{lem}[thm]{Lemma}
\AddToHook{env/prp/begin}{\zcsetup{countertype={thm=proposition}}}
\newtheorem{prp}[thm]{Proposition}
\AddToHook{env/cor/begin}{\zcsetup{countertype={thm=corollary}}}
\newtheorem{cor}[thm]{Corollary}
\zcRefTypeSetup{conjecture}{
Name-sg = Conjecture ,
name-sg = conjecture ,
Name-pl = Conjectures ,
name-pl = conjectures ,
}
\AddToHook{env/conj/begin}{\zcsetup{countertype={thm=conjecture}}}
\newtheorem{conj}[thm]{Conjecture}
\AddToHook{env/prop/begin}{\zcsetup{countertype={thm=proposition}}}

\theoremstyle{definition}
\AddToHook{env/defn/begin}{\zcsetup{countertype={thm=definition}}}
\newtheorem{defn}[thm]{Definition}
\AddToHook{env/ex/begin}{\zcsetup{countertype={thm=example}}}
\newtheorem{ex}[thm]{Example}
\zcRefTypeSetup{cons}{
Name-sg = Construction ,
name-sg = construction ,
Name-pl = Constructions ,
name-pl = constructions ,
}
\AddToHook{env/cons/begin}{\zcsetup{countertype={thm=cons}}}

\theoremstyle{remark}
\zcRefTypeSetup{rem}{
Name-sg = Remark ,
name-sg = remark ,
Name-pl = Remarks ,
name-pl = remarks ,
}
\AddToHook{env/rem/begin}{\zcsetup{countertype={thm=rem}}}
\newtheorem{rem}[thm]{Remark}
\newtheorem*{notation}{Notation}

\DeclarePairedDelimiter{\abs}{\lvert}{\rvert}

\newcommand{\stephead}[1]{\vspace{.2\baselineskip}\noindent\textit{{#1}.}}

\newcommand{\df}{\textit}

\newcommand{\union}{\cup}
\newcommand{\inter}{\cap}

\newcommand{\mc}{\mathcal}

\newcommand{\mf}{\mathfrak}

\newcommand{\B}{\mathbb{B}}
\renewcommand{\H}{\mathbb{H}}
\renewcommand{\P}{\mathbb{P}}
\newcommand{\R}{\mathbb{R}}
\newcommand{\C}{\mathbb{C}}

\newcommand{\Z}{\mathbb{Z}}
\newcommand{\IS}{\mathbb{S}}

\newcommand{\CB}{\mathcal{CB}}
\newcommand{\F}{\mathcal{F}}

\DeclareMathOperator{\PSL}{\mathsf{PSL}}
\DeclareMathOperator{\SL}{\mathsf{SL}}

\DeclareMathOperator{\Isom}{Isom}

\DeclareMathOperator{\Teich}{Teich}
\DeclareMathOperator{\QH}{QH}
\DeclareMathOperator{\Fix}{Fix}
\DeclareMathOperator{\Hom}{Hom}

\DeclareMathOperator{\tr}{tr}
\DeclareMathOperator{\Hol}{Hol}
\DeclareMathOperator{\Sing}{Sing}
\DeclareMathOperator{\Sign}{sign}
\DeclareMathOperator{\Int}{int}
\DeclareMathOperator{\CAT}{CAT}

\begin{document}
\zcsetup{noabbrev, cap}

\title[Changing topological type of compression bodies]{Changing topological type\\of compression bodies\\through cone manifolds}
\author[A. Elzenaar]{Alex Elzenaar}
\address{School of Mathematics\\ Monash University\\ Melbourne\\ Australia}
\email{alexander.elzenaar@monash.edu}
\thanks{During the period that this work was undertaken, the author was supported by an Australian Government Research Training Program Scholarship.
Parts of the manuscript were written during a visit to Jeroen Schillewart, supported by the NZ Marsden Fund through grant UoA-2122. Part of
this research was supported by the USA National Science Foundation under Grant No. DMS-2424139 while the author was in residence at the Simons Laufer
Mathematical Sciences Institute in Berkeley, California during the Spring 2026 semester.}

\subjclass[2020]{Primary 57K35; Secondary 20F65, 30F40, 30F60, 57K32, 58H15}
\keywords{Kleinian groups, compression bodies, cone manifolds, unknotting tunnels, combination theorems, function groups}

\begin{abstract}
  Homeomorphism types of compression bodies form the vertices of a graph where two vertices are joined by an edge
  if one compression body is obtained by gluing a $2$-handle onto the other. Motivated by earlier work of Lackenby and Purcell
  on geodesicity of unknotting tunnels for hyperbolic links, we show that it is possible to realise all of the edges in the graph
  of compression bodies by paths of cone manifold holonomy groups such that the handle that is glued in is obtained as a limit of singular arcs of cone angle increasing from $0$ to $ 2\pi $.
  We apply standard techniques from the theory of $ \CAT(0) $ spaces, and do not rely on the harmonic deformation theory of Hodgson and Kerckhoff. Along the way we prove
  a generalisation of a classic theorem of Koebe and Maskit on existence of function groups which implies
  existence results for reflex angled hyperbolic cone structures on a wide range of compression bodies.
\end{abstract}

\maketitle

\section{Introduction}
It was first recognised by Thurston~\cite{thurston82} that the generic $3$-manifold, which cannot be naturally decomposed by cutting along embedded $2$-spheres
and tori, admits a hyperbolic structure: for instance, almost all knot and link complements are hyperbolisable. Thurston also proposed a rigidity theorem, subsuming earlier
results of Mostow and Prasad~\cite{mostow,prasad73} and Marden~\cite{marden74g}. This result, now known as the \df{ending lamination theorem}, was eventually proved by Brock,
Canary, and Minsky using very deep results of Agol, Calegari, Gabai, and others. It reduces the classification of hyperbolic $3$-manifolds to a classification of their
homeomorphism types and the conformal structures on their ends---including both conformal structures on boundary components, and asymptotic conformal structures on ends
not associated to boundary components. Detailed statements and references to the papers comprising the proof may be found in~\cite[\S 5.7]{marden}.

The theory of hyperbolic $3$-manifolds is intimately related to the $\PSL(2,\C)$ representation theory of their holonomy groups. Every complete hyperbolic $3$-manifold
is of the form $ \H^3/\Gamma $ where $ \Gamma \leq \PSL(2,\C) $ is a torsion-free discrete group acting by isometries on hyperbolic space. We use the term \df{Kleinian group}
to mean a discrete subgroup of $ \PSL(2,\C) $ (possibly with elements of finite order, in which case the quotient is an orbifold not a manifold). The representation
entirely determines the manifold, and thus the ending laminations. Conversely a choice of ending laminations and homeomorphism type determines a representation up to conjugacy. The
exact relationship between the two points of view is subtle and we can pass explicitly from one world to the other only in very simple cases.

Let $S$ be a compact surface of finite type (so $S$ is topologically determined by its genus $g$). The set of all representations $ \rho : \pi_1(S) \to \PSL(2,\C) $, modulo inner
automorphism in $ \PSL(2,\C) $, is called the \df{character variety} $ X(\pi_1(S)) $~\cite{przytycki00}. Most of these representations are not discrete. The representations which \emph{are} discrete
and non-elementary (and so represent hyperbolic $3$-orbifolds) come in two kinds, depending on whether or not they admit quasiconformal deformations.
The non-rigid groups are those groups $ \Gamma $ such that $ \H^3/\Gamma $ has some end with a nontrivial Teichm\"uller space and they lie in an open subset
of $ X(\pi_1(S)) $ made up of discrete representations. The rigid groups, on the other hand, are either finite covolume groups (with no conformal ends), or are infinite covolume
groups with only thrice-marked spheres at conformal ends. The finite covolume groups are isolated in representation space by Mostow--Prasad rigidity,
and the infinite covolume groups either lie on the boundary of quasiconformal deformation spaces, or are orbifold groups with only thrice-marked spheres at conformal ends and
arbitrarily complicated nonconformal ends and can be joined to the boundary of a quasiconformal deformation space by a cone deformation as outlined in~\cite[\S 7.2]{chk}.

By the ending lamination theorem and associated results, we can enumerate all of the possible topological types of non-rigid discrete representations
that can occur in $ X(\pi_1(S)) $. However, we cannot easily convert this geometric and topological information to algebraic data and extract concrete information about the positions and relative
arrangement of the islands of discreteness inside the character variety. In addition, the indiscrete groups in $ X(\pi_1(S)) $ lying between these islands can be very badly behaved. In various
settings~\cite{gallo00,mathews11,mathews12} it is known that almost all of these groups are holonomy groups of cone manifolds---a cone manifold is a generalisation of an orbifold
where the cone angles are allowed to take any value, not just submultiples of $ 2\pi $, see e.g.~\cite[Chapter 3]{chk}. However, the coarse
structure of these cone manifolds is chaotically dependent on the trace parameters: small trace deformations give rise to wild changes in the singular structure of the cone manifolds,
completely unlike the behaviour of deformations of discrete groups where the large-scale structure of the quotient space is always preserved.

The aim of this paper is to construct paths of representations with controlled global geometry that lead from one deformation space to another; this is related to, but does
not depend on, the extensive theory of cone deformations studied by Hodgson and Kerckhoff~\cite{hodgson98} and later developed by others including Wei\ss~\cite{weiss05,weiss07,weiss13},
Kojima~\cite{kojima98}, Mazzeo and Montcouquiol~\cite{mazzeo11}, and Montcouquiol~\cite{montcouquiol13}. To summarise the results of those works, it is known that cone angles can be deformed
locally within the interval $ (0,2\pi) $ (that is, if a cone manifold has a singular arc with cone angle $ \theta \in (0,2\pi) $ then there exists some $ \epsilon > 0 $ such that the hyperbolic
metric can be varied continuously to deform this cone angle within the interval $ (\theta - \epsilon, \theta + \epsilon) $---this is the \df{local rigidity theorem}) and globally within in the
interval $ (0,\pi) $ (this is the \df{global rigidity theorem}). There are significant obstructions to any general result allowing continuous deformations of cone angles throughout the full
interval $ (0,2\pi) $, essentially because at angle $ \pi $ the hyperbolic structure can limit onto a cone manifold modelled on a non-hyperbolic Thurston geometry; see the references in Cooper,
Danciger, and Wienhard~\cite{cooper18}. In order to obtain general results, one must have additional geometric information, such as bounds on the radius of a cylindrical neighbourhood around
the cone arc to be deformed.  Such estimates have been studied before~\cite{futer22,futer22b,hodgson05}, but our methods to perform global cone deformations in the current paper do not involve the study of the
hyperbolic metric. Rather, they derive from the classical theory of combination theorems in geometric and combinatorial geometry.

Combination theorems of one kind or another have been of great interest ever since the original Klein combination theorem for free products~\cite{klein82}.
Combination theorems for Fuchsian groups were first given by Fenchel and Nielsen~\cite[\S IV.18]{fenchel}, for Kleinian groups by Maskit~\cite{maskit,maskit93}, and for $\CAT(0)$-spaces by
Bridson and Haefliger~\cite[Chapter~II.11]{bridson_haefliger}. In recent years, combination theorems have been given for higher-rank analogues of Kleinian groups~\cite{dey25,danciger24}. Our
combination theorems may be found in \zcref{sec:combination_theorems}.

\subsection{The main result}
A \df{compression body} of genus $n$ is a $3$-manifold obtained from a thickened genus $n$ surface $ \Sigma_n \times I $
by gluing in $2$-handles along curves on one end and then gluing in $3$-handles to remove any essential spheres that are produced~\cite[\S3.3]{bonahon02}. The end along which no discs are glued
is called the \df{compression end}. If the compression end is identified with $ \Sigma_n $ then boundary inclusion gives a canonical surjective map $ \pi_1(\Sigma_n) \to \pi_1(M) $.
This, if $M$ has a complete hyperbolic structure then the holonomy representation $ \Hol(M) $ lies in the $\PSL(2,\C)$-character variety of $ \pi_1(\Sigma_n) $.
Our main result is the following.
\begin{thm}\label{thm:main}
  Let $M$ and $M'$ be a pair of (topological) compression bodies such that $ M $ is obtained from $M'$
  by gluing a $2$-handle onto the non-compression end. Then there exists a smooth path $\gamma$ parameterised by $ [0,2\pi] $ in the character variety of $ M' $ such that:
  \begin{itemize}
    \item $ \gamma(0) $ is a complete hyperbolic structure on $ M' $;
    \item $ \gamma(2\pi) $ is a complete hyperbolic structure on $ M $; and
    \item For $ \theta \in (0,2\pi) $, $ \gamma(\theta) $ is the holonomy group of a hyperbolic cone manifold, topologically
          supported on $ M $, with a cone arc of angle $ \theta $ which when drilled produces a manifold homotopic to $ M' $.
  \end{itemize}
  This path continuously realises the gluing of a $2$-handle by a cone deformation.
\end{thm}

An alternative way of phrasing this result is that the graph of compression bodies joined by cone deformations is connected.
\begin{cor}\label{thm:connectedness}
  Let $ \mc{G}_{\text{CB}} $ be the graph with vertex set `the set of homeomorphism types of hyperbolic compression bodies', so
  that two vertices are joined by an edge if there exists a smooth path of cone manifolds modelling a deformation from a compression body
  of one type to the other. Then the graph $ \mc{G}_{\text{CB}} $ is connected.
\end{cor}
\begin{proof}
  Consider the graph with vertices corresponding to homeomorphism types of compression bodies, with two vertices joined by an
  edge if the two compression bodies are related by gluing in a $2$-handle~\cite{maher21}. This graph is connected, essentially by definition of compression bodies.
  \zcref[S]{thm:main} implies that this graph is isomorphic to the graph $\mathcal{G}_{\text{CB}}$, since it implies that every $2$-handle gluing
  is realised by a cone deformation.
\end{proof}

In \zcref{sec:triangle_gps} we study the basic building blocks (cone triangle groups) used to build the intermediate cone manifold groups of \zcref{thm:main},
and we prove the theorem itself in \zcref{sec:maskit_decomposition}. One interesting thing about the proof is that, since it is performed by physically constructing
fundamental domains via describing necessary combinatorics of circles in the complex plane, small perturbations of the path still give cone manifolds of known topological
type. This shows that, around the path we have constructed in the character variety of $ M' $, there is a small tube with a product structure $ [0,2\pi] \times \B^r $
where each level-disc is identified canonically with a small ball (of possibly high co-dimension) in the Teichm\"uller space of the compression end of $M'$. It follows from our proof that this tube
has bounded-below diameter (since the limit groups at $ \theta = 0 $ and $ \theta \to 2\pi $ are also flexible), and so actually \zcref{thm:main} implies that every
geometrically infinite hyperbolic metric in a small neighbourhood of the constructed geometrically finite metric of $M$ also has the property that the core tunnel is geodesic.
Unfortunately it seems very difficult to effectively bound the level-sets of this tube in the Teichm\"uller metric. We do note that if we could extend this tube to
a foliation $ [0,2\pi] \times \Teich(S_{g,0}) $ then this would give a product structure on a subset of $ X(\Hol(M'),\PSL(2,\C))$ with $2\pi$-level set equal to $ \QH(\Hol(M'))  $,
which seems incompatible with the known bumping results for $ \partial \QH(\Hol(M')) $~\cite{ohshika20,magid12,bromberg11}.

Following the proof of the main result, we remark in \zcref{sec:koebemaskit} that using the same techniques one can prove a strengthening of the Koebe--Maskit
existence theorem for function groups, and in particular that one can construct cone manifold structures on compression bodies with the property that many angles
are reflex. The original theorem was an important piece of the classical project of Ahlfors--Bers--Maskit to classify of Kleinian groups via their
quasiconformal deformations, and this new result forms part of a project of the author to recover results about cone manifolds from classical Kleinian group theory.

\subsection{Relation to unknotting tunnels}\label{sec:intro_lp}
The final section of the paper is devoted to the study of an explicit cone manifold deformation which is motivated by the theory of unknotting
tunnels, so we spend a few paragraphs now to explain this.

An \df{unknotting tunnel system} of size $n$ for $M$ is a collection $ \alpha_1, \ldots, \alpha_n $
of embedded ideal arcs in $ M $ such that $ M \setminus (\union_{j=1}^n \alpha_j) $ is a handlebody, and the \df{tunnel number} of $M$ is the minimal size of an unknotting tunnel
system for $M$. Once it is known that a given manifold $ M $ has tunnel number $n$, then one can try to count the number of possible unknotting tunnel systems with that size,
or describe how those systems interact with a metric on $ M $---in particular, whether they are isotopic to geodesics in that metric.
It is a variation on this last question, posed in general by Lackenby and Purcell~\cite{lackenby14} and Burton and Purcell~\cite{burton14}, which was our original motivation.

One context where geodesicity of unknotting tunnels has been well-studied is the setting of link complements with tunnel number $1$. Given a link $ \mf{k} \subset \IS^3 $, an unknotting tunnel
for $ \mf{k} $ is an embedded open arc $ \alpha \subset \IS^3 \setminus \mf{k} $ with the property that $ \IS^3 \setminus (\mf{k} \union \alpha) $ is a genus two
handlebody. It was conjectured by Adams~\cite{adams95} that if $ \IS^3 \setminus \mf{k} $ is hyperbolic, then every unknotting tunnel is isotopic
to a geodesic in the hyperbolic metric. In the same paper the conjecture was proved for the special case that $ \mf{k} $ has two components, and subsequently
the conjecture was confirmed for the special case of the upper and lower tunnels of $2$-bridge knots by Adams and Reid~\cite{adams96}. These results
rely on obtaining hands-on estimates on the injectivity radius around the tunnel in the hyperbolic metric on $ \IS^3 \setminus \mf{k} $.

A different and very flexible method to prove conjectures of the form `a fixed arc $ \alpha $ is isotopic to a geodesic' is to show that there exists a continuous path of cone manifolds with singular
locus of angle $ \theta $ along $\alpha$, where $ \theta $ ranges over the entire closed interval $ [0,2\pi]$. In a previous work~\cite{elzenaar25ej}, the author used this
technique to give an alternative proof that the upper unknotting tunnel of a highly twisted $2$-bridge link is isotopic to a geodesic in the hyperbolic metric.
We will utilise it here in relation to the following conjecture of Lackenby and Purcell~\cite{lackenby14}:
\begin{conj}\label{conj:lackenby_purcell}
  Let $ M $ be a hyperbolic $3$-manifold that is homeomorphic to a compression body with genus $2$ compression end
  and a single other end of genus $1$; i.e.\ it is obtained by taking a thickened torus and adjoining a $1$-handle.
  Then the core arc of the $1$-handle, extended to meet the genus $1$ end twice as in \zcref{fig:lp_group}, is isotopic to a geodesic in $M$.
\end{conj}

\begin{figure}
  \centering
  \begin{subfigure}{.5\textwidth}
    \labellist
    \small\hair 2pt
    \pinlabel {torus end} [l] at 149 76
    \pinlabel {core tunnel} [r] at 230 55
    \endlabellist
    \centering
    \includegraphics[width=\textwidth]{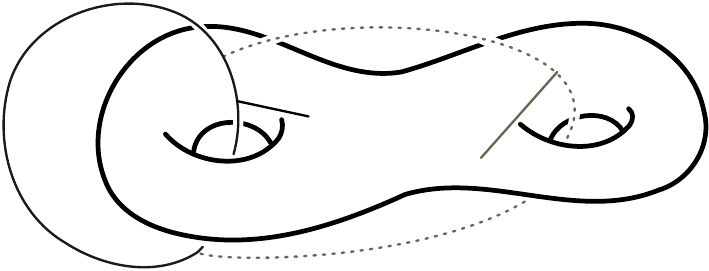}
    \caption{$(1;2)$-compression body.\label{fig:lp_group}}
  \end{subfigure}\hfill
  \begin{subfigure}{.4\textwidth}
    \labellist
    \small\hair 2pt
    \pinlabel $\gamma$ [b] at 132 158
    \pinlabel $\beta$ [b] at 180 154
    \pinlabel $\alpha$ [b] at 213 158
    \endlabellist
    \centering
    \includegraphics[width=\textwidth]{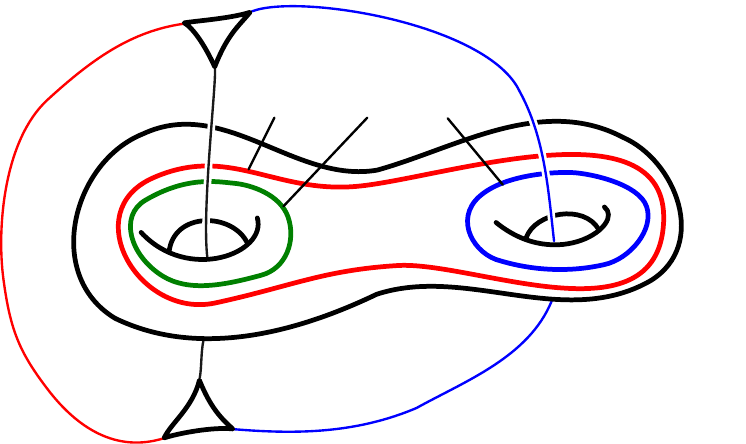}
    \caption{$\mc{M}(2)$-manifold.\label{fig:m2_group}}
  \end{subfigure}
  \caption{The two hyperbolic manifolds of interest, viewed from the interior. Solid arcs and surfaces are drilled out, and the dotted arc in (\textsc{a}) is embedded in the manifold.}
\end{figure}

A compression body such as the one in \zcref{conj:lackenby_purcell}, i.e.\ with one genus $2$ end and one genus $1$ end, will be called a \df{$(1;2)$-compression body}~\cite{lackenby14}. The core arc acts as an
analogue of an unknotting tunnel---it is an arc which, when drilled, makes the manifold into a thickened surface rather than a handlebody. This drilling then supports a certain kind of hyperbolic
metric:
\begin{defn}\label{defn:m2mfd}
  A $ \mc{M}(2)$-manifold is a geometrically finite hyperbolic compression body, homeomorphic to the thickened genus $2$ surface, so that one end is a compact surface and the
  other end is a pair of thrice-punctured spheres arranged as in \zcref{fig:m2_group}.
\end{defn}
\begin{rem}
  The deformation space of $ \mc{M}(2)$-manifold holonomy groups is a \df{Maskit slice} of genus $2$ surfaces, c.f.\ Marden~\cite[\S 5.10]{marden}.
\end{rem}

An immediate consequence of \zcref{thm:main} is that there exists a path (in fact, a large family of paths) connecting $ (1;2)$-compression body groups to $ \mc{M}(2)$-manifold groups so that the isomorphism
and quasi-isometry types of the intermediate groups are controlled.
\begin{thm}\label{thm:weakthm}
  Let $G$ be the holonomy group of a geometrically finite hyperbolic $ (1;2)$-compression body. There exists a path of cone manifold holonomy
  groups $ p : [0,2\pi] \to \Hom(G, \PSL(2,\C)) $ such that $ p(2\pi) = G $, $ p(0) $ is an $ \mc{M}(2)$-group, and
  for every $ \theta \in (0,2\pi) $ the group $ p(\theta) $ is the holonomy group of a geometrically finite cone manifold structure on a $(1;2)$-compression body with
  a single singular arc of angle $ \theta $ along the arc with meridian labelled $ \beta $ in \zcref{fig:m2_group}.
\end{thm}
In \zcref{sec:explicit_deformation} we physically construct the fundamental domains for every cone manifold along the path, and give matrix realisations for their holonomy groups.
Though it is doubtless unsurprising to experts that such paths exist, to our knowledge this is the first explicit example of cone
deformations between deformation spaces of infinite covolume Kleinian groups. Some related work is that of Akiyoshi~\cite{akiyoshi18}, who has given families of cone deformations from spaces
of punctured torus groups to the space of rank two parabolic elementary groups. In addition, Yoshida~\cite{yoshida22} has given explicit cone deformations between some finite covolume
groups; and the author~\cite{elzenaar25ej} has given cone deformations joining highly twisted $2$-bridge link groups to genus two Schottky space.

The motivation for \zcref{thm:weakthm} is that a strengthening of it would prove \zcref{conj:lackenby_purcell} for geometrically finite groups:
\begin{conj}[{\zcref[S]{thm:weakthm} with conformal rigidity}]\label{conj:conformal_rigidity}
  Let $G$ be the holonomy group of a geometrically finite hyperbolic $ (1;2)$-compression body. There exists a path of cone manifold holonomy
  groups $ p : [0,2\pi] \to \Hom(G, \PSL(2,\C)) $ satisfying the three conditions of \zcref{thm:weakthm} together with
  \begin{enumerate}[start=4]
    \item The complex structure on the genus $2$ end is held constant: that is, the map $ [0,2\pi] \to \Teich(S_{2,0}) $ induced by looking at the genus $2$ surface is a constant map.
  \end{enumerate}
\end{conj}
\begin{lem}
  \zcref[S]{conj:conformal_rigidity} implies \zcref{conj:lackenby_purcell} for geometrically finite manifolds.
\end{lem}
\begin{proof}
  Fix a geometrically finite hyperbolic structure on a $(1;2)$-compression body. By the Marden isomorphism theorem~\cite{marden74g},
  the structure is uniquely determined by the complex structure on the genus $2$ end. By \zcref{conj:conformal_rigidity}, there exists a cone deformation from an $ \mc{M}(2) $ group
  (namely, the unique such group with the same complex structure on the genus $2$ end) so that at every point the core tunnel is represented by the axis of an elliptic
  element and is therefore isotopic to a geodesic. By continuity of the metric, the limit of the core tunnels is geodesic. It is also embedded, since by the structure of cone manifolds
  for each cone angle $ \theta \in (0,2\pi) $ we can find an embedded tube in the manifold around the elliptic axis which lifts to a piece of the fundamental domain such that the elliptic
  pairs up two faces that meet at a single edge with dihedral angle $ \theta $; as $ \theta \to 2\pi $ this edge vanishes, the two adjacent faces become identified in the interior of the manifold,
  and the axis of the elliptic limits onto a geodesic curve embedded in the interior of the polyhedron so it descends to an embedded geodesic.
\end{proof}

\begin{rem}
  For a geometrically infinite structure, by the Density Theorem proved by Namazi and Souto~\cite{namazi12} and Ohshika~\cite{ohshika11} there is a sequence of
  geometrically finite $ (1;2)$-compression body metrics converging to it, but the limit geodesic might be self-intersecting since there is no uniform control over
  the fundamental polyhedra of the groups in the limiting sequence. As such, it may no longer be isotopic to the core tunnel.
\end{rem}

Unfortunately the methods we use in this paper cannot prove \zcref{conj:conformal_rigidity}.

\subsection{Acknowledgments}
I thank Jeroen Schillewaert and Ari Markowitz for helpful discussions on geometric group theory.
I thank Jessica Purcell and other readers for helpful comments on this manuscript.

\section{Triangle groups and cone manifolds}\label{sec:triangle_gps}
Our proof of \zcref{thm:connectedness} will go via decomposing the groups of interest up into the simplest possible pieces, deforming those pieces, and then gluing
them back up. In this section we will describe these simple pieces, which will be triangle groups.

When $ p,q,r \in \Z_{\geq 0}$ satisfy $ 1/p + 1/q + 1/r < 1 $, there is a geodesic triangle in $ \H^2 $ with dihedral angles $ \pi/p $, $ \pi/q $, and $ \pi/r $.
Classically, a $(p,q,r)$-\df{triangle group} is a Coxeter group generated by the reflections in the sides of this triangle. We will be interested only
in the index $2$ subgroup consisting of orientation-preserving elements, which has presentation
\begin{displaymath}
  \langle x, y : x^p = y^q = (xy)^r \rangle
\end{displaymath}
and has a fundamental domain consisting of a single quadrilateral~\cite[\S 10.6]{beardon}. These groups are
discrete, orientation-preserving subgroups of $ \Isom^+(\H^2) $. By a series of classical results in conformal analysis, such groups can be realised as discrete subgroups
of $ \PSL(2,\C) $ which preserve the two components of $ \hat{\C} \setminus \IS^1 $, where $ \hat{\C} $ is the Riemann sphere $ \C \union \{\infty\} $ and $ \IS^1 = \{ z \in \C : \abs{z} = 1\} $
is the unit circle. We will consider similarly defined groups which are defined in terms of reflections in the sides of a triangle with arbitrary dihedral angles. These are not discrete in general.

\begin{defn}
  For $ \alpha,\beta,\gamma \in \R_{\geq 0} $ satisfying $ \alpha+\beta+\gamma<\pi$, the $ [\alpha,\beta,\gamma]$-\df{triangle group} is the orientation-preserving
  half of the group generated by reflections in the hyperbolic triangle with angles $ \alpha,\beta,\gamma $.
\end{defn}
\begin{rem}
  We use square brackets to avoid a notational clash with the classical (discrete) triangle groups. The orientation-preserving half of the $ (p,q,r)$-triangle group is the $ [\pi/p,\pi/q,\pi/r]$-triangle group.
\end{rem}

Generalising the association between discrete triangle groups and certain hyperbolic orbifolds, we can find geometric structures for which the indiscrete triangle groups
act as holonomy groups. A \df{hyperbolic cone manifold} of dimension $n$ is an orientable topological manifold $M$ that is triangulated so that each simplex is homeomorphic to a geodesic simplex in
$\H^n $; pulling back the metric from $\H^n$ to $ M$ gives the latter a complete metric structure; there is a codimension $\geq 2$ subset $ \Sing(M) \subset M $ such that
the manifold $ M \setminus \Sing(M) $ is an incomplete Riemann manifold of sectional curvature $ -1 $. See Cooper, Hodgson, and Kerckhoff~\cite{chk} for further details.
Given a hyperbolic cone manifold $M$, one can define its holonomy group just as for a complete hyperbolic manifold. We shall only be interested in dimensions $ 2 $ and $3$, in which
case the holonomy group $ \Hol(M) $ is a subgroup of $ \PSL(2,\C) $ with the property that if $ \alpha $ is a singular arc (in dimension $3$) or a singular point (in dimension $2$)
of cone angle $ \theta $, then the meridian loop of $ \alpha $ is represented in $ \Hol(M) $ by a hyperbolic rotation (i.e.\ an elliptic element) of angle $\theta$.

\begin{lem}\label{lem:triangle_is_holonomy}
  The holonomy group of each of the following cone manifolds is an $ [\alpha,\beta,\gamma]$-triangle group:
  \begin{enumerate}
    \item A cone manifold of dimension $2$ supported on $ \IS^2 $ with singular locus consisting of\/ $3$ points, one each of
          cone angle $ 2\alpha $, $ 2\beta $, and $ 2\gamma $, such that $ \alpha + \beta + \gamma < \pi $;
    \item A cone manifold of dimension $3$ which is the product of the manifold in (1) with an open interval.
  \end{enumerate}
\end{lem}

\begin{figure}
  \labellist
  \small\hair 2pt
  \pinlabel {$2\gamma$} [l] at 38 83
  \pinlabel {$2\alpha$} [tr] at 130 154
  \pinlabel {$2\beta$} [br] at 125 33
  \pinlabel {$C_1$} [r] at 200 113
  \pinlabel {$C_2$} [l] at 442 81
  \pinlabel {$A$} [b] at 311 191
  \pinlabel {$B$} [t] at 313 10
  \pinlabel {$\theta$} [l] at 207 115
  \pinlabel {$2\gamma-\theta$} [r] at 435 81
  \pinlabel {$2\alpha$} [t] at 311 185
  \pinlabel {$2\beta$} [b] at 313 16
  \endlabellist
  \centering
  \includegraphics[height=.15\textheight]{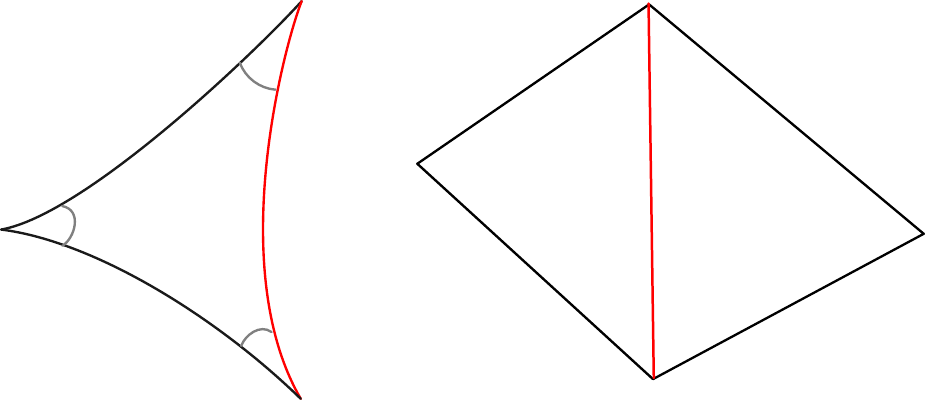}
  \caption{The holonomy group of a $3$-marked sphere has a quadrilateral fundamental domain obtained by cutting along two (black) geodesics joining marked points, leaving one (red) uncut geodesic
  that is an axis of reflective symmetry. In this diagram straight lines represent hyperbolic (not Euclidean) geodesic segments.\label{fig:triangle_quad}}
\end{figure}
\begin{proof}
  Let $M$ be the manifold defined in (1). To see that the holonomy group of $M$ is as claimed, consider the two geodesic triangles with vertices at the three cone points. Cut along two of the edges,
  to produce two hyperbolic triangles joined along one common edge. Without loss of generality, we may assume that this edge joins the cone points of angle $ 2\alpha $ and $ 2\beta $,
  so the quadrilateral has angles $ 2\alpha $, $ 2\beta $, $ \theta $, and $ 2\gamma -\theta $ for some $ \theta \in [0,2\gamma] $, at respective corners $ A, B, C_1, C_2 $; this notation is shown
  in \zcref{fig:triangle_quad}. Because the triangle was obtained by cutting, $ d(A, C_1) = d(A,C_2) $ and $ d(B,C_1) = d(B,C_2) $ where $d$ is the hyperbolic metric. Elementary hyperbolic geometry
  tells us now that the quadrilateral must be symmetric upon reflection across $ [A,B] $ and thus $ \theta = \gamma $. We see that a fundamental domain for $ \Hol(M) $ consists of a quadrilateral
  with angles $ 2\alpha $, $2\beta $, $ \gamma $, $ \gamma $; additionally the side-pairing transformations must be rotations of angles $ 2\alpha $ and $ 2\beta $; thus $ \Hol(M) $ is a triangle group
  of the claimed type.

  That the holonomy group of (2) is as claimed follows from the usual properties of product spaces. One can realise it explicitly by embedding $ \H^2 $ as a round ball into the Riemann sphere $ \hat{\C} $;
  reflecting the fundamental quadrilateral for the manifold in (1) across $ \partial \H^2 $, and taking the polyhedron in $ \H^3 $ cut out by the geodesic domes supported on the sides of the two polygons
  on $ \hat{\C} = \partial \H^3 $. The result follows from the Poincar\'e polyhedron theorem for cone manifolds~\cite[Appendix]{elzenaar25ej}.
\end{proof}

Some further examples of fundamental polygons for $ [0,0,\theta]$-triangle groups are given in \zcref{fig:triangles}; as $ \theta \to \pi $, the group degenerates
to an elementary group generated by a single translation fixing $ \infty $ and the interior of the fundamental domain limits onto $ \infty $. Geometrically, this corresponds
to the manifolds of type (2) in \zcref{lem:triangle_is_holonomy} degenerating to a solid torus as the cone angle around one of the singular arcs degenerates to $2\pi$.

\begin{figure}
  \begin{subfigure}{0.2\textwidth}
    \centering
    \includegraphics[width=\textwidth]{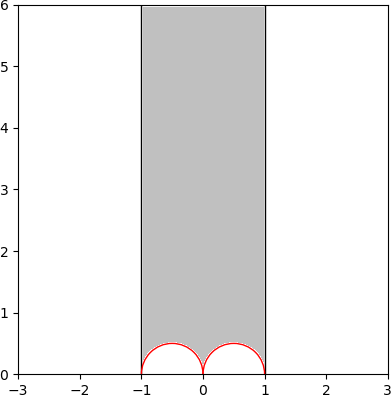}
    \caption*{$\theta=0$}
  \end{subfigure}
  \begin{subfigure}{0.2\textwidth}
    \centering
    \includegraphics[width=\textwidth]{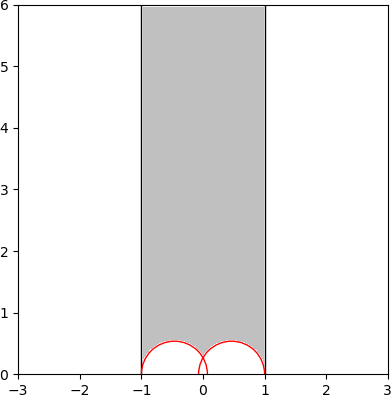}
    \caption*{$\theta=\pi/6$}
  \end{subfigure}
  \begin{subfigure}{0.2\textwidth}
    \centering
    \includegraphics[width=\textwidth]{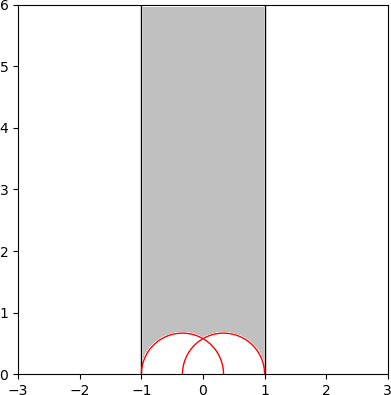}
    \caption*{$\theta=\pi/3$}
  \end{subfigure}
  \begin{subfigure}{0.2\textwidth}
    \centering
    \includegraphics[width=\textwidth]{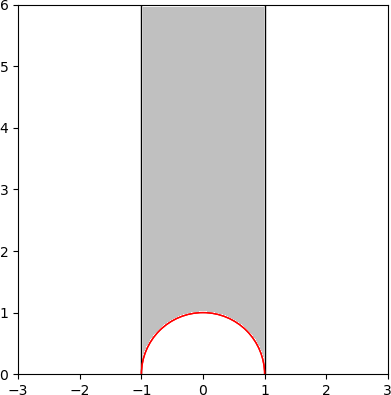}
    \caption*{$\theta=\pi/2$}
  \end{subfigure}\\
  \begin{subfigure}{0.2\textwidth}
    \centering
    \includegraphics[width=\textwidth]{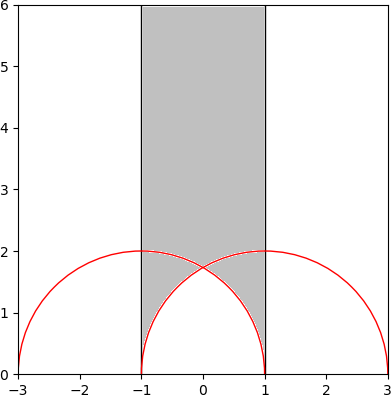}
    \caption*{$\theta=2\pi/3$}
  \end{subfigure}
  \begin{subfigure}{0.2\textwidth}
    \centering
    \includegraphics[width=\textwidth]{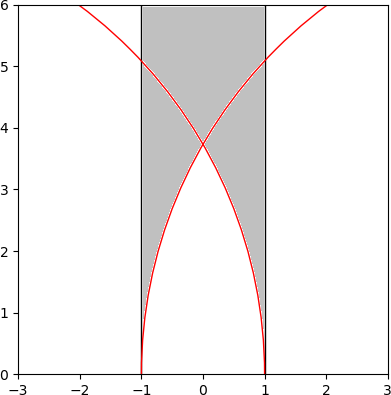}
    \caption*{$\theta=5\pi/6$}
  \end{subfigure}
  \begin{subfigure}{0.2\textwidth}
    \centering
    \includegraphics[width=\textwidth]{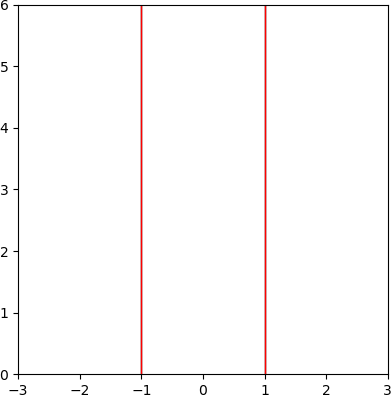}
    \caption*{$\theta=\pi$}
  \end{subfigure}
  \caption{Fundamental quadrilaterals for the $ [0,0,\theta]$-triangle group: quadrilaterals in $ \H^2 $ with with three angles $0$ and fourth angle $ 2\theta $.
  One side-pairing map is an elliptic with rotation angle $ 2\theta $ that fixes the vertex on the imaginary axis and pairs the two sides adjacent to it; the other side-pairing
  map is the horizontal translation that fixes $ \infty $.\label{fig:triangles}}
\end{figure}

\begin{defn}\label{defn:canonical_quad}
  If $ G = \langle x, y \rangle $ is a marked triangle group where $ x $ and $ y $ are primitive rotations, then the \df{canonical quadrilateral} for $ G $
  is the hyperbolic quadrilateral with vertices
  \begin{displaymath}
    \Fix x,\; \Fix y,\; \Fix xy^{-1},\quad\text{and}\quad \Fix yx^{-1}.
  \end{displaymath}
\end{defn}

Observe that the condition $ \alpha + \beta + \gamma < \pi $ is necessary in order to ensure that the fundamental domain for $ \Hol(M) $ embeds into $ \H^2 $ itself. However, $ \H^2 $ is not
the space onto which the cone surfaces of (1) develop, except in the special case that $ \alpha,\beta,\gamma$ is a submultiple of $ \pi$, i.e.\ when $M$ is an orbifold and $ \Hol(M) $ is a classical
triangle group. It is not too hard to describe the developing surface.

\begin{lem}\label{lem:developing_surface}
  As in (1) of \zcref{lem:triangle_is_holonomy}, let $M$ be the cone manifold of dimension $2$ supported on $ \IS^2 $ with singular locus consisting of\/ $3$ points, one each of cone angle $ 2\alpha $, $ 2\beta $,
  and $ 2\gamma $, such that $ \alpha + \beta + \gamma < \pi $. Then the developing surface of $ M $ is a metric space homeomorphic to the union of the unit disc $ \B^2 \subset \C $ with
  countably many of its boundary points.
\end{lem}
\begin{proof}
  The developing surface $\hat{M} $ of $M$ is obtained as the abstract union of the hyperbolic triangles of angles $ \alpha,\beta,\gamma$ along matching edges (i.e.\ the edge joining the $ \alpha $ and $\beta $
  vertices of one triangle is attached to the edge joining the same marked vertices of an adjacent triangle). Let $ \Sing^O(M) $ be the set of singular points of $M$ with cone angle that is a submultiple
  of $ \pi $, and let $ \Sing^C(M) $ be the remaining singular points. Consider the cone orbifold $ S $ that is the quotient of $ \H^2 $ by the classical triangle group with angle-set consisting
  of the same angles as the points in $ \Sing^O(M) $, and the remaining angles being $0$. Roughly speaking, $S$ is the unique complete hyperbolic orbifold obtained by drilling
  the non-orbifold singular points of $M$. There is a unique biconformal map from $ S \setminus \Sing S $ to $ M \setminus \Sing(M) $, and this map lifts to a biconformal map
  from the development $ \hat M \setminus \widehat{\Sing^C(M)} $ to $ \H^2 $ which preserves angles around the orbifold points. Thus $ \hat M $ is the union of a conformal copy of $ \H^2 $
  with the lifts of the non-orbifold singular points $\widehat{\Sing^C(M)}$, and the latter set is naturally identified with the points on $ \partial \H^2 $ corresponding to the limit points of
  the parabolic elements of $ \Hol(S) $ coming from the deleted points of $ \Sing^C(M) $.
\end{proof}

The developing surface $ \hat{M} $ is thus embeddable in the Riemann sphere $ \hat{\C} $ just like $ \H^2 $ is, but the geodesics in $ \hat{M} $ will be distorted due to the cone points on the
boundary. That is, the biconformal map $ \H^2 \to \Int \hat{M} $ induced by the Riemann mapping theorem is not an isometry. In the sequel we will also make use of the isometric development of hyperbolic
cone $3$-manifolds, which is similar: one obtains a singular metric structure on $ \B^3 $.

\section{Compression bodies and function groups}\label{sec:maskit_decomposition}
In the classical theory of Kleinian groups, a \df{function group} is a pair $ (\Gamma,\Delta) $ where $\Gamma$ is a Kleinian group (a discrete subgroup of $\PSL(2,\C)$)
and $ \Delta $ is a connected component of the domain of discontinuity $\Omega(\Gamma) $ that is left invariant by $\Gamma$. In particular, $ \Delta/\Gamma $ is a Riemann
surface with fundamental group surjecting onto $\Gamma$. These groups first arose in early attempts to classify Kleinian groups~\cite{maskit74}, and the topological classification of function
groups was achieved by Maskit~\cite{maskit75,maskit77} (see also the exposition in his monograph~\cite[Chapter~X]{maskit}) who proved that every function group is the holonomy group
of a hyperbolic orbifold structure on a compression body with a compact compression end; in other words, every hyperbolic compression body manifold lies on the boundary of a quasiconformal deformation
space of torsion free function groups. Maskit's work followed earlier work by Koebe~\cite{koebe12}, who gave a procedure to construct a function group of every topological type.


The classification given by Maskit, which we will refer to as the \df{Maskit decomposition}, assigns a combinatorial object (called a \df{signature}) to every function group. This gives a recipe for
reproducing the function group up to quasiconformal conjugacy in $ \PSL(2,\C) $ (equivalently, up to the quasi-isometry type of the hyperbolic quotient manifold). The result is a sequence of amalgamated products
and HNN extensions starting from elementary groups, quasi-Fuchsian groups (here, this means specifically quasiconformal deformations of cocompact Fuchsian groups), and degenerate groups
(Kleinian groups with domain of discontinuity that is non-empty, connected, and simply-connected). Function groups are geometrically finite if and only if no degenerate groups appear in
their Maskit decomposition.

In \zcref{sec:combination_theorems} we introduce the abstract tools which we will use (combination theorems). In \zcref{sec:proofmain} we prove \zcref{thm:main}, by
manufacturing and deforming Maskit decompositions for certain hyperbolic cone manifold structures on the manifolds in the theorem statement. Finally in \zcref{sec:koebemaskit}
we include a brief discussion of the Maskit decomposition for more general hyperbolic cone manifold structures on compression bodies; this is not part of the logical development
of the paper but may be of interest to workers in the area.

\subsection{Combination theorems}\label{sec:combination_theorems}
In order to prove \zcref{thm:main}, we will need cone manifold analogues of the ingredients for the Maskit decomposition. We need only consider the
case of compression bodies with holonomy groups built from triangle groups and elementary rank $2$ parabolic groups. Thus every component group $G$ preserves two discs,
and all amalgamations and HNN extensions will occur by identifying or conjugating together cyclic subgroups generated by a single parabolic element.

We assume that the reader is somewhat familiar with the general notion of a combination theorem in geometric group theory. We will need combination theorems
that work with negatively curved metric spaces slightly more general than hyperbolic space. The natural generality to work in here is that of metric
spaces that satisfy the `$\CAT(0)$ property', which roughly states that geodesic triangles are thinner than Euclidean triangles with the same edge-lenths.
A detailed definition may be found in Bridson and Haefliger~\cite[Chapter~II.1]{bridson_haefliger}. It is enough to know that the developing
surfaces of hyperbolic cone manifolds are complete $ \CAT(0) $ spaces.

The first result that we recall is on amalgamated products. We add an additional assumption compared to Bridson and Haefliger~\cite[Theorem II.11.18]{bridson_haefliger}
which holds in our setting and which makes the proof slightly easier.
\begin{thm}
  Suppose that $ G_1 $, $ G_2 $, and $H$ are groups acting properly by isometries on respective complete $ \CAT(0) $ spaces $ X_1 $, $ X_2 $, and $ Y $. If there exist
  for $ j \in \{1,2\} $ a pair of monomorphisms $ \iota_j : H \to G_j $ and a pair of $\iota_j$-equivariant isometric embeddings $ f_j : Y \to X_j $
  such that $ g \iota_j(H) \inter \iota_j(H) = \emptyset $ whenever $ g \in G_j \setminus \iota_j (H) $, then the amalgamated product $ G_1 *_H G_2 $ acts by isometries on
  a complete $ \CAT(0) $ space.
\end{thm}
\begin{proof}
  For the sake of notation we will identify $ \iota_1 H $ and $ \iota_2 H $ with $ H $; we also view $ G_1 $, $ G_2 $, and $H$ as subgroups of $ G_1 *_H G_2 $ in the canonical way.
  We recall that there is a \df{normal form} for words in the formal amalgamated product $ G_1 *_H G_2 $~\cite[\S VII.A]{maskit}: namely, every such word can be written essentially uniquely (up to choices
  of coset representatives for $ G_1/H $ and $ G_2/H $) as a product
  \begin{equation}\label{eq:normalform1}
    h g_1 \cdots g_r
  \end{equation}
  where $ h \in H $, and where the elements $ g_i $ are alternately coset representatives for $ G_1/H $ or $ G_2/H $ which
  do not lie in $H$.


  Construct a metric space $ X_1 \intercal_Y X_2 $ via the following iterative process. Take a copy of $ X_1 $, and consider all the translates of $ Y \subset X_1 $
  by elements of $ G_1 $; the translates are indexed by cosets $ G_1/H $, and are all disjoint. Attach to every such translate $ g Y $ a copy of $ X_2 $, so that $ gY \subset X_1 $
  is identified with $ Y \subset X_2 $. For every translate $ g Y \subset X_2 $ where $ g $ ranges over nontrivial coset representatives for $ G_2/H $, attach a copy of $ X_1 $
  so that $ gY \subset X_2 $ is identified with $ Y \subset X_1 $. Continue this process inductively.
  The result is $\CAT(0)$ because it is locally $ \CAT(0) $ and coarsely a tree. We call the copies of $ X_1 $ and $ X_2 $ the \df{flats} of the tree.

  There is a natural action of the group $ G_1 *_H G_2 $ on this tree, coming from the normal form \eqref{eq:normalform1}. The point is that the flats are labelled by cosets
  in $ (G_1 *_H G_2)/H $, i.e.\ by sequences of elements in $ G_i $ and $ G_j $, and this comes with an action of $ G_1 *_H G_2 $ by left-multiplication.
  Concretely, an element of $ G_1 $ acts on each copy of $ X_1 $ in the usual way and permutes the copies of $ X_2 $ while fixing each of them pointwise away from the embedded copy of $ Y $,
  and an element of $ G_2 $ acts on each copy of $ X_1 $ in the usual way while permuting the copies of $ X_2 $. This action is clearly isometric.
\end{proof}

From the proof, we obtain a fundamental domain of the amalgamated group on the tree of flats.
\begin{cor}\label{cor:amalgamation_fd}
  Let $ D_1 $ and $ D_2 $ be fundamental domains for the actions of $ G_1 $ and $ G_2 $ in $ X_1 $ and $ X_2 $, and let $ P $ be a fundamental domain for the action of $ H $ on $Y$
  such that $ f_1(P) \subset D_1 $ and $ f_2(P) \subset D_2 $. Suppose that two flats in $ X_1 \intercal_Y X_2 $ are chosen, one isometric copy of $ X_1 $ and one isometric copy of $ X_2 $,
  and that $ D_1 $ and $ D_2 $ are embedded into $X_1 \intercal_Y X_2$ by the coresponding identity maps. If $ f_1(P) = f_2(P) $ under this embedding, then
  the union $ D_1 \union D_2 $ is a fundamental domain for the group action of $ G_1 *_H G_2 $ on $ X_1 \intercal_Y X_2 $, and $ (X_1 \intercal_Y X_2) / (G_1 *_H G_2) $
  is obtained by gluing $ X_1/G_1 $ and $ X_2/G_2 $ together along $ Y/H $. \qed
\end{cor}

When $G_1$ and $ G_2$ are Kleinian and the maps $ \iota_j $ are conjugacy maps in $ \PSL(2,\C) $ (that is, they are induced by type-preserving
conformal automorphisms of the Riemann sphere), this gives a version of Maskit's first combination theorem~\cite[\S VII.C]{maskit}. We will need
the same statement but for cone manifold holonomy groups.

\begin{figure}
  \labellist
  \small\hair 2pt
  \pinlabel {$\langle h \rangle \leq G_1$} [b] at 48 175
  \pinlabel {$\langle h \rangle \leq G_2$} [b] at 155 172
  \endlabellist
  \centering
  \includegraphics[width=.4\textwidth]{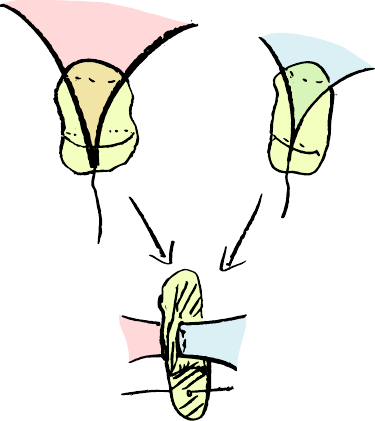}
  \caption{Local picture of an amalgamation along $\langle h \rangle$: the interiors of the yellow balls are sliced out and the two manifolds are glued along the two marked discs.\label{fig:amalgamated_product_abstract}}
\end{figure}

\begin{cor}\label{cor:cone_amalgamation}
  Let $ G_1 = \Hol(M) $ be a hyperbolic cone $3$-manifold holonomy group that admits a finite-sided fundamental domain in $ \H^3 $, with a subgroup $ H = \langle h \rangle $ such
  that $h$ is either elliptic (possibly of infinite order) or parabolic, and such that $ h $ represents a simple closed curve $ \gamma_h$ homotopic to either a puncture or a
  cone-point on the conformal boundary of $M$. Let $ G_2 = \Hol(S) $ be a hyperbolic cone-$2$-manifold group which contains a copy of $H$, so that $S$ has either a puncture
  or a marked point of the same angle as that on $ \partial M $.

  Let $M^{*}$ be the cone $3$-manifold with a once-marked disc on its boundary, obtained by taking a small neighbourhood of $ \partial M $ around the marked point represented by $h$
  (i.e.\ take the marked disc bounded by the projection of the axis of $h$, push it slightly into the interior of $M$, and delete the region it sweeps out). Similarly let $ S^{*} $ be
  the cone $3$-manifold obtained by deleting a small neighbourhood of the marked point on one end of the cone $3$-manifold $ S \times (0,1) $, so $S^{*}$ also has a once-marked disc boundary component.

  Then $ G_1 *_H G_2 $ is the holonomy group of a cone $3$-manifold obtained by gluing $ M^{*} $ and $ S^{*} $ along the two distinguished once-marked discs (\zcref{fig:amalgamated_product_abstract}).
\end{cor}
\begin{proof}
  Since hyperbolic cone manifolds come equipped with a locally finite PL-structure, we can use the usual theory of Kleinian groups~\cite[Proposition~VI.A.10]{maskit} to find that there is a pair of
  small round discs in the plane that are preserved by $H \leq G_1 $ which intersect a fundamental domain of $ G_1 $ to form neighbourhoods of the marked point on the boundary of $ G_1 $. By
  modifying $ \gamma_h $ via isotopy on the boundary surface, we can assume that its lift $\hat{\gamma}_h$ to a fundamental domain for $ G_1$ lies entirely in one of these two discs, bounded by
  a circle $ C_1 $, and that there is a nested disc (say, bounded by the circle $C_2 $) in the same pencil of discs preserved by $ h $ that descends to a neighbourhood of the marked point which
  lies inside the region bounded by $ \gamma_h $. See \zcref{fig:amalgamated_product_domain}.

  Consider now the group $ G_2 $. Since it is a surface group, it preserves a round circle in the plane and acts on the two complementary discs as a subgroup of $ \Isom^+(\H^2) $. Embed one of these discs
  into the small disc $ D_2 $, so that the fixed points of $ h $ in $ G_1 $ and $G_2$ match up. We now have a new polyhedron defined by taking the portion $D_1$ of a polyhedron of $ G_1 $ outside the dome
  in $ \H^3$ above $ C_1 $, and the portion $ D_2 $ of a polyhedron of $ G_2 $ inside the dome in $ \H^3 $ above $ C_2 $. Let $ W_1 $ be one of the walls paired by $ h $ in $D_1$ and let $ W_2 $ be one
  of the walls paired by $ h $ in $D_2$; then draw any embedded surface $W_{12}$ inside the region in $\H^3$ bounded by the two domes that interpolates between the intersection of $W_1$ with the dome above $C_1 $
  and $W_2$ and the dome above $C_2$. The region $ B $ lying between $ W_{12} $ and $ hW_{12} $ in the dome forms a `bridge' between the two halves of the polyhedron already constructed.

  We may now apply \zcref{cor:amalgamation_fd}: $X_1$ and $ X_2$ are respectively the developing spaces of the two hyperbolic cone $3$-manifolds $ M $ and $ S \times [0,1] $, and $ Y $ is the development of the region
  between the projections to $ M $ and $ S \times [0,1] $ of the regions bounded by the two domes---these projections are the small neighbourhoods of the marked points described in the statement of the
  current claim and sketched in \zcref{fig:amalgamated_product_abstract}. The region $B$ takes the role of $P$ in the statement of \zcref{cor:amalgamation_fd}. In particular $ G_1 *_H G_2 $
  has a fundamental domain which is isometric to $ D_1 \union P \union D_2 $. This fundamental domain embeds into $ \H^3 $ by construction, so the quotient $ (X_1 \intercal_Y X_2) / (G_1 *_H G_2) $
  is actually a hyperbolic cone $3$-manifold.
\end{proof}

\begin{figure}
  \labellist
  \small\hair 2pt
  \pinlabel {$G_1$}  at 78 142
  \pinlabel {$G_2$} [tl] at 112 40
  \pinlabel {$\Fix h $} [tr] at 70 33
  \pinlabel {$\hat{\gamma}_h $} [t] at 86 115
  \pinlabel {$C_1 $} [br] at 46 98
  \pinlabel {$C_2 $} [br] at 66 78
  \endlabellist
  \centering
  \includegraphics[width=.4\textwidth]{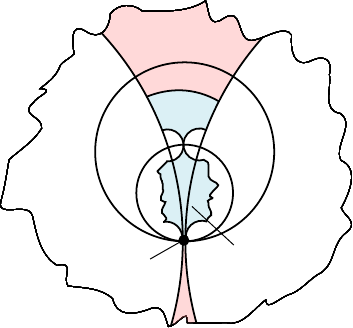}
  \caption{Surgery on domains to prove \zcref{cor:cone_amalgamation}. The labels $G_1 $ and $ G_2 $ indicate the (shaded) fundamental domains of those groups.\label{fig:amalgamated_product_domain}}
\end{figure}

Similar results hold for HNN extensions of cone-manifold groups: again there must be a pair of topological curves on the quotient surfaces that are identified by a M\"obius transformation,
but this time the M\"obius transformation is adjoined to the group. Here, the new fundamental domain is obtained by slicing off parts of the old domain and there is no gluing occuring, so
unlike the case of amalgamated products there is no need to check that the new domain is non-overlapping---it comes for free. The relevant combination theorem for Kleinian groups
is the second Maskit combination theorem~\cite[\S VII.E]{maskit}, and for general $\CAT(0)$ spaces the appropriate theorem is given by Bridson and Haefliger~\cite[Proposition II.18.21]{bridson_haefliger}.
We will simply state the consequence which we need.

\begin{cor}\label{cor:cone_hnn}
  Let $ G = \Hol(M) $ be a hyperbolic cone $3$-manifold holonomy group that admits a finite-sided fundamental domain in $ \H^3 $ and which contains two subgroups $ H_1 = \langle h_1 \rangle $
  and $ H_2 = \langle h_2 \rangle $ so that either $ h_1 $ and $ h_2 $ are both parabolic, or $ h_1 $ and $ h_2 $ are both elliptic of the same rotation angle.

  Let $M^{*}$ be the cone $3$-manifold with two once-marked discs on its boundary, obtained by taking horoball neighbourhoods of $ \partial M $ around the marked points represented by $h_1$ and $ h_2 $
  as in \zcref{fig:amalgamated_product_abstract} (in the picture, set $ G_1 = G_2 $). Then there exists $ f \in \PSL(2,\C) $ such that $ G *_{\langle f \rangle} $
  is the holonomy group of a cone $3$-manifold obtained by gluing the two marked discs on the boundary of $ M^{*} $ together. \qed
\end{cor}

\subsection{Proof of \zcref{thm:main}}\label{sec:proofmain}
We proceed with the proof of the main theorem, and so we assume the notation in the statement: thus $M$ and $M'$ are compression bodies differing by a $2$-handle gluing.
The proof essentially goes by decomposing $M' $ into $[0,0,0]$-triangle groups and continuously replacing some of them with $ [0,0,\theta/2]$-triangle
groups for $ \theta \in (0,2\pi) $ without changing any of the global combinatorics.

\begin{lem}\label{lem:maskit_decomposition}
  We can write $ M' = \H^3/\Gamma $, where $ \Gamma $ is a geometrically finite group of the form
  \begin{displaymath}
    \Gamma = (G_1 *_{\langle x_1 \rangle} G_2*_{\langle x_2 \rangle} \cdots *_{\langle x_{l-1} \rangle} G_l * H_1 * \cdots * H_n * S) *_{\langle y_1 \rangle} \cdots *_{\langle y_m \rangle}
  \end{displaymath}
  where every $ x_i $ is parabolic or the identity, every $ G_i $ is a $[0,0,0]$-triangle group, every $ H_i $ is a rank $2$ parabolic group, $S$ is a Schottky group,
  and each $ y_i \in \PSL(2,\C)$ is loxodromic. Further, all the parabolic subgroups
  appearing in the decomposition appear exactly once up to conjugacy, and the compression end is a compact surface.
\end{lem}
\begin{proof}
  The compression body $ M' $ admits a geometrically finite hyperbolic structure with the property that (i) the compression end is a compact surface;
  (ii) every other end is either a rank $2$ cusp, or a thrice-punctured sphere. This follows (for instance) from the fact that every system of disjoint curves
  on a geometrically finite end can be pinched to parabolics~\cite[Theorem~3.5]{keen93}.

  Such a structure corresponds to a function group with signature $ (K,t) $ which we now describe. For a general signature, $K$ is a complex of Riemann surfaces
  joined by $1$-cells called \df{connectors} labelled with numbers in $ (1,\infty] $ (see \zcref{defn:signature} below). In the special case of interest here for $M'$, it consists of:
  \begin{itemize}
    \item a sphere $S_\epsilon$ for every thrice-punctured sphere $ \epsilon $ of $ M' $;
    \item for every rank $1$ cusp joining two thrice-punctured spheres $ \epsilon $ and $ \epsilon' $ (possibly not distinct), an $ \infty$-connector joining $ S_\epsilon $ and $ S_{\epsilon'} $;
    \item a sphere $ T_\nu $ for every rank $2$ cusp $ \nu $, and an $ \infty$-connector joining $ T_\nu $ to itself.
  \end{itemize}
  The number $t$ is the difference between the genus of the compression end of $M' $ and the sum of the genera of the non-compression ends. We will now show how
  to explicitly construct a group realising this complex, following Maskit's construction of Koebe groups realising any admissible signature~\cite[\S X.F]{maskit}; the
  procedure is very flexible and there are many arbitrary choices to be made, but all possible choices will produce groups that are quasi-conformally conjugate to each other.

  For each connected component of $K$ we construct a group and an associated fundamental domain in $ \hat{\C} $ cut out by circles.

  \stephead{Step 1}
  For every connected component consisting of $n>1$ spheres, it is possible to linearly order the spheres $ S_1,\ldots, S_n $ so that $ S_i $ and $ S_{i+1} $ are
  always connected by an $\infty$-connector. Inductively produce a chain $n$ of $ [0,0,0]$-triangle groups as follows: let $ G_1 = \langle x_0, x_1 \rangle $ be an
  arbitrary triangle group, where the marked generating set consists of two generators of maximal parabolic subgroups, which preserves a round circle $ \Sigma_1 $.
  This triangle group preserves two round discs, one bounded and one unbounded.

  For $ j > 1 $, let $ B_{j-1}^+ $ denote the unbounded disc preserved by $ G_{j-1} $. Recall from \zcref{defn:canonical_quad} that a fundamental domain for the action of $ G_{j-1} $ on $ B_{j-1}^+$ is
  given by the canonical quadrilateral with vertices
  \begin{equation}\label{eq:chain_dommains}
    \Fix x_{j-2},\; \Fix x_{j-1},\; \Fix x_{j-2} x_{j-1},\quad\text{and}\quad \Fix x_{j-1} x_{j-2}.
  \end{equation}
  Let $ \Sigma_j $ be a horosphere in the hyperbolic metric on $ B_{j-1}^+ $ based at $ \Fix x_{j-1} $ which meets only the
  two edges of this quadrilateral that are incident with $ \Fix x_{j-1} $. Let $ G_j $ be an $[0,0,0]$-triangle
  group which preserves the two discs bounded by $ \Sigma_j $, generated by $ x_{j-1} $ and some additional parabolic $ x_j $ such
  that $ \Fix x_j x_{j-1} $ and $ \Fix x_{j-1} x_j $ lie outside the canonical quadrilateral for $ G_{j-1} $ while $\Fix x_{j} $ lies inside it (see \zcref{fig:chain_of_discs}).

  \begin{figure}
    \begin{subfigure}{\textwidth}
      \labellist
      \small\hair 2pt
      \pinlabel {$\Sigma_1$} [bl] at 189 265
      \pinlabel {$\Sigma_2$} [bl] at 293 239
      \pinlabel {$\Sigma_3$} [bl] at 397 180
      \pinlabel {$\Fix x_0$} [r] at 27 196
      \pinlabel {$\Fix x_1$} [r] at 146 58
      \pinlabel {$\Fix x_2$} [r] at 210 39
      \pinlabel {$\Fix x_3$} [l] at 424 140
      \pinlabel {$\Fix x_0 x_1$} [l] at 190 321
      \pinlabel {$\Fix x_1 x_0$} [r] at 98 96
      \pinlabel {$\Fix x_1 x_2$} [l] at 313 300
      \pinlabel {$\Fix x_2 x_1$} [l] at 330 31
      \pinlabel {$\Fix x_2 x_3$} [l] at 398 272
      \pinlabel {$\Fix x_3 x_2$} [l] at 366 56
      \endlabellist
      \centering
      \includegraphics[width=.5\textwidth]{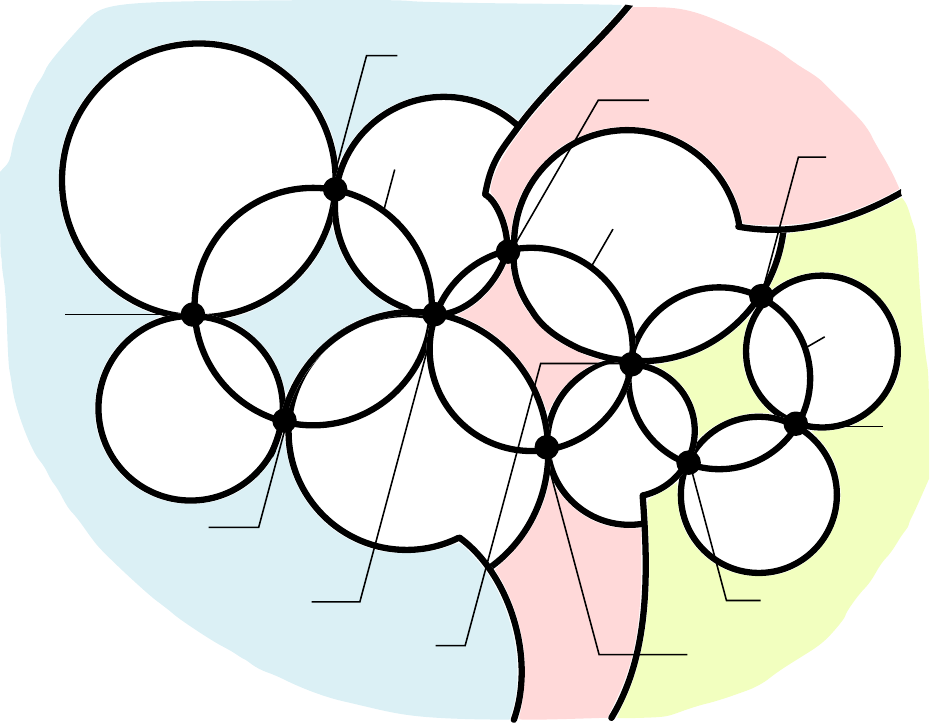}
      \caption{The chain of tangent discs forming a fundamental domain.}
    \end{subfigure}\\
    \begin{subfigure}{\textwidth}
      \labellist
      \small\hair 2pt
      \pinlabel {$x_0$} [r] at 2 80
      \pinlabel {$x_1$} [t] at 108 195
      \pinlabel {$x_2$} [t] at 291 190
      \pinlabel {$x_3$} [l] at 457 104
      \pinlabel {$x_0 x_1$} [bl] at 86 112
      \pinlabel {$x_1 x_2$} [r] at 206 120
      \pinlabel {$x_2 x_3$} [br] at 335 116
      \endlabellist
      \centering
      \includegraphics[width=.5\textwidth]{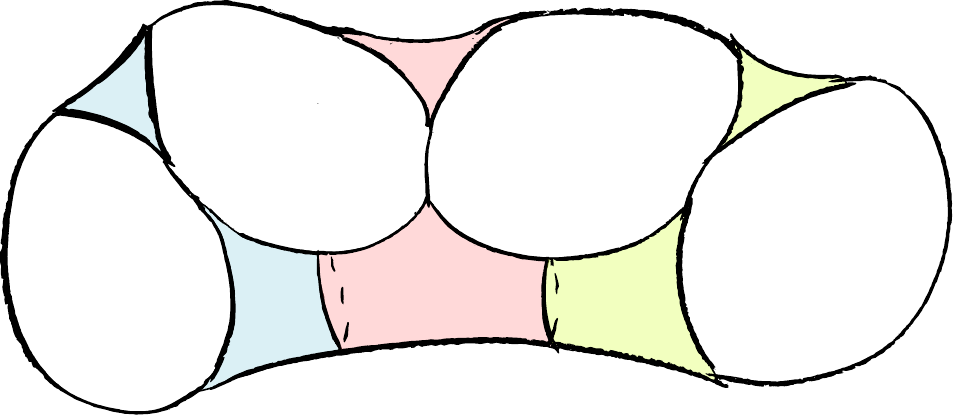}
      \caption{Arrangement of quotient surfaces.}
    \end{subfigure}
    \caption{A sequence of amalgamated products of $[0,0,0]$-triangle groups along boundary parabolics.\label{fig:chain_of_discs}}
  \end{figure}

  Having done this for every connected component of $K$ of this type, relabel all the $G_i $ and $ x_j $ so that the groups for the first connected component are $ G_1, \ldots, G_{l_1} $
  and the shared parabolics are $ x_1, \ldots, x_{l_1 - 1} $, those for the second are $ G_{l_1 + 1}, \ldots, G_{l_1 + l_2} $ and $ x_{l_1 + 1} , \ldots, x_{l_1 + l_2 - 1} $,
  and so on. Note that we have skipped $ x_{l_1} $, $ x_{l_1 + l_2} $, and so on; set these to the identity.

  \stephead{Step 2}
  For every connected component $ \tau_i $ ($ i \in \{1,\ldots,n\} $) of $K$ consisting of a single sphere (which is necessarily connected to itself by an $\infty$-connector), construct two
  orthogonal pairs of tangent circles meeting at a point $ \xi $, and let $ H_i $ be the rank $2$ parabolic group with fixed point $ \xi $ and fundamental domain consisting of the common
  exterior of the circles.

  \stephead{Step 3}
  Let $ S $ be the Schottky group of rank $ t $ obtained by drawing $2t$ mutually disjoint round circles with common exterior and pairing them arbitrarily with $t$ hyperbolic transformations.

  \stephead{Step 4}
  Conjugating all the groups constructed in Steps 1--4 by M\"obius transformations if necessary, we may assume that all the circle patterns constructed are disjoint, and that each connected set
  of circles lies in the common exterior of all the other connected sets of circles, and in fact every set of circles lies in the interior of an unbounded piece of fundamental domain of every
  group $ G_1,\ldots,G_l $, $ H_1,\ldots,H_n$, and $S$. By \zcref{cor:cone_amalgamation}, the resulting pattern forms a fundamental domain for the group
  \begin{displaymath}
    \hat\Gamma = G_1 *_{\langle x_1 \rangle} G_2 *_{\langle x_2 \rangle} \cdots *_{\langle x_{l-1} \rangle} G_l * H_1 * \cdots * H_n * S.
  \end{displaymath}

  There may still exist $\infty$-connectors in the complex $K$ which are not realised by one of the $ x_i $. These correspond to rank $1$ parabolics in $ M' $, and so we must realise
  them by HNN extensions that glue together cusp neighbourhoods of the corresponding parabolic fixed points. Let $\mc{H}$ denote the set of remaining connectors.

  \stephead{Step 5}
  For each pair of spheres joined by an element of $\mc{H}$, arbitrarily
  pick one of the four vertices of the canonical fundamental domains of \eqref{eq:chain_dommains} for each of the corresponding $ G_i $'s which is not glued to an adjacent $ G_{i \pm 1} $
  by the construction of tangent chains in (1) above; do this simultaneously for all connectors. We obtain for each connector a pair of parabolic fixed points, all of which are
  mutually non-conjugate in $ \hat{\Gamma} $, and which are all non-conjugate to any of the $ x_i $ that lie in intersections $ G_i \inter G_{i+1} $ (the $x_i$ at the beginning and end of each
  chain will all appear in this list). For each of these parabolic fixed points, say in $ G_i $, choose a small tangent circle to $\Sigma_i $ at the fixed point which lies in the common exterior of all
  the circles constructed in (1)--(3); these circles should be sufficiently small to be a cusped region for each parabolic (\zcref{fig:hnn_extension}). These exist since in our construction above we ensured that no
  constructed discs came arbitrarily close to the fixed points of parabolics that were not being amalgamated; abstractly, it follows since $\hat\Gamma$ is geometrically finite~\cite[\S VII.A.6--10]{maskit}).
  Now for each $\infty$-connector $ \eta \in \mc{H}$ choose a loxodromic element $ y_\eta \in \PSL(2,\C) $ which maps the bounded disc of one of the two corresponding circles onto the unbounded disc of
  the other and sends the point of tangency (i.e.\ parabolic fixed point) to the point of tangency.

  \stephead{Step 6}
  Some parabolic fixed points will still not be decorated with horocycles on their unbounded side, but any parabolics of this form will be $\hat\Gamma$-conjugate to parabolics whose
  fixed points are basepoints of horocycles; assign corresponding cusped regions to these fixed points by translating the cusped regions for the conjugates via the conjugating element.

  \begin{figure}
  \begin{subfigure}{\textwidth}
    \labellist
    \small\hair 2pt
    \pinlabel {$\hat{y}$} [b] at 134 134
    \pinlabel {$y$} [b] at 138 23
    \endlabellist
    \centering
    \includegraphics[width=.6\textwidth]{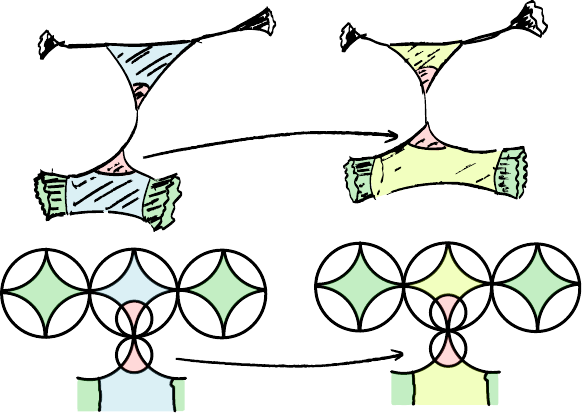}
    \caption{Horoballs identified by the stable element $y$ of some HNN extension, and the gluing homeomorphism $ \hat{y} $ on the quotient manifold.}
  \end{subfigure}\\[1.5em]
  \begin{subfigure}{\textwidth}
    \centering
    \includegraphics[width=.6\textwidth]{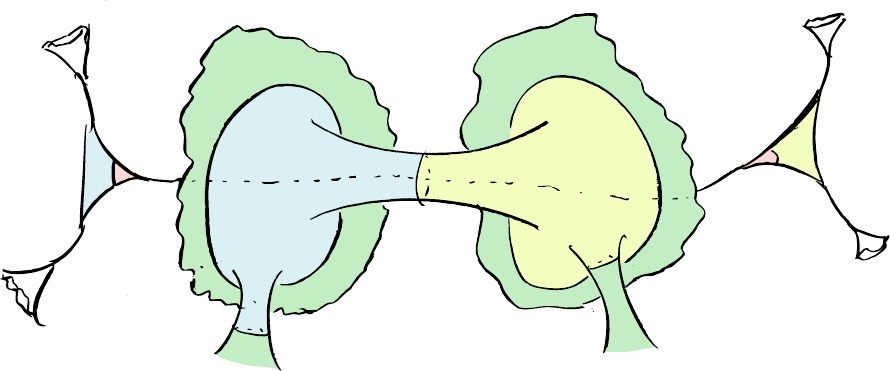}
    \caption{The manifold obtained after performing gluing by $ \hat{y}$. Unlike the other pictures, we view this manifold from the exterior so as to get a good view of the new $1$-handle (with core
             loop represented by $ y $) which is produced.}
  \end{subfigure}
  \caption{The geometric meaning of HNN extensions.\label{fig:hnn_extension}}
  \end{figure}

  \stephead{Step 7}
  Relabel the elements $ \{y_\eta\}_{\eta \in \mc{H}} $ as $ y_1,\ldots, y_m $; by \zcref{cor:cone_hnn}, the group
  \begin{displaymath}
    \hat\Gamma *_{\langle y_1 \rangle} \cdots *_{\langle y_m \rangle}
  \end{displaymath}
  has $\H^3$-quotient with the desired quasiconformal type. A fundamental domain for the action of $ \Gamma $ on its domain of discontinuity is obtained from the fundamental
  domain of $ \hat\Gamma $ by slicing out all the cusped regions constructed in steps 5 and 6.
\end{proof}

  \begin{figure}
  \begin{subfigure}{\textwidth}
    \centering
    \includegraphics[width=.6\textwidth]{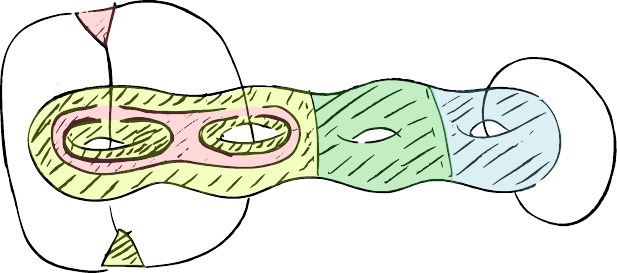}
    \caption{The compression body manifold, with three rank $1$ cusps and one rank $2$ cusp.\label{fig:example_complex_1}}
  \end{subfigure}\\[1.5em]
  \begin{subfigure}{\textwidth}
    \labellist
    \small\hair 2pt
    \pinlabel {$G_1$} at 173 37
    \pinlabel {$G_2$} at 47 151
    \pinlabel {$S$} at 205 61
    \pinlabel {$H$} [b] at 249 239
    \pinlabel {$y_2$} [b] at 118 229
    \pinlabel {$x_1^{-1}y_2x_1$} [t] at 104 51
    \pinlabel {$y_1$} [t] at 50 51
    \endlabellist
    \centering
    \includegraphics[width=.6\textwidth]{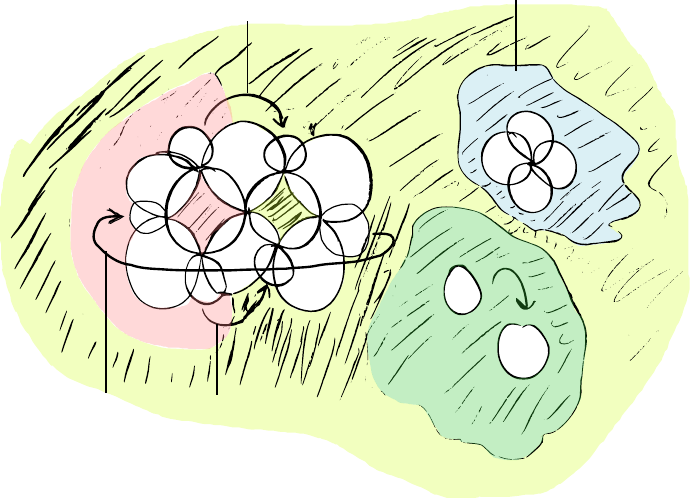}
    \caption{The fundamental domain produced by following \zcref{lem:maskit_decomposition}.\label{fig:example_complex_2}}
  \end{subfigure}
  \caption{Maskit decomposition of a genus $4$ compression body.\label{fig:example_complex}}
\end{figure}
\begin{ex}
  The hyperbolic compression body depicted in \zcref{fig:example_complex_1} has corresponding function group $\Gamma$ with signature
  \begin{displaymath}
    \Sign(\Gamma) = \Bigg( \begin{tikzpicture}[baseline=-.12cm]
                                      \draw[fill=red!30] (-.75,0) ellipse (.33cm and .75cm);
                                      \draw[fill=yellow!30] (.75,0) ellipse (.33cm and .75cm);
                                      \draw (-.75,.4) --node[above]{$\infty$} (.75,.4);
                                      \draw (-.75,0) --node[above]{$\infty$} (.75,0);
                                      \draw (-.75,-.4) --node[above]{$\infty$} (.75,-.4);
                                      \draw[fill=blue!30] (1.75,0) ellipse (.33cm and .5cm);
                                      \draw (1.75, .3) to[bend left=90, looseness=3] node[right]{$\infty$} (1.75,-.3);
                                    \end{tikzpicture}\,,\; 1 \Bigg).
  \end{displaymath}
  The procedure in \zcref{lem:maskit_decomposition} gives us a decomposition
  \begin{displaymath}
    (G_1 *_{\langle x_1 \rangle} G_2 * H * S) *_{\langle y_1 \rangle} *_{\langle y_2 \rangle}
  \end{displaymath}
  where $ G_1 = \langle x_0, x_1 \rangle $ and $ G_2 = \langle x_1, x_2 \rangle $ are $[0,0,0]$-triangle groups marked with parabolic generators, $H$ is
  an elementary rank $2$ parabolic group, $S$ is a genus $2$ Schottky group, and $ y_1, y_2 \in \PSL(2,\C) $ are loxodromic elements such that $ y_1 $
  conjugates $ x_0 $ to $ x_2 $, and $ y_2 $ conjugates $ x_0 x_1 $ to $ x_1 x_2 $. The exact arrangement of the actions of these groups is shown in \zcref{fig:example_complex_2},
  where the entire shaded region is a fundamental domain for $ \Gamma $.
\end{ex}

We choose a marking for each $ G_i $, writing $ G_i = \langle a_i, b_i \rangle $ where the elements $ a_i $ and $ b_i $ generate maximal parabolic subgroups,
and $ H_i = \langle c_i, d_i \rangle $ where both $ c_i $ and $ d_i $ are parabolic.

\begin{figure}\centering
  \begin{subfigure}{.5\textwidth}
    \centering
    \includegraphics[width=.9\textwidth]{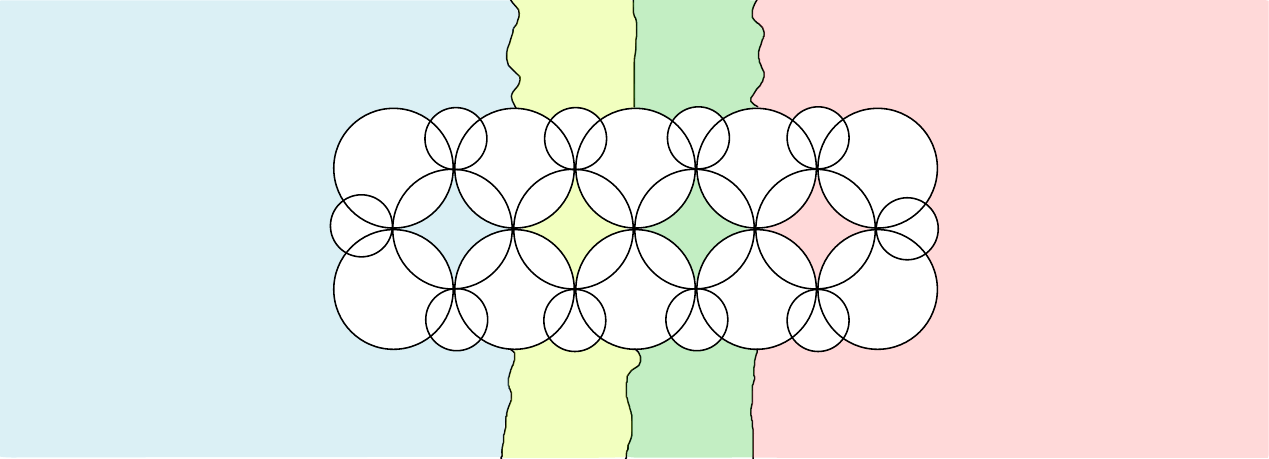}
    \caption{$\theta=0$}
  \end{subfigure}%
  \begin{subfigure}{.5\textwidth}
    \centering
    \includegraphics[width=.9\textwidth]{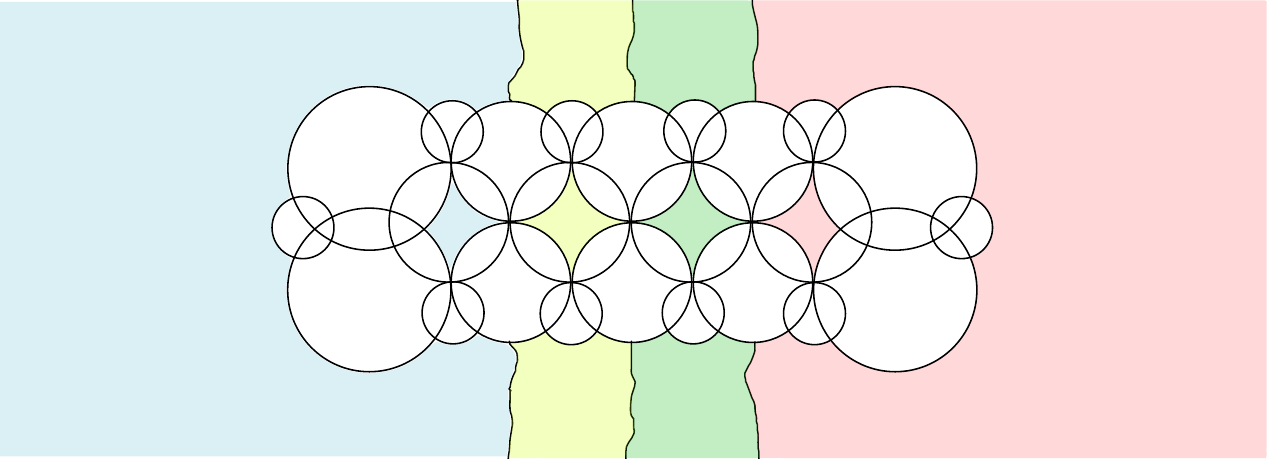}
    \caption{$0<\theta<\pi$}
  \end{subfigure}\\
  \begin{subfigure}{.5\textwidth}
    \centering
    \includegraphics[width=.9\textwidth]{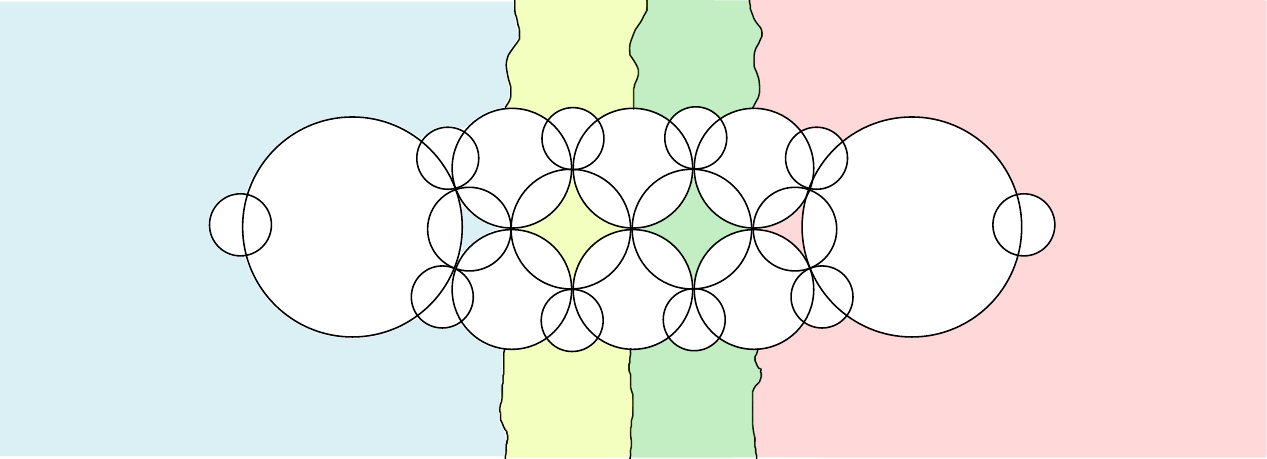}
    \caption{$\theta=\pi$}
  \end{subfigure}%
  \begin{subfigure}{.5\textwidth}
    \centering
    \includegraphics[width=.9\textwidth]{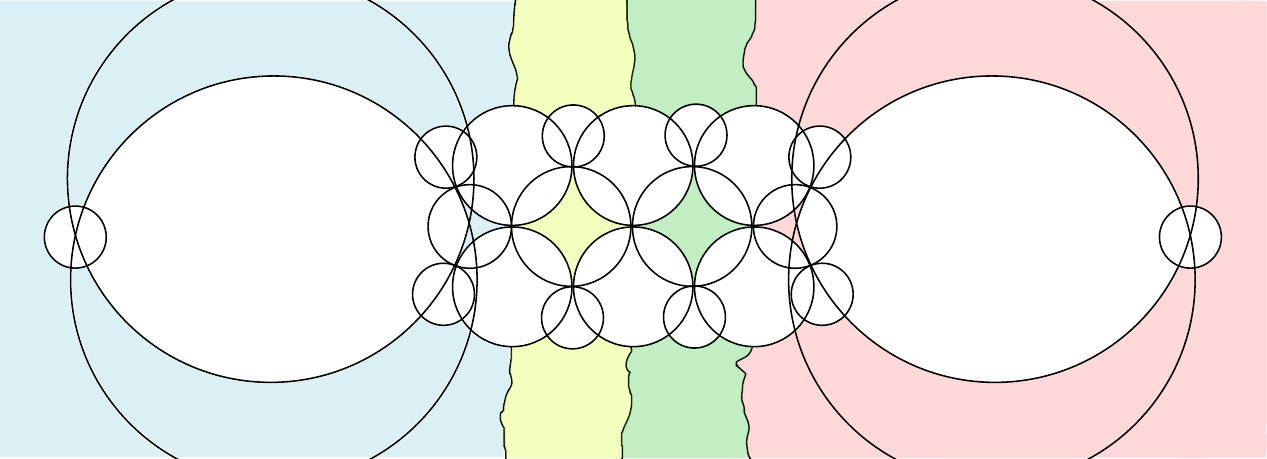}
    \caption{$\pi<\theta<2\pi$}
  \end{subfigure}\\
  \begin{subfigure}{.5\textwidth}
    \centering
    \includegraphics[width=.9\textwidth]{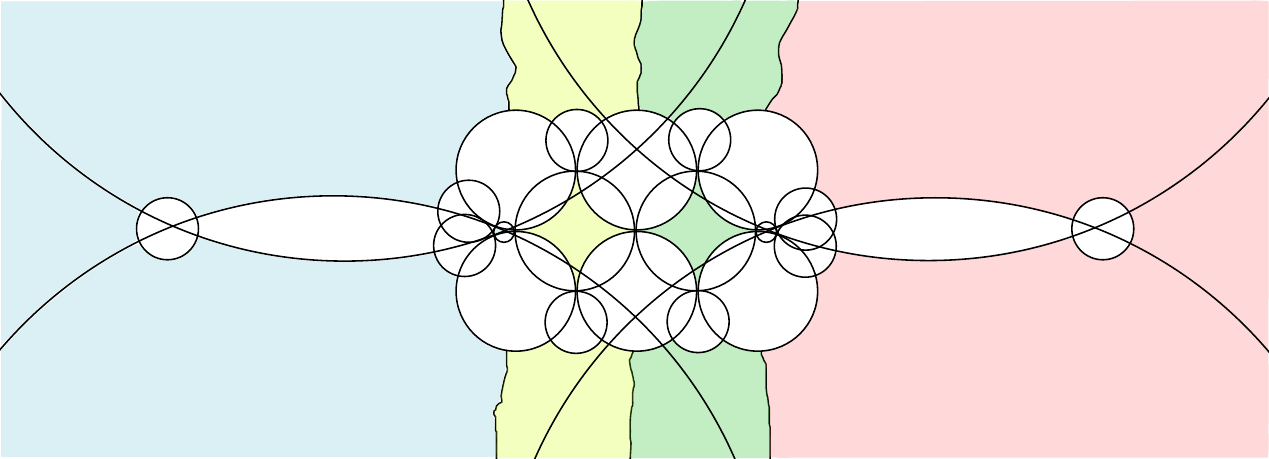}
    \caption{$\theta\to 2\pi$}
  \end{subfigure}%
  \begin{subfigure}{.5\textwidth}
    \centering
    \includegraphics[width=.9\textwidth]{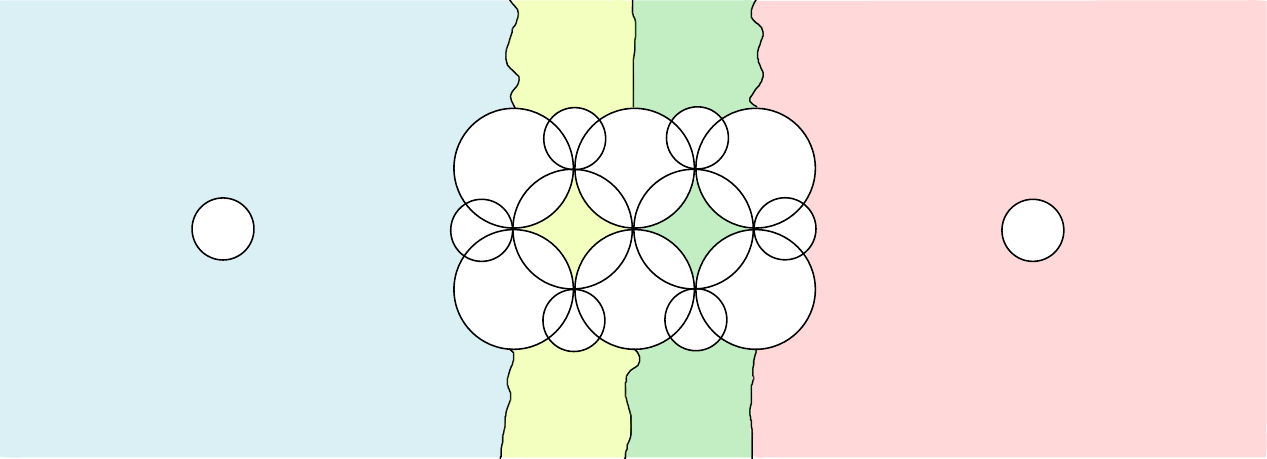}
    \caption{$\theta = 2\pi$}
  \end{subfigure}
  \caption{The deformation of a single connector from a parabolic with holonomy angle $ \theta = 0$, through elliptics of various angles $ \theta \in (0,2\pi) $, to the identity at $ \theta = 2\pi$.\label{fig:deform}}
\end{figure}

\begin{lem}\label{lem:thetalimit}
  For any $ i \in \{1,\ldots,l\} $, one can replace $ x_i $ in \zcref{lem:maskit_decomposition} with an
  elliptic $ x_i^\theta $ of order $ \theta \in [0,2\pi) $. That is, there exists a group
  \begin{displaymath}
    \Gamma^\theta = (G^\theta_1 *_{\langle x^\theta_1 \rangle} *_{\langle x^\theta_2 \rangle} G^\theta_2 \cdots *_{\langle x^\theta_{l-1} \rangle} G^\theta_l * H^\theta_1 * \cdots * H^\theta_n) *_{\langle y^\theta_1 \rangle} \cdots *_{\langle y^\theta_m \rangle}
  \end{displaymath}
  such that
  \begin{itemize}
    \item For all $ i' \neq i $, the element $ x^\theta_{i'} $ is the same type as $ x_i $, and for $ i' = i $ the element $ x^\theta_{i'} $ is elliptic with holonomy angle $\theta $;
    \item For all $ i' $, the element $ y^\theta_{i'} $ is loxodromic;
    \item If $ G_k $ contains $ x_i $, then $ G_k^\theta $ is an $ [0,0,\theta/2] $-triangle group, and otherwise $ G_k^\theta $ is an $ [0,0,0]$-triangle group;
    \item The groups $ G_k^\theta $ and $ H_k^\theta $ have markings, namely $ G_k^\theta = \langle a_k^\theta, b_k^\theta \rangle $ and $ H_k^\theta = \langle c_k^\theta, d_k^\theta \rangle $,
          such that the map on generators
          \begin{displaymath}\begin{tikzcd}[row sep = scriptsize]
            (a_1, \ldots, a_{l-1}, b_1, \ldots, b_{l-1}, c_1, \ldots, c_n, d_1, \ldots, d_n, y_1, \ldots, y_m) \arrow[d, mapsto] \\ (a^\theta_1, \ldots, a^\theta_{l-1}, b^\theta_1, \ldots, b^\theta_{l-1}, c^\theta_1, \ldots, c^\theta_n, d^\theta_1, \ldots, d^\theta_n, y^\theta_1, \ldots, y^\theta_m)
          \end{tikzcd}\end{displaymath}
          extends to an epimorphism $ \phi^\theta : \Gamma \mapsto \Gamma^\theta $ that is type-preserving away from the set $ \{ g^{-1} x_i g : g \in \Gamma \} $;
    \item As $ \theta \to 2\pi $, the two $[0,0,\theta/2]$-triangle groups containing $ x_i^\theta $ converge to cyclic parabolic groups;
    \item $\phi^\theta$ is induced by a homeomorphism $ \Phi^\theta : \H^3/\Gamma = M' \to M^\theta \setminus \Sing(M^\theta) $ where $ M^\theta $ is a cone manifold
          with a single singular arc of angle $ \theta $ whose cylindrical neighbourhood is the image under $ \Phi^\theta$ of a cusp neighbourhood of $ x_i $; and
    \item All maps and metrics involved are smooth in $ \theta $.
  \end{itemize}
\end{lem}
\begin{proof}
  The constructions in \zcref{lem:maskit_decomposition} involve the construction of a family of circles such that the embeddings of $ [0,0,0]$-triangle groups
  into the discs which they bound give the Maskit decomposition of the manifold $ M' $. One can carry out exactly the same construction if the triangle
  groups containing $ x_i $ are replaced by $ [0,0,\theta/2]$-triangle groups, except any step in the construction involving tangent circles at $ \Fix x_i $ should be replaced
  with circles lying in the elliptic and hyperbolic pencils with limits at the two fixed points of $ x_i^\theta $. For the limit to be correct, it is necessary for the fixed
  point of $ x_i^\theta $ lying in the bounded disc $D$ preserved by the $ [0,0,\theta/2]$-triangle group to converge as $ \theta \to 2\pi $ to the opposite vertex of the quadrilateral,
  while the fixed point lying outside $D$ remains fixed. The circle surrounding the latter fixed point (one of the sides paired by the HNN extension that acts
  by conjugacy on $ x_i^\theta $) should be chosen to have small radius, and the two circles paired by $ x_i $ should have radius tending to $\infty$ as $\theta\to 2\pi $.
  This necessarily implies that the discs bounded by the triangle groups containing $ x_i^\theta $ and its conjugate converge to points as $ \theta \to 2\pi $. Finally, the two
  additional circles paired by HNN extensions that are tangent to $ \partial D $ must converge as $ \theta \to 2\pi $, since necessarily the two HNN extensions themselves must converge
  to the same transformation.

  Everything involved can be chosen to move smoothly with $\theta$, since the only constraints are that the various circles constructed do not become close to each other, and this can be arranged
  by shrinking them and shifting them around by continuously varying M\"obius transformations as $ \theta \to 2\pi $. The two groups containing $x_i^\theta$ converge to elementary
  parabolic groups, as in the final frame of \zcref{fig:triangles}. The process should be clarified by \zcref{fig:deform}.
\end{proof}

\begin{proof}[Completion of the proof of \zcref{thm:main}]
  \zcref[S]{lem:thetalimit} shows that one can glue a $2$-handle onto a curve on a non-compression end with genus at least two. It remains to show that one can do the same for curves on a genus one
  non-compression end, and that one can decrease the genus of the compression end via a cone deformation (i.e.\ glue a $2$-handle along a curve which has intersection number $1$ with a
  compression disc boundary).

  First, the genus $1$ end case corresponds to a deformation of a component $H_i$ of $ \Gamma $ which is a rank $2$ elementary parabolic group through
  elementary indiscrete groups generated by a commuting loxodromic and an elliptic of order $ \theta \in (0,2\pi) $ to a group generated by a single
  loxodromic, the elliptic having converged to the identity. Again this is a local move on the fundamental domain, and so this can be done without interacting
  with other component subgroups of $ \Gamma $ that are amalgamated to $ H_i $ by a free product.

  A genus $ n $ classical Schottky group can be deformed into a genus $n-1$ classical Schottky group for all $ n > 1 $ by a similar construction to the rank $2$ parabolic case,
  i.e.\ by pushing the two round circles paired by the $n$th generator together and so deforming that generator through elliptics of order $ (0,2\pi) $ so as to converge to the identity,
  while keeping the remainder of the fundamental domain fixed.

  The theorem now follows follows with the observation that for all of these constructions the two faces of the fundamental domain for $\Gamma^\theta$ containing the axis of $ x_i^\theta $
  converge to embedded surfaces within a fundamental polyhedron for $M$ as $ \theta \to 2\pi $.
\end{proof}

\subsection{A Koebe--Maskit existence theorem}\label{sec:koebemaskit}
As a tangential remark, our combination theorems can be used to prove a version of the Maskit--Koebe uniformisation theorem for function groups in the more general setting
of cone manifolds. We will follow the exposition given by Maskit~\cite[Chapter~X]{maskit}, but since the proof follows the same outline as \zcref{lem:maskit_decomposition}
we will try to be brief.

\begin{defn}\label{defn:signature}
  A \df{signature} is a pair $ (K, t) $ where $ t \geq 0 $ is an integer and $ K $ is a complex consisting of a marked, finite type, possibly disconnected Riemann
  surface $X$ together with a number of $1$-cells joining the marked points of $X$, called \df{connectors}. Each connector is labelled with a value from $ (1,\infty] $.
  A \df{part} of $K$ is a connected component of $ X$. The Euler characteristic of a marked Riemann surface $S$ of genus $ g $ with $n$ marked points
  labelled $ (\alpha_1,\ldots,\alpha_n) $ (so $ \theta_j = 2\pi/\alpha_j $ is the cone angle at each marked point) is the sum
  \begin{displaymath}
    \chi(S) = 2 - 2g - \sum_{i=1}^n (1 - \frac{2\pi}{\alpha_i})
  \end{displaymath}

  We say that the signature $ (K,t) $ is \df{admissible} if it satisfies the following axioms:
  \begin{enumerate}
    \item If there is a part which is a sphere with no marked points, then it is the only part of $X$.
    \item There is no part which is a sphere with a single marked point.
    \item For every part $P$, if $ \chi(P) \geq 0 $ (i.e.\ $P$ is \df{non-hyperbolic}) then $P$ is either
      \begin{enumerate}
        \item a \df{spherical triangle}: a $3$-marked sphere with markings $ (2,\alpha,\beta) $ where $ \alpha, \beta \geq 2 $,
        \item a \df{torus}: a $2$-marked sphere with markings $ (\nu, \nu) $ for $ \nu \in [2,\infty] $,
        \item a \df{Euclidean triangle}: a $3$-marked sphere with markings $ (\alpha,\beta,\gamma) $ so that $ 1/\alpha + 1/\beta + 1/\gamma = 1 $, or
        \item a \df{pillow}: a $4$-marked sphere with markings $(2,2,2,2)$.
      \end{enumerate}
    \item If $ P_1 $ and $ P_2 $ are distinct parts, each a thrice-marked sphere with markings $ (2,2,\nu) $,
          then there is no $\nu$-connector between $P_1$ and $ P_2 $.
  \end{enumerate}
\end{defn}

\begin{rem}
  Maskit's definition of signature~\cite[\S X.D]{maskit} is actually defined for function groups. Recall that a function
  group $\Gamma$ has a distinguished connected component $ \Omega^{*} $ of its domain of discontinuity preserved by $\Gamma$
  so that $ \Omega^{*}/\Gamma $ is the compression end of $ \H^3/\Gamma$. The compression discs
  of $ \H^3/\Gamma $ can be detected by homotopically nontrivial loops in $\Omega^{*}$. Accidental parabolics and elliptics
  correspond to arcs $ \alpha \subset \Omega^{*} $ so that there is an intermediate covering space $ \Omega^{*} \to F \to \Omega^{*}/\Gamma $
  where $ \alpha $ descends to a closed loop in $ F $. The parabolic or elliptic element itself can be recovered from the deck transformation
  group of the maximal such $F$. It appears that much of this machinery goes through when $ \Gamma $ is replaced with cone manifold group $\Hol(M)$ with the following property:
  \begin{equation}\label{eq:conefg}
    \parbox{.75\textwidth}{\centering the visual boundary of the isometric development $ \hat{M} $ has a connected component which is preserved by $ \Hol(M) $.}
  \end{equation}
  This includes the key \textit{planarity theorem} of Maskit~\cite{maskit65p}, see
  also Maskit~\cite[\S X.A]{maskit} and Bowditch~\cite{bowditch22}, since the proof given in \cite{maskit} relies only on \emph{topological} planarity of the domain of discontinuity
  (which holds for cone manifolds, as the isometric developing space is homeomorphic to a sphere) rather than \emph{conformal} planarity (which does not hold for any cone manifolds
  that are not manifolds or orbifolds). In particular, though we do not develop the theory in full in this paper as it would be a large digression from the theme of our
  main results in geometric topology, one should be able to fully classify the `cone manifold function groups' $ \Hol(M) $ satisfying \eqref{eq:conefg} using Maskit's machinery
  without knowing \textit{a priori} that the underlying cone manifold $M$ is a compression body.
\end{rem}

We will define a class of cone manifolds which generalise the compression body manifolds that are uniformised by function groups.
\begin{defn}
  A hyperbolic cone $3$-manifold $M$ is a \df{Koebe--Maskit compression body} if $ M \setminus \Sing(M) $ is a compression body and the
  compression end is compact.
\end{defn}
By the general Ahlfors-Bers theory of deformations of hyperbolic $3$-manifolds and later classification work of geometrically infinite hyperbolic manifolds,
every hyperbolic structure on a compression body lies in the closure of a quasiconformal deformation space of some Koebe--Maskit compression body. The main difference in this paper
is that we also allow elliptics (of possibly infinite order) anywhere that a parabolic can occur.

\begin{lem}
  To each Koebe--Maskit compression body $M$ there exists an assignment $ \Sign(M) $ of a signature so that $ \Sign(M) = \Sign(N) $ if and only if there is
  a type-preserving homeomorphism $ M \to N $: i.e.\ a homeomorphism which sends loops represented by parabolic holonomy elements (resp. elliptic elements) to loops represented
  by parabolic elements (resp. elliptic elements of the same order).
\end{lem}
\begin{proof}
  The proof of well-definedness is exactly that given in Maskit~\cite[\S X.D]{maskit}, so we only sketch the construction itself. Take the compression end
  $ S $, and consider a maximal set $D$ of simple closed essential curves on $ S $ which are represented in $ \Hol(M) $ by elliptic elements, parabolic elements, or the identity:
  i.e.\ they bound embedded discs or $1$-marked discs in $M$. Cut $ S $ along every curve in $D$; for each curve we now have a pair of circle boundary components.
  If a pair is obtained by cutting along a curve represented by a parabolic, glue in two once-marked discs to the corresponding boundaries together with an $\infty$-connector
  joining their centres. If a pair is obtained by cutting along an curve represented by an elliptic of angle $ \theta $, glue in two once-marked discs joined by a $2\pi/\theta$-connector,
  Finally if a pair is obtained by cutting along a compression disc boundary then glue in two unmarked discs. The result is a complex $ K $. Let $ t $ be the number of compression disc
  boundaries that were replaced by pairs of unmarked discs. Then $ \Sign(M) = (K,t) $.
\end{proof}

The following result has been proved independently by a number of authors~\cite{mcowen88,troyanov91,heins62,feng25}, see also Chen and Zhong~\cite{chen25}.
\begin{lem}[Berger--Nirenberg problem]\label{lem:bn_prob}
  If $ S $ is any marked Riemann surface with $ \chi(S) < 0 $, then for any conformal structure on the complement
  of the marked points there exists a unique hyperbolic cone manifold structure on $ S $ which agrees with that conformal
  structure and which has the angles at every marked point given by the marking. The structure is continuous upon variation of the angles. \qed
\end{lem}

We now show that every signature is admissible: that is, one can replace many parabolics in a compression body with compact compression end with singularities of any cone angle in
the interval $ [0,2\pi] $ and the result will still be hyperbolic. This generalises Maskit~\cite[Theorem~X.F.2]{maskit}. We note
that Maskit's arguments are somewhat more involved, because he defines signatures in an entirely complex-analytic way using his Planarity Theorem
and does not use any three-dimensional topology machinery (his signatures are invariants of the group action on $ \hat{\C} $, rather than of the
quotient manifold).
\begin{thm}
  Every admissible signature is realised by some hyperbolic structure on a Koebe--Maskit compression body. More precisely, for every admissible signature $(K,t)$ there
  exists a set of hyperbolic cone surface groups and reducible subgroups of $ \PSL(2,\C) $ together with a sequence of amalgamated products
  and HNN extensions that combines these groups into the holonomy group of a hyperbolic cone manifold
  with signature $ (K,t) $.
\end{thm}
\begin{proof}
  For each part $P$, define a Kleinian group $G_P$ which realises it in the following sense:
  \begin{enumerate}
    \item If $ \chi(P) < 0 $ (hyperbolic case) then $G_P$ is a cone surface group as constructed by \zcref{lem:bn_prob}.
    \item If $ P$ is a spherical triangle part, i.e.\ $ P $ is a 3-marked sphere with $ \chi(P) > 0 $ and markings $ (2,\alpha,\beta)$, then there is a spherical
          triangle with angles $ \pi$,$2\pi/\alpha$, and $2\pi/\beta$ and $G_P$ is the orientation-preserving half of the group generated by the reflections along its sides.
    \item If $ P$ is a torus part, i.e.\ $P$ consists of a sphere with two marked points joined by an $\alpha$-connector, then $G_P$ is the elementary
          group generated by an elliptic of holonomy angle $ 2\pi/\alpha $ (or parabolic, if $ \alpha = \infty$) and a loxodromic element with the same fixed-point
          set (or a parabolic with the same fixed point, if $ \alpha = \infty $).
    \item If $ P$ is a Euclidean triangle part, i.e.\ $ P $ is a 3-marked sphere with $ \chi(P) = 0 $ and markings $ (\alpha,\beta,\gamma)$, then there is a Euclidean
          triangle with angles $ 2\pi/\alpha $, $2\pi/\beta$, and $2\pi/\gamma$ and $G_P$ is the orientation-preserving half of the group generated by the reflections along its sides.
    \item If $ P$ is a pillow part, i.e.\ $ P $ is a sphere with four $2$-marked points, then $G_P $ is the 2222-orbifold group as constructed in Maskit~\cite[\S V.D.9]{maskit}.
  \end{enumerate}
  Each of these groups comes with a canonical fundamental domain cut out by finitely many circles, and every connector incident to a part is associated with one or two parabolic
  or elliptic fixed points that form vertices of this domain.

  The proof now mirrors that of \zcref{lem:maskit_decomposition}.
  Choose a spanning tree for each connected component of the complex $K$ (i.e.\ a maximal subtree of the graph with a vertex for each part and an edges for each connector).
  List the parts in each component in a breadth-first manner, say $ P_1, \ldots, P_n $. Inductively apply \zcref{cor:cone_amalgamation}: let $ \Gamma_1 = G_{P_1} $,
  then for each $ i > 1 $ let $ \Gamma_i = \Gamma_{i-1} *_{\langle h \rangle} G_{P_i} $ where $ \langle h \rangle $ is the cyclic subgroup in corresponding to the connector
  in the tree which is used to reach $ G_{P_i} $ from the part $ G_{P_{j}} $ ($j<i$) which it is connected to. This is possible since $ G_{P_j} $ has a finite-sided
  fundamental domain and so the cusp or cone locus corresponding to the connector in $ G_{P_j} $ has a spherical neighbourhood that the group $ G_{P_i} $ can be glued into.

  (Example: if $ G_{P_i} $ is a spherical triangle group, then let $D$ be its fundamental quadrilateral and let $ H $ be a small hyperbolic disc around the order $2$ fixed point of $ G_{P_i} $
  which is preserved. Then the region $ D \setminus H $, which has a semi-circular boundary, is glued into the half-horoball neighbourhood of the order $2$ fixed point of $ G_{P_j} $ in its fundamental domain.)

  For each connected component $ c $ of $ K $, we have produced a group $ \Gamma_c $.
  Let $ S $ be the Schottky group of rank $ t $ obtained by drawing $2t$ mutually disjoint round circles with common exterior and pairing them arbitrarily with $t$ hyperbolic transformations.
  Conjugating all the groups $\Gamma_c$ by M\"obius transformations if necessary, we may assume that the fundamental domains on the Riemann sphere of $S$ and each group $ \Gamma_c $ lie
  in the interior of an unbounded piece of fundamental domain of every group. By \zcref{cor:cone_amalgamation}, the resulting pattern forms a fundamental domain for the group
  \begin{displaymath}
    \hat\Gamma = \bigast_{c} \Gamma_c * S.
  \end{displaymath}

  Consider now the set $\mc{H}$ of connectors in the complex $K$ that are not realised by one of the amalgamated products already constructed. For each of these, we may find again
  a disc neighbourhood of the corresponding fixed point on the boundary of the corresponding fundamental domain which is sufficiently small to only meet the two incident edges of
  the domain, and we can find a M\"obius transformation which identifies the two discs corresponding to the two ends of each connector in $ \mc{H} $; again we need to deal
  with the problem that some $\infty$-connectors may correspond to multiple vertices in each $\Gamma_c$, but this is fixed as in Step~6 of \zcref{lem:maskit_decomposition}
  by choosing the disc around one vertex arbitrarily and then choosing the others to be translates of this disc by appropriate side-pairing elements. Let $ f_1,\ldots,f_r $
  be the various elements of $ \PSL(2,\C) $ realising these disc identifications, one for each element of $ \mc{H} $. Then by \zcref{cor:cone_hnn} the group
  \begin{displaymath}
    \Gamma = \hat\Gamma *_{f_1} \cdots *_{f_r}
  \end{displaymath}
  is a hyperbolic cone manifold holonomy group that realises the signature $(K,t) $.
\end{proof}

\begin{rem}
\begin{enumerate}
  \item This theorem is not exhaustive, in particular one can actually realise other elementary groups (e.g.\ spherical triangle groups with one reflex angle), but
        the definition of `admissible signature' becomes more convoluted.
  \item It is not clear if all Koebe--Maskit compression body cone manifolds with the same signature are quasiconformally conjugate: the classical theory depends
        on the domain of discontinuity of the groups being conformally planar, which will no longer be the case here (the isometric development of the cone
        manifolds onto a ball $ \B^3 $ with singular hyperbolic metric does not extend conformally to the boundary).
\end{enumerate}
\end{rem}

\section{An explicit deformation}\label{sec:explicit_deformation}
In this section, we give a fully explicit deformation of the form studied in the previous sections,
in order to prove \zcref{thm:weakthm}. We recall from the discussion in \zcref{sec:intro_lp} that
we will deal with two kinds of compression bodies in this section. A \df{(1;2)-compression body}
is a topological $3$-manifold obtained by adding a $1$-handle to one boundary component of a
thickened torus (equivalently, adding a $2$-handle to a thickened genus $2$ surface), and so it has one boundary component of genus two and one boundary component of
genus $1$. The other kind of compression body we will deal with, the $\mc{M}(2)$ manifold, is topologically just a thickened
genus $2$ surface.

\subsection{Construction of explicit parameter spaces}
In this section we construct a parameter space for a subvariety $ \mc{G} $ of the character variety of the fundamental group of the compact genus two surface $ S_{2,0} $ (\zcref{sec:gpgenus2}). We
then construct explicit subspaces of $ \mc{G} $ which parameterise the character varieties of holonomy groups of $(1;2)$-compression bodies with a single rank two cusp and no other cusps (\zcref{sec:gplp}),
and compression body groups with one end a compact genus two surface $S$ and the other end a pair of thrice-punctured spheres arranged so that $S$ has three accidental parabolics (\zcref{sec:gpmax}).

\begin{rem}
  The term variety is justified in all settings we deal with: the $ \SL(2,\C)$ character variety of the thickened genus $2$ surface is reduced and irreducible~\cite[\S 7]{przytycki00},
  and we deal only with $ \SL(2,\C) $ parameterisations of our deformation spaces throughout the paper.
\end{rem}

\subsubsection{Genus two surface groups}\label{sec:gpgenus2}
\begin{figure}
  \centering
  \includegraphics[width=0.45\textwidth]{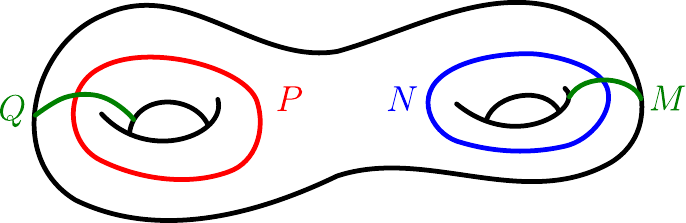}
  \caption{Standard marking on the genus $2$ surface.\label{fig:genus2marking}}
\end{figure}
Take the standard generating set $ \langle P,Q,M,N \rangle $ for $ \pi_1(S_{2,0}) $ shown in \zcref{fig:genus2marking}.
In particular, $ [P,Q][M,N] = I $. We consider the space of representations $ \rho : \pi_1(S_{2,0}) \to \PSL(2,\C) $ such that $ \tr \rho(P) = \tr \rho(P^{-1} N)= 2 $,
normalised so that $ \Fix(P) = \infty $ and
\begin{displaymath}
  M = \begin{bmatrix} \lambda & \lambda^2 - 1 \\ 1 & \lambda \end{bmatrix}\quad\text{where}\quad \lambda \in \C.
\end{displaymath}

\begin{notation}
Since we will be writing many matrices down, we will occasionally write matrices inline: as an example, $ M = (\lambda, \lambda^2 - 1 \mid 1, \lambda) $.
\end{notation}

We can quickly compute the dimension of the representation space: $4$ generators with $3$ degrees of freedom, minus $2$ (the two trace conditions), minus $3$ (surface group relation) and
minus $3$ (normalisation) gives complex dimension $4$.

We will construct an explicit realisation of the space. Setting
\begin{gather*}
  P = \begin{bmatrix} -1 & \alpha \\ 0 & -1 \end{bmatrix}\!,\quad
  Q = \begin{bmatrix} \rho & \beta \\ \sigma & (1+\beta \sigma)/\rho \end{bmatrix}\!,\\
  M = \begin{bmatrix} \lambda & \lambda^2-1 \\ 1 & \lambda \end{bmatrix}\!,\quad\text{and}\quad
  N = \begin{bmatrix} r & b \\ s & (1+b s)/r \end{bmatrix}
\end{gather*}
we have an ambient eight-dimensional space which contains a subvariety $ \mc{G} $ obtained by pushing forwards the surface group relator. In fact, $ \mc{G} $ is
a holomorphic parameter space of genus two surface group representations, and using a computer algebra system we can eliminate all but the four parameters
$ (\alpha, \beta, \sigma, \lambda) $ to give an exact parameterisation of this subvariety in $ \C^8 $. More precisely, we can use the substitutions
\begin{align*}
  b &= \frac{\alpha}{\xi}\Big(\sqrt{\kappa} (1 - \lambda^2 + \alpha \sigma -  \lambda \sigma - \alpha \lambda^2 \sigma + \lambda^3  \sigma) + \zeta + 2\sigma\lambda(1 - \lambda^2)\\
    &\quad{} + \sigma^2(2+\alpha^2+3\alpha\lambda +(\alpha^2-4)\lambda^2-3 \alpha \lambda^3+2 \lambda^4)+ \alpha\sigma^3(1+\alpha\lambda-\lambda^2)\Big)\\
  s &= \frac{1}{\xi}\Big(\alpha \sqrt{\kappa} (-1 - \lambda \sigma) + \alpha (-2 - 2 \lambda \sigma + 2 \sigma^2 + \alpha \lambda \sigma^2 - 2 \lambda^2 \sigma^2 + \alpha \sigma^3)\Big)\\
  r &= -\frac{1}{2\xi} \Big(\sqrt{\kappa} (-\alpha^2 + 2 \alpha \sigma - 2 \alpha \lambda^2 \sigma + \alpha^2 \sigma^2)\\
    &\quad{}+ 4 \zeta - \alpha^3 \sigma + 2\alpha\sigma^2(\alpha + 2 \lambda + \alpha \lambda^2- 2 \lambda^3) + \alpha^3 \sigma^3\Big)\\
  \rho &= \frac{1}{2}\Big( \alpha\sigma - \sqrt{\kappa} \Big)
\end{align*}
where
\begin{gather*}
  \kappa = 4 - 4  \sigma^2 + \alpha^2  \sigma^2 - 4  \alpha  \lambda  \sigma^2 + 4  \lambda^2  \sigma^2\\
  \xi = 4 + \alpha^2 - 4  \lambda^2 + 4  \alpha  \sigma + 2  \alpha^2  \lambda  \sigma - 4  \alpha  \lambda^2  \sigma + \alpha^2  \sigma^2\\
  \zeta = 2 - 2 \lambda^2 + 2 \alpha  \sigma + \alpha^2  \lambda  \sigma - 2 \alpha  \lambda^2  \sigma.
\end{gather*}

By standard trace formulae results surveyed in~\cite[\S\S3.4--3.5]{maclachlan_reid}, we know that the traces in $ \langle P,Q,M,N \rangle $ should
be representable in terms of polynomials in traces of words of length at most $3$. Since $ \tr P $ is fixed and we know we need $4$ complex parameters,
this suggests that a better (in the sense of being easier for human consumption) set of parameters would be $ \tr Q $, $ \tr M $, $ \tr N $, and $ \tr [P,Q] $
(or some other nice word that includes $ P $). Writing down explicit matrices in terms of these new parameters is a nontrivial task.

\subsubsection{$(1;2)$-compression body groups}\label{sec:gplp}
We will use the following normalisation for $(1;2)$-compression body groups in $ X(\pi_1(S_{2,0})) $, which has been chosen to make the isometric circles of $ M $ symmetric about $0$:
\begin{displaymath}
  P = \begin{bmatrix} 1 & \alpha \\ 0 & 1 \end{bmatrix}\!,\quad Q = \begin{bmatrix} 1 & \beta \\ 0 & 1 \end{bmatrix}\!,\quad\text{and}\quad M = \begin{bmatrix} \lambda & \lambda^2 - 1 \\ 1 & \lambda \end{bmatrix}\!.
\end{displaymath}
We write $ G(\alpha,\beta,\lambda) $ for the marked subgroup of $ \PSL(2,\C) $ given by this normalisation. An example of a discrete group in this family
is shown in \zcref{fig:limit_set}. The set $ \mc{CB} = \{ (\alpha,\beta,\sigma,\lambda) \in \mc{G} : \sigma = 0 \} $ parameterises the $(1;2)$-compression body group representations
in $ \mc{G} $.

\begin{figure}
  \centering
  \includegraphics[width=0.5\textwidth]{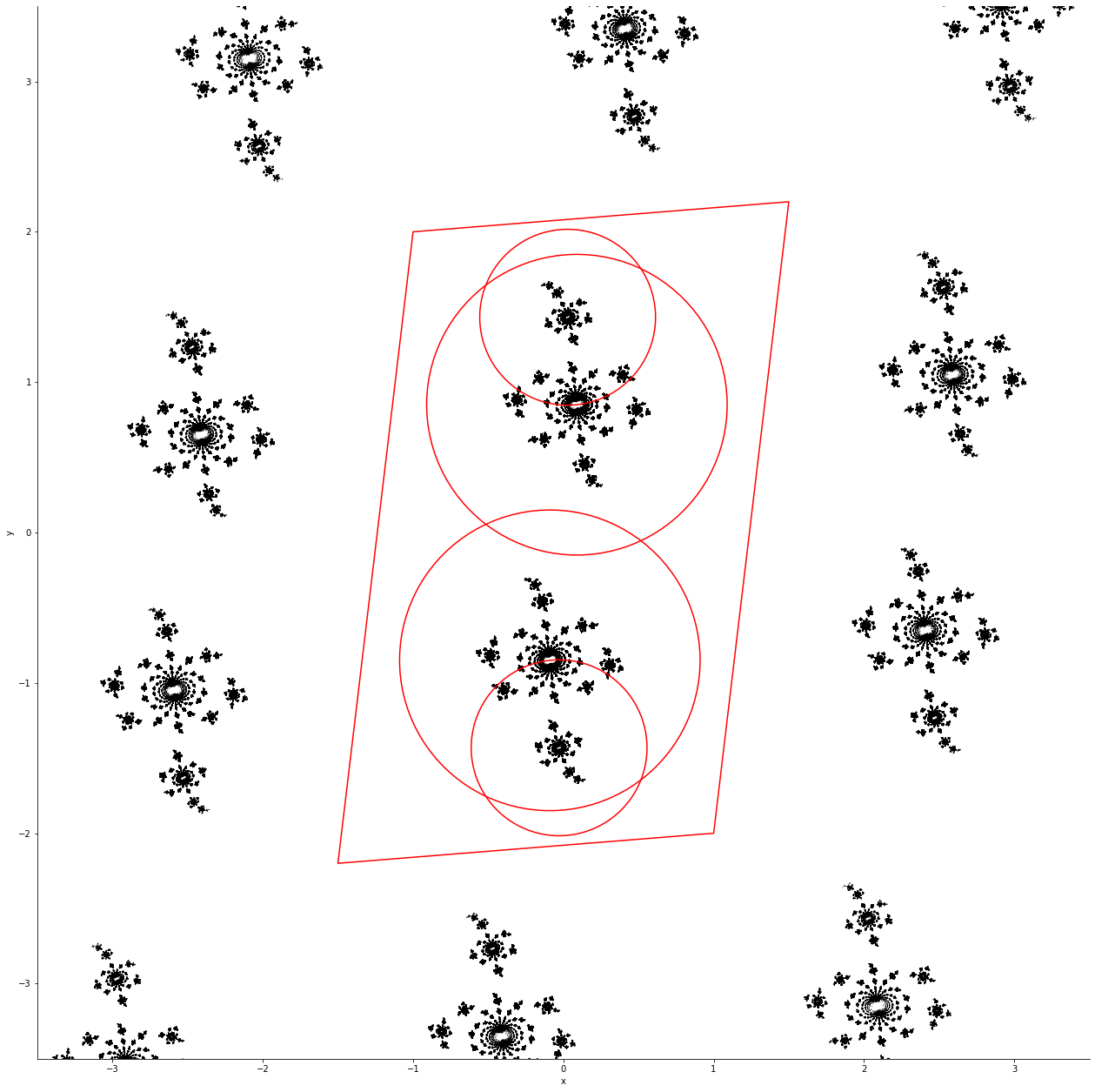}
  \caption{The limit set (in black) of the $(1;2)$-compression body group with parameters $ \alpha = 2.5+0.2i $, $ \beta = 0.5 + 4.2i $, and $ \lambda = 0.09 + 0.85i $. We also show (in red)
  the boundary on the Riemann sphere of a fundamental domain (Ford domain) for the group from which the hyperbolic quotient can be determined by an application of the Poincar\'e polyhedron
  theorem~\cite[\S IV.H]{maskit}.\label{fig:limit_set}}
\end{figure}

\subsubsection{Groups with a maximally cusped end}\label{sec:gpmax}
Finally, we will construct an explicit example of a holonomy group of an $ \mc{M}(2)$-manifold (see \zcref{defn:m2mfd} above) and then embed this example into a deformation space.

\begin{lem}
  The group
  \begin{align*}
    B = \Bigg\langle X = \begin{bmatrix} 1 & 2 \\ 0 & 1 \end{bmatrix}\!,\quad
                     Y_1 &= \begin{bmatrix} 1+2i & 2 \\ 2 & 1-2i \end{bmatrix}\!,\quad
                     Y_2 = \begin{bmatrix} 1-2i & 2 \\ 2 & 1+2i \end{bmatrix}\!,\\
                     J &= \begin{bmatrix} 3 & -2i \\ 4i & 3 \end{bmatrix}\!,\quad
                     J' = \begin{bmatrix} 3+4i & -6i \\ 4i & 3-4i \end{bmatrix} \Bigg\rangle
  \end{align*}
  is the holonomy group of an $ \mc{M}(2)$-manifold.
\end{lem}
\begin{proof}
  The signature of the desired function group is
  \begin{displaymath}
    \Bigg(\begin{tikzpicture}[baseline=-.12cm]
              \draw[fill=black!10] (-.75,0) ellipse (.33cm and .75cm);
              \draw[fill=black!10] (.75,0) ellipse (.33cm and .75cm);
              \draw (-.75,.4) --node[above]{$\infty$} (.75,.4);
              \draw (-.75,0) --node[above]{$\infty$} (.75,0);
              \draw (-.75,-.4) --node[above]{$\infty$} (.75,-.4);
            \end{tikzpicture}\,,\; 0 \Bigg).
  \end{displaymath}
  Following the construction in \zcref{lem:maskit_decomposition}, we begin with two thrice-punctured sphere groups and perform one amalgamated
  product and two HNN extensions.

  Let $ X = (1, 2 \mid 0, 1) $ and $ Y = (1,0\mid2,1) $, so $ \langle X,Y \rangle $ is the
  standard $ [0,0,0] $-triangle group. We take two copies $ G_1 $ and $ G_2 $ of this group, viewed as Fuchsian groups where $ G_1 $ acts
  on $ \H^2 $ and $ G_2 $ acts on $ \bar{\H}^2 $. These groups have cusp neighbourhoods $ B_1 = \{ \Im z > 1 \} $ and $ B_2 = \{ \Im z < -1 \} $ respectively.
  Define $ \Phi_1 = (1,i\mid0,1) $ and $ \Phi_2=(1,-i\mid0,1) $, and set $ Y_1 = Y^{\Phi_1} = \Phi_1 Y \Phi_1^{-1} $ and $ Y_2 = Y^{\Phi_2} $. The
  conjugate groups $ \tilde{G}_1 = G_1^{\Phi_1} = \langle X, Y_1 \rangle $ and $\tilde{G}_2 = G_2^{\Phi_2} = \langle X, Y_2 \rangle$ both have
  cusps at $ \infty $, with cusp neighbourhoods $ \H^2 $ (for $ \tilde{G}_1 $) and $ \bar{\H}^2 $ (for $ \tilde{G}_2 $), and common cusp neighbourhood
  stabiliser $ \langle X \rangle $. Hence $ G' = \tilde{G}_1 *_{\langle X \rangle} \tilde{G}_2 = \langle X, Y_1, Y_2 \rangle $ has quotient surface
  consisting of three components, two thrice-punctured spheres and one four-punctured sphere, such that the four-punctured sphere has an accidental parabolic represented by $ X $.

  The next step is to hide the extra cusps of the four-punctured sphere by HNN extensions. First consider the two cusps $ \Fix(Y_1) = i $ and $ \Fix(Y_2) = -i $; these cusps
  have respective horosphere neighbourhoods on the four-punctured sphere component $ B_1 = \B_{1/4}(3i/4) $ and $ B_2 = \B_{1/4}(-3i/4) $, and we wish to find a transformation
  $ J = (a,b\mid c,d) $ which maps the exterior of $ B_1 $ onto the interior of $ B_2 $ and which sends $ i \mapsto -i $. Since $ \partial B_1 $ and $ \partial B_2 $ are equal radius circles they must
  be the isometric circles of $ J $ and so we obtain the equations $ a/c = d/c = -3i/4 $ and $ ai+b = -i(ci+d) $; direct computation produces that $ J = (3, -2i \mid 4i, 3) $,
  and so we obtain the HNN extension $ G'' = G' *_{\langle J \rangle} $ as a group which has two out of the three desired accidental parabolics. To obtain the final extension, it
  is enough to take $ J'(z) = J(z-1)+1 $ so $ J' = (3+4i, -6i\mid4i, 3-4i) $ and then set $ B = G'' *_{\langle J'\rangle} $.
\end{proof}

The group $ B $ has the desired quotient space structure, and from the construction we also have a natural fundamental domain, \zcref{fig:function_gp}.

\begin{figure}
  \centering
  \begin{subfigure}{0.48\textwidth}
    \centering
    \includegraphics[width=\textwidth]{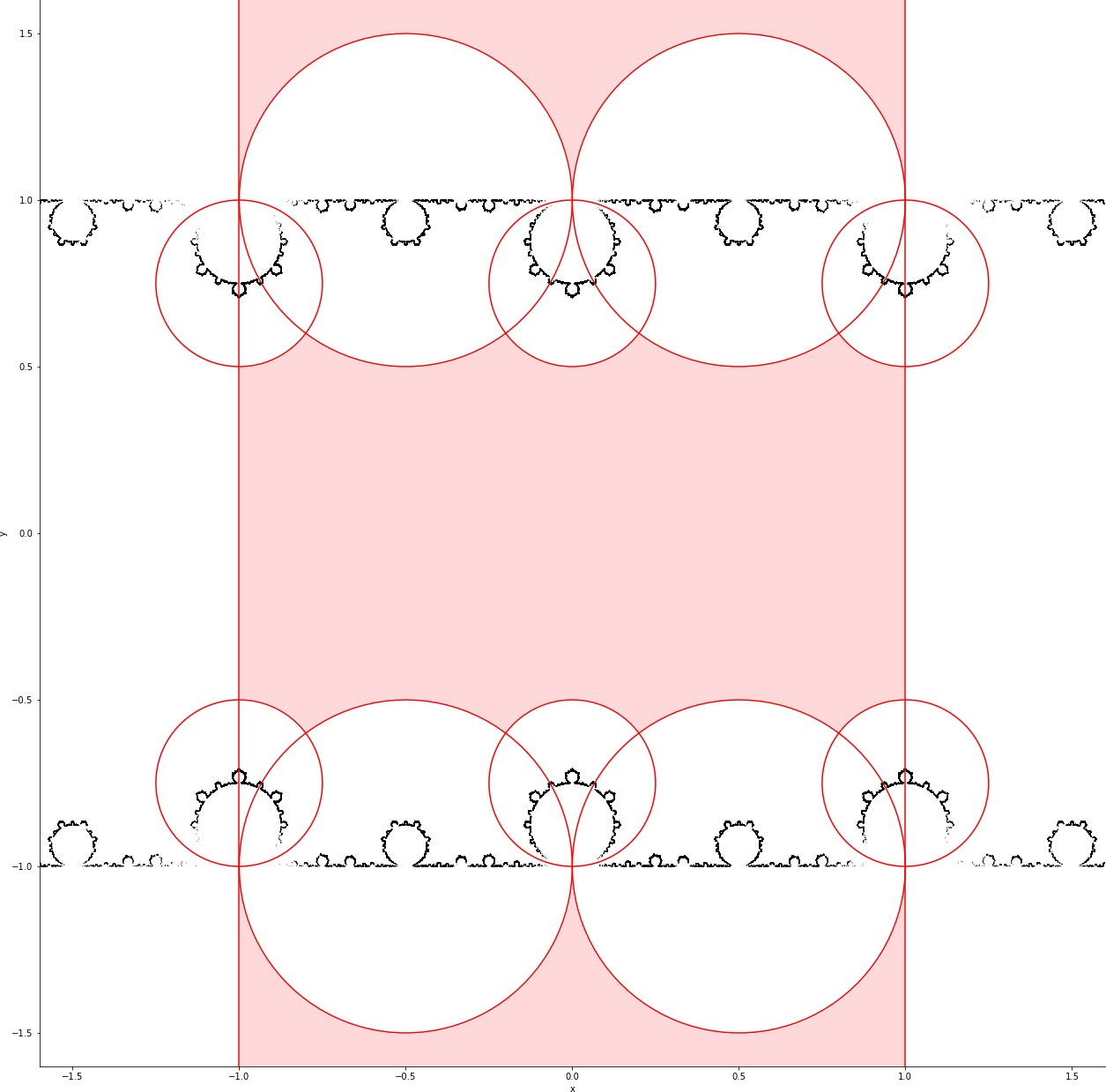}
    \caption{Fundamental domain and limit set.\label{fig:function_gp}}
  \end{subfigure}\hfill
  \begin{subfigure}{0.48\textwidth}
    \centering
    \includegraphics[width=\textwidth]{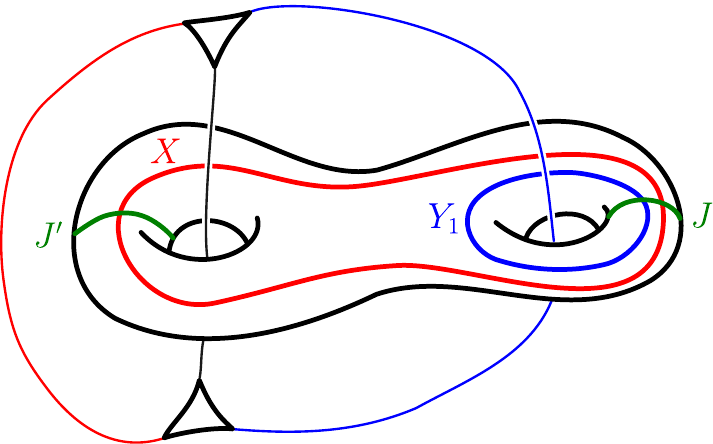}
    \caption{Quotient $3$-manifold.\label{fig:fgroup_marking}}
  \end{subfigure}
  \caption{Geometric objects associated to the Koebe group $B$ constructed in \zcref{sec:gpmax}. The curves
  labelled in (\textsc{B}) are the loops represented by the corresponding elements, not the projection of the paired sides of the fundamental polygon.}
\end{figure}

\begin{lem}\label{lem:conjugated_B}
  The group $ B $ is conjugated to a representation in $ \mc{G} $ by
  \begin{equation}\label{eq:conjugator}
    \Phi = \frac{1}{\sqrt{17}} \begin{bmatrix} 4\sqrt{-1+4i} & \sqrt{8+2i} \\
                          -\sqrt{-2-i} & -\sqrt{1-4i} \end{bmatrix}.
  \end{equation}
  The parameters of the conjugate $\Phi B \Phi^{-1} $ are
  \begin{displaymath}
    \begin{array}{lll}
    \alpha = 16-4i & \rho = (-5+20i)/17 & \beta = (552-32i)/17\\
                    & \sigma = (-1+4i)/17\\
    \lambda = 3 & r = (31+12i)/17 & b = (32+76i)/17\\
                & s = (-4-i)/17.
    \end{array}
  \end{displaymath}
\end{lem}
\begin{proof}
  We define $ \Phi $ to send the fixed point of $ Y_1 X^{-1} $ to $ \infty $ and to conjugate $ J $ to $ (\lambda, \lambda^2-1 | 1,\lambda) $ where $ 2\lambda = \tr J $.
  \textit{A priori} this is an overdetermined system of equations, but it has the unique solution \eqref{eq:conjugator}.
  Applying $ \Phi $ to $B$ and comparing \zcref{fig:genus2marking} to \zcref{fig:fgroup_marking}, the normalised generators are
  \begin{gather*}
    P = \Phi Y_1 X^{-1} \Phi^{-1} = \begin{bmatrix} 1 & -16+4i \\ 0 & 1 \end{bmatrix}\!,\quad
    Q = \Phi J' \Phi^{-1} = \frac{1}{17} \begin{bmatrix} -5+20i & 552-32i \\ -1+4i & 107 - 20i \end{bmatrix}\!,\\
    M = \Phi J \Phi^{-1} = \begin{bmatrix} 3 & 8 \\ 1 & 3 \end{bmatrix}\!,\quad\text{and}\quad
    N = \Phi Y_1 \Phi^{-1} = \frac{1}{17} \begin{bmatrix} 31+12i & 32+76i \\ -4-i & 3-12i \end{bmatrix}
  \end{gather*}
  from which we read the parameters directly.
\end{proof}

\begin{prp}\label{prp:parameterised_B}
  Let $ X $ be the subvariety of $ \mc{G}$ of representations $ \rho $ such that $ \tr N = 2 $ and $ \tr P^{-1} N = 2 $ which contains the group constructed in \zcref{lem:conjugated_B}.
  This is a three-dimensional variety parameterised by $ \alpha $, $ \beta $, and $ \sigma $ cut out by the equations
  \begin{gather*}
    \rho = \frac{-8+\alpha\sigma^2(-4+\alpha-\alpha\sigma)}{8+4\alpha\sigma},\quad \lambda = \frac{-8+\alpha\sigma(-4+\alpha-\alpha\sigma)}{8+4\alpha\sigma}\\
    r = \frac{2+4\sigma+\alpha\sigma(2+\sigma)}{2+\alpha\sigma}, \quad b = \frac{\alpha\sigma^2(4+\alpha+\alpha\sigma)^2}{4(2+\alpha\sigma)^2}, \quad\text{and}\quad s = -\frac{4}{\alpha}.
  \end{gather*}
\end{prp}
\begin{proof}
  We carried out a standard computer algebra elimination of variables, and chose the correct component of the solution set by hand.
\end{proof}

Denote by $ B(\alpha,\beta,\sigma) $ the groups $ \langle P,Q,M,N \rangle $ defined by the parameterisation in \zcref{prp:parameterised_B}; they form a subvariety $ \mc{F} \subset \mc{G} $
which parameterises $ \mc{M}(2)$-groups.

\subsection{Deformations through cone manifolds}
We have constructed three parameter spaces:
\begin{displaymath}
  \begin{tikzcd}[sep=small, column sep=tiny]
  \CB \arrow[r] & \mc{G} & \F \arrow[l] & \text{concrete parameter spaces}\\
  \QH(G) \arrow[dr]\arrow[u,hook] && \QH(B) \arrow[dl]\arrow[u,hook] & \text{discrete loci}\\
  &\Teich(S_{2,0}) && \text{conformal structures}
  \end{tikzcd}
\end{displaymath}
In this diagram, $ \QH(G) $ and $ \QH(B) $ denote the quasiconformal deformation spaces of the $(1;2)$-compression body groups and the maximally cusped genus $2$ groups,
respectively. These spaces lift to interiors of components of the discreteness loci inside $ \CB $ and $ \F $, and are embeddings of quotients of the Teichm\"uller space
of the compact genus two surface~\cite[Theorem~5.27]{matsuzakitaniguchi}. While $ \CB $ and $ \F $ are cut out by algebraic equations in $ \mc{G} $, the
quasiconformal deformation spaces are incredibly wild open subsets that we have very little control over~\cite{ohshika20,magid12,bromberg11}. We will construct a family of representations which joins
the two discreteness loci by deforming a single element of the group from a parabolic to the identity through cone manifolds.

We consider the $\mc{F}$-group constructed above, in \zcref{sec:gpmax}. We ask for a continuous family of deformations of the fundamental domain of this group
to the fundamental domain of a $\mc{CB}$-group. We do this by deforming the two $ [0,0,0]$-triangle groups through a series of fundamental domains as
described in the previous section. The point now is that the action of deforming the peripheral circles of one of the thrice-punctured
sphere groups to `bubble' off the isometric circles of one of the stable elements of the HNN extensions is a local deformation in $ \hat{\C} $, and there are no obstructions
for the deformation to hit. With the methods here, we do not control the conformal structures on either end---we simply show that there exists \emph{some} such path.
Of course, there will in reality be \emph{many} such paths.

In order to construct the deformation, we require the following to hold:
\begin{itemize}
  \item The vertical lines paired by $ X $ should be fixed in place since $ X $ will remain a parabolic through the deformation.
  \item The circles paired by $ Y_1 $ and $ Y_2 $ should remain tangent to the vertical lines paired by $ X $.
  \item The circles paired by $ J $ should split off and become the circles paired by the loxodromic generator of the $\mc{CB}$-group.
  \item The circles paired by $ J' $ and $ J'' $ should become horizontal lines, paired by the second parabolic generator of the $\mc{CB}$-group,
        and should always be orthogonal to the vertical lines paired by $ X $ and the circles paired by the $ Y_j $.
\end{itemize}

\begin{figure}
  \centering
  \begin{tikzpicture}[scale=1.5]
    \draw[color=black!50,->] (-1.5,-.2) -- (1.5,-.2) node[label={right:$\R$}] {};
    \draw[color=black!50,->] (0,-.5) -- (0,2) node[label={above:$i\R$}] {};
    \draw[thick] (-1,-0.5)--(-1,2);
    \draw[thick] (1,-0.5)--(1,2);
    \draw (0, .5) node[inner sep=1pt, fill=black, circle, label={below:$i$}] {};
    \draw (-0.45, .6) node {$\theta$};
    \draw (-0.268328, 0.634164) arc [end angle = 180, delta angle=26.565, radius=.3];
    \draw[thick] (-0.378, 1) circle (0.625);
    \draw[thick] (0.378, 1) circle (0.625);
    \draw (0,.5)--(-1,1)--(-1,.5)--(0,.5);
    \draw[dashed] (-1,1)--(0,1.5)--(0,.5);
    \draw[dashed] (-1,1)--(-0.378, 1)--(0,.5);
    \draw[dashed] (-0.378, 1)--(0,1.5);
    \draw (-0.378, 1) node[inner sep=1pt, fill=black, circle] {};
    \draw (-0.463, 1.132) node {$\zeta$};
  \end{tikzpicture}
  \caption{Our deformation is parameterised by $ \theta $, which controls the isometric circles of $ Y_1 $ as shown.\label{fig:two_triangles_notation}}
\end{figure}

In order to compute the deformation, we will fix one of the intersection points of the isometric circles of $ Y_1 $ at $ i $,
send the other intersection point to $ \infty $ along the positive $\Im$-axis while keeping the isometric circles tangent to the vertical
lines $ \Im z = \pm 1 $, and parameterise based on the angle $ \theta $ that joins $ i$ to the point of tangency (\zcref{fig:two_triangles_notation}).
One can show that the centre of the left circle is
\begin{displaymath}
  \zeta = -\frac{\tan(\theta)}{\tan(2\theta)} + i(1+\tan\theta),
\end{displaymath}
the centre of the other circle is $ -\overline{\zeta} $, and the two circles have radius $ \tan(\theta)/\sin(2\theta) $. The angle between the two circles is $4\theta$. Observe that all these
quantities have well-defined limits as $ \theta \to 0 $ (the $\mc{F}$-group) and as $ \theta \to \pi/2 $ (the $\mc{CB}$-group). We can directly compute
that the elliptic transformation pairing them is
\begin{displaymath}
  Y_1 = \begin{bmatrix} -\cos2\theta + i(1+\cos 2\theta + \sin 2\theta) & 1+\cos2\theta+2\sin2\theta \\
                                \sin 2\theta \cot\theta & -\cos 2\theta - i(1+\cos 2\theta + \sin 2\theta) \end{bmatrix}.
\end{displaymath}

In order to ensure the angle-gluing equations are correct at all stages, it is sufficient to require all orthogonal circles to remain orthogonal throughout the deformation, that $J$ and $ J'$ remain
hyperbolic throughout the deformation, and that complex conjugation remains a symmetry of the circle configuration. These pin down all data except for the radii of the isometric circles
of $J$ (resp. $ J'$). The pencil of circles in which the upper isometric circle of $ J $ must lie is the pencil of circles orthogonal to the two isometric circles of $ Y_1 $,
i.e.\ the hyperbolic pencil with loci at $ i $ and $ i (1+2\tan\theta) $. If $ [k:a:b:h] \in \P^3 $ represents the circle $ 0 = k(x^2+y^2) +ax + by +h $, c.f.~\cite[\S 3.2]{beardon} (though
we use a slightly different projective base compared to that reference), then the isometric circle of $J$ is of the form
\begin{align*}
  &\nu \left[1 : 0 : -2 : 1\right] + (1-\nu) \left[1 : 0 : -2(1+2\tan\theta) : (1+2\tan\theta)^2\right]\\
  &\quad{}= \left[1 : 0 : -2 + 4(\nu-1)\tan\theta : 1+4(1-\nu)\tan\theta+4(1-\nu)\tan^2\theta\right]
\end{align*}
where $ \nu \in \R $ is some parameter.

Similarly the upper isometric circle of $ J' $, which passes through the point $ -1 + i(1+\tan\theta) $ and is orthogonal to the line $ \Im z = -1 $,
is of the form
\begin{displaymath}
  \left[1 : 2 : -2(1+\tan\theta+\varrho) : 1-\varrho^2 + (1+\tan\theta-\varrho)^2\right]
\end{displaymath}
parameterised by its (signed depending on whether it lies above or below its tripoint with $ -1+i\R $ and the circle of $ Y_1 $) radius $ \varrho $.

We now choose our parameters $ \nu $ and $ \varrho $, which we can do essentially arbitrarily except we need to ensure that the isometric circles of $ J $ and $ J' $ stay
disjoint throughout the deformation. To make this simple, we choose $ \nu $ so that the isometric circle of $ J$ has constant radius $ 1/4 $; i.e.
\begin{displaymath}
  \frac{1}{4} = (-2 + 4(\nu-1)\tan\theta)^2 - 4(1+4(1-\nu)\tan\theta+4(1-\nu)\tan^2\theta).
\end{displaymath}
Take
\begin{displaymath}
  \varrho = \frac{1}{2}\left( \tan \theta + 1 - \frac{1}{4} (5-3\cos 2\theta) \right),
\end{displaymath}
which is computed by requiring the bottom of the isometric circle of $J$ in the upper halfplane to be at height $ 1/2 + (3/4)(1-\cos(2\theta)) $. It is a routine
calculation to see that the isometric circles do not collide.

To compute explicit matrices for $ J $ and $ J' $, we just substitute these calculations into the hyperbolic transformation with isometric circles centred at $ w $
and $ \overline{w} $ with radius $ r $: this has the matrix
\begin{displaymath}
  \begin{bmatrix}
    iw/r & i(r- w\overline{w}/r) \\
    i/r & -i \overline{w}/r
  \end{bmatrix}.
\end{displaymath}

The resulting deformation is shown in \zcref{fig:deformation_animation}; we have produced a smooth path of cone manifold holonomy groups joining discrete groups
of the desired kinds. Since we started with the group produced by the Koebe construction in \zcref{sec:gpmax} and not a group of the form $ B(\alpha,\beta,\sigma) $ it is not immediately a path
in $ \mc{G} $, but it is a routine calculation in M\"obius geometry to construct a smoothly parameterised family of M\"obius transforms, exactly as in \zcref{lem:conjugated_B}, which transform
each group constructed here into a corresponding point in $ \mc{G} $. The point is that these maps are defined only in terms of fixed points of generators, and these fixed points move smoothly
under the deformation we have constructed.

\begin{figure}
  \begin{subfigure}{0.32\textwidth}
    \centering
    \includegraphics[width=\textwidth]{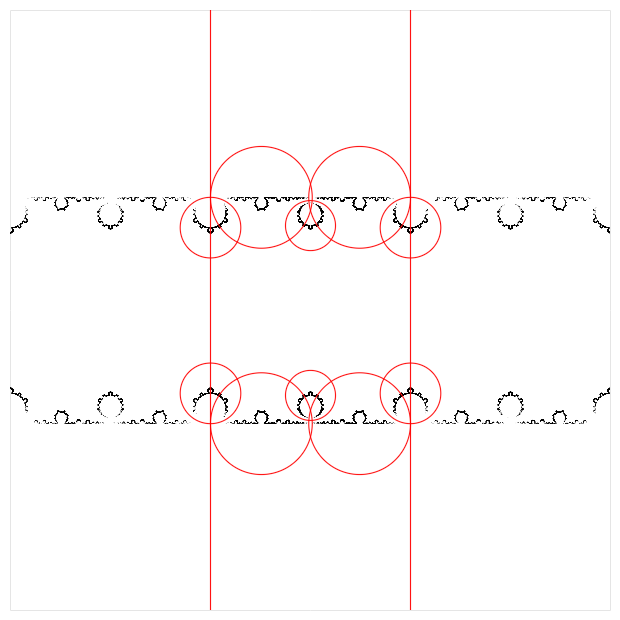}
    \caption*{$\theta = \pi/24$}
  \end{subfigure}%
  \begin{subfigure}{0.32\textwidth}
    \centering
    \includegraphics[width=\textwidth]{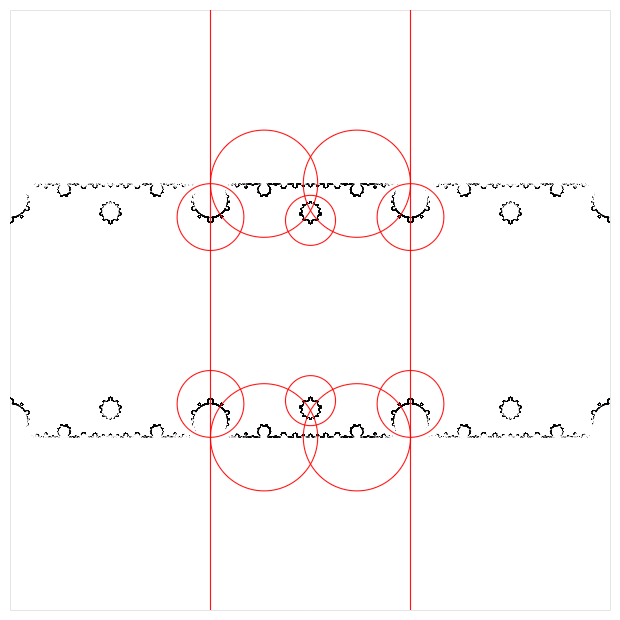}
    \caption*{$\theta = 2\pi/24$}
  \end{subfigure}%
  \begin{subfigure}{0.32\textwidth}
    \centering
    \includegraphics[width=\textwidth]{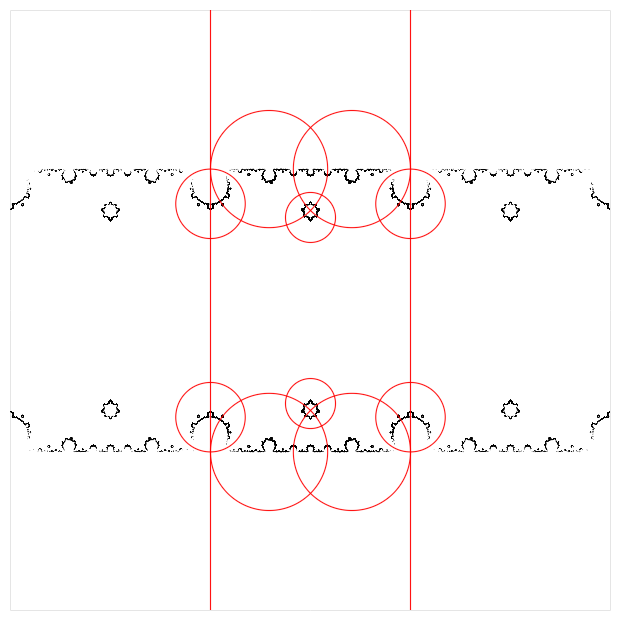}
    \caption*{$\theta = 3\pi/24$}
  \end{subfigure}\\
  \begin{subfigure}{0.32\textwidth}
    \centering
    \includegraphics[width=\textwidth]{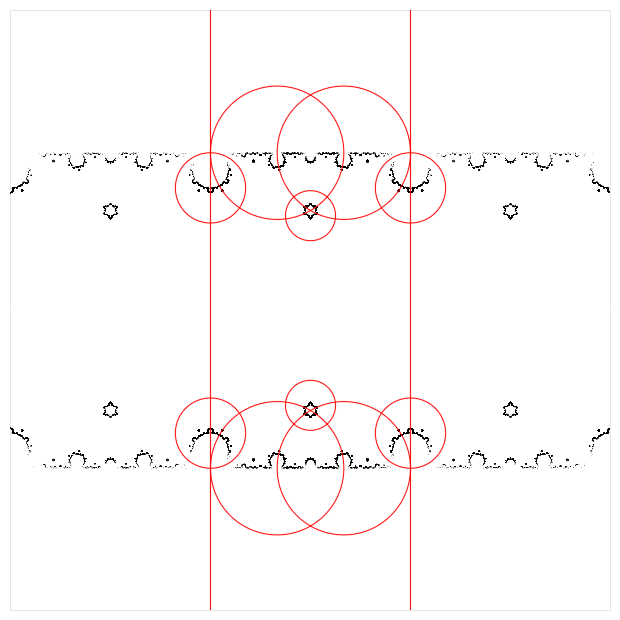}
    \caption*{$\theta = 4\pi/24$}
  \end{subfigure}%
  \begin{subfigure}{0.32\textwidth}
    \centering
    \includegraphics[width=\textwidth]{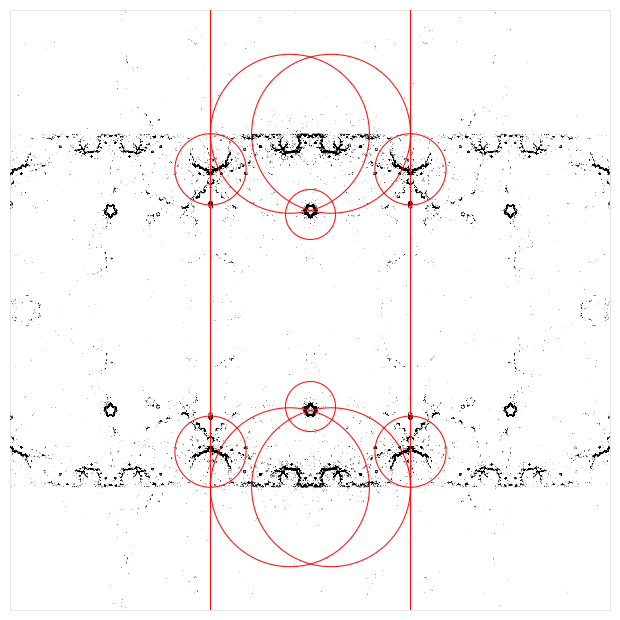}
    \caption*{$\theta = 5\pi/24$}
  \end{subfigure}%
  \begin{subfigure}{0.32\textwidth}
    \centering
    \includegraphics[width=\textwidth]{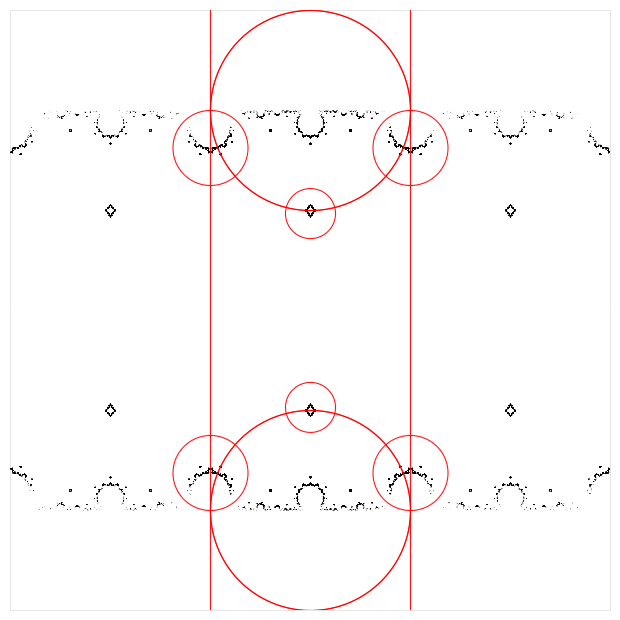}
    \caption*{$\theta = 6\pi/24$}
  \end{subfigure}\\
  \begin{subfigure}{0.32\textwidth}
    \centering
    \includegraphics[width=\textwidth]{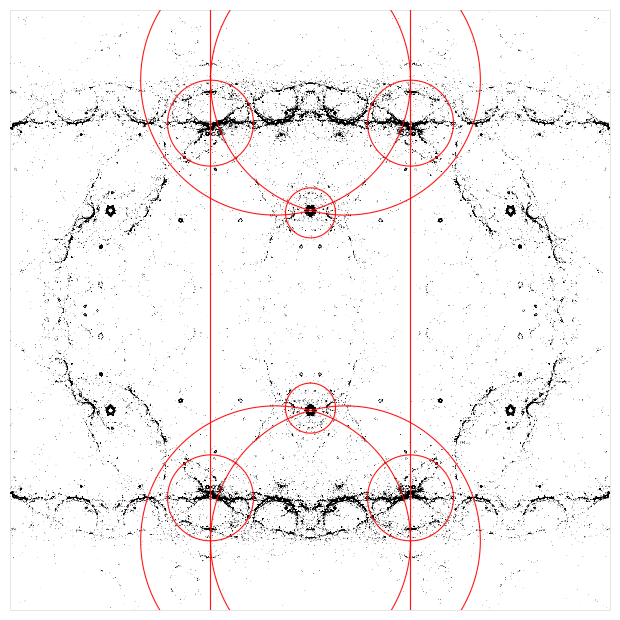}
    \caption*{$\theta = 7\pi/24$}
  \end{subfigure}%
  \begin{subfigure}{0.32\textwidth}
    \centering
    \includegraphics[width=\textwidth]{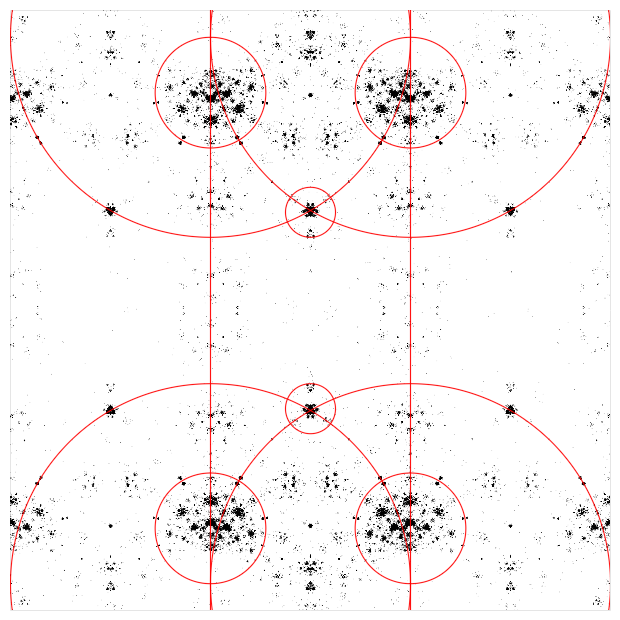}
    \caption*{$\theta = 8\pi/24$}
  \end{subfigure}%
  \begin{subfigure}{0.32\textwidth}
    \centering
    \includegraphics[width=\textwidth]{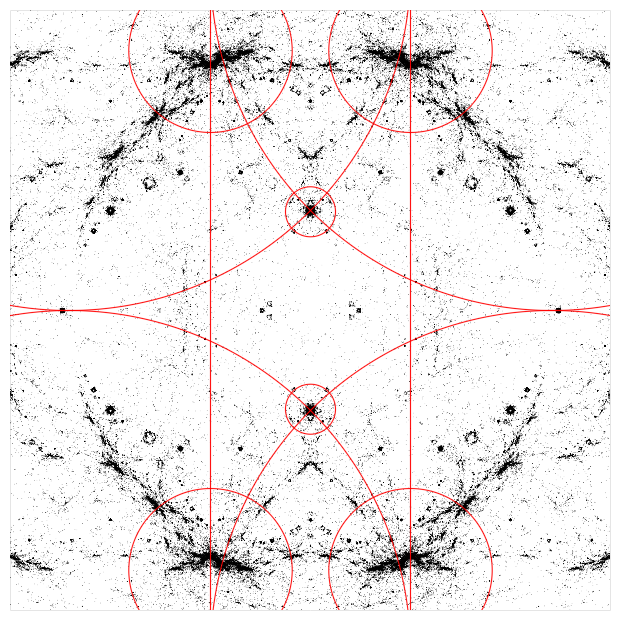}
    \caption*{$\theta = 9\pi/24$}
  \end{subfigure}\\
  \begin{subfigure}{0.32\textwidth}
    \centering
    \includegraphics[width=\textwidth]{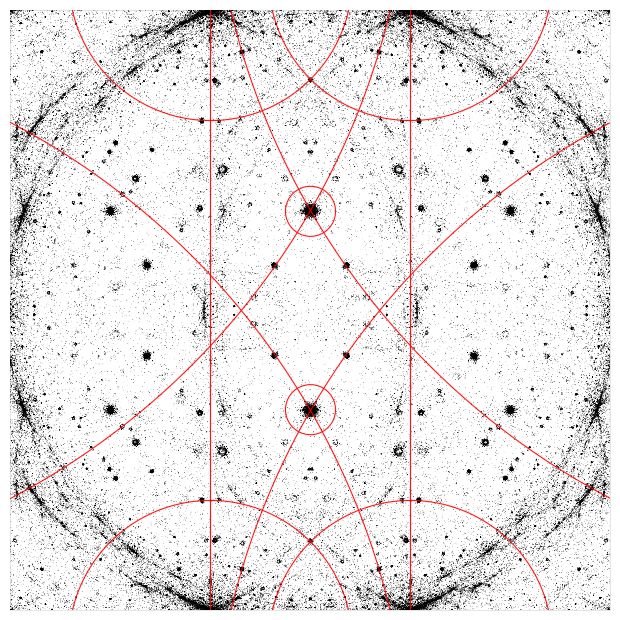}
    \caption*{$\theta = 10\pi/24$}
  \end{subfigure}%
  \begin{subfigure}{0.32\textwidth}
    \centering
    \includegraphics[width=\textwidth]{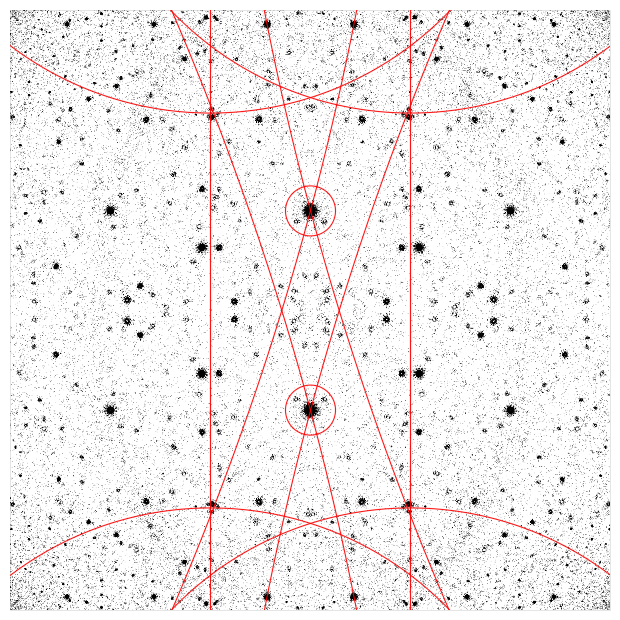}
    \caption*{$\theta = 11\pi/24$}
  \end{subfigure}%
  \begin{subfigure}{0.32\textwidth}
    \centering
    \includegraphics[width=\textwidth]{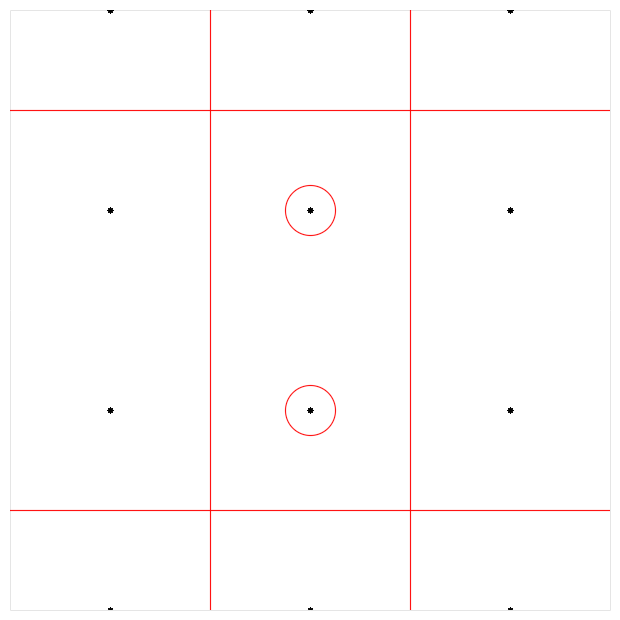}
    \caption*{$\theta = 12\pi/24$}
  \end{subfigure}
  \caption{Explicit deformation from a maximally degenerated group to a $(1;2)$-compression body; the cone angle in the quotient is $4\theta$. All axes are $ [-3,3] $.\label{fig:deformation_animation}}
\end{figure}

The deformation can be continued past the $ (1;2)$-compression body group, as shown in \zcref{fig:deformation_animation2}, although this continuation is of less interest
since it passes through a point where the isometric circles of $ J' $ degenerate to points and the deformation leaves $ \PSL(2,\C)$ momentarily. The endpoint of this
extended deformation is a group obtained from the starting function group by blowing up a puncture on the pinched end to become a hyperbolic and pinching down a handle
on the genus two end to become parabolic; the corresponding manifold is a compression body where one end is a twice-punctured torus and the other end is a four-punctured sphere.

\begin{figure}
  \begin{subfigure}{0.32\textwidth}
    \centering
    \includegraphics[width=\textwidth]{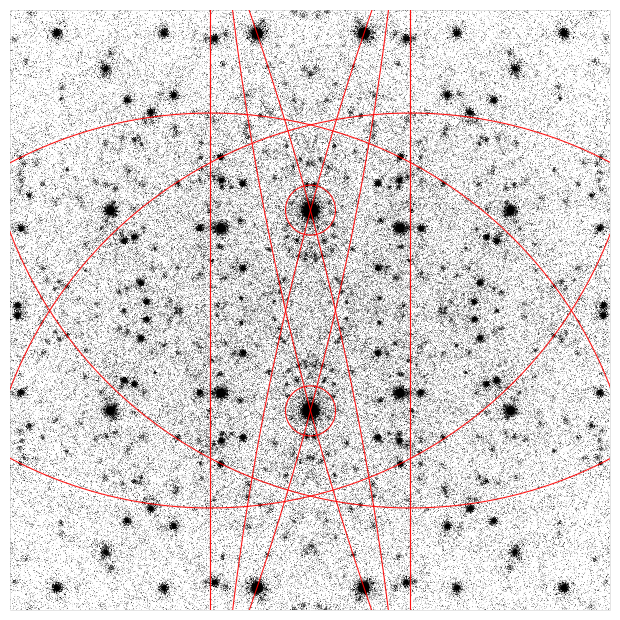}
    \caption*{$\theta = 13\pi/24$}
  \end{subfigure}%
  \begin{subfigure}{0.32\textwidth}
    \centering
    \includegraphics[width=\textwidth]{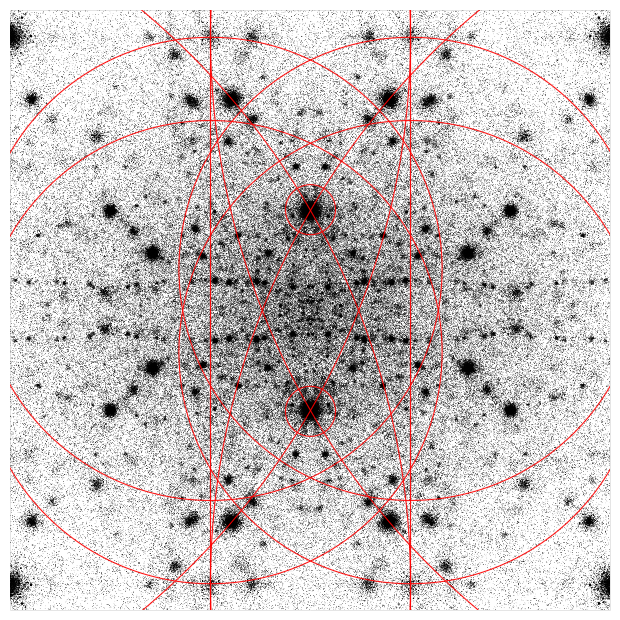}
    \caption*{$\theta = 14\pi/24$}
  \end{subfigure}%
  \begin{subfigure}{0.32\textwidth}
    \centering
    \includegraphics[width=\textwidth]{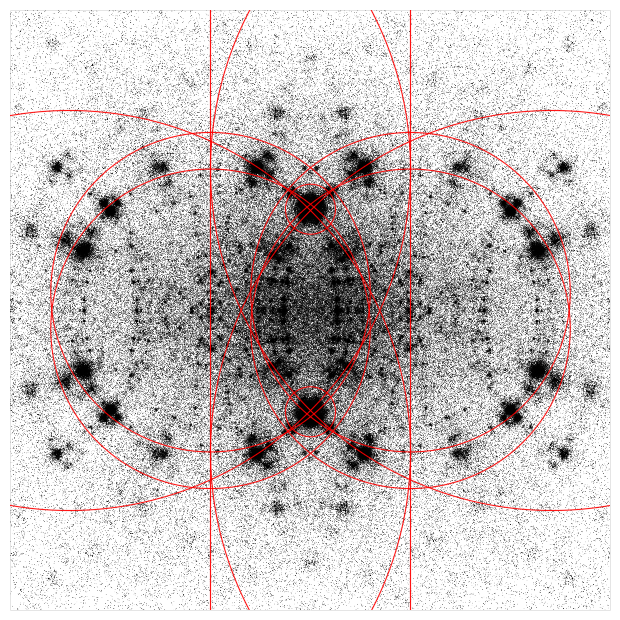}
    \caption*{$\theta = 15\pi/24$}
  \end{subfigure}\\
  \begin{subfigure}{0.32\textwidth}
    \centering
    \includegraphics[width=\textwidth]{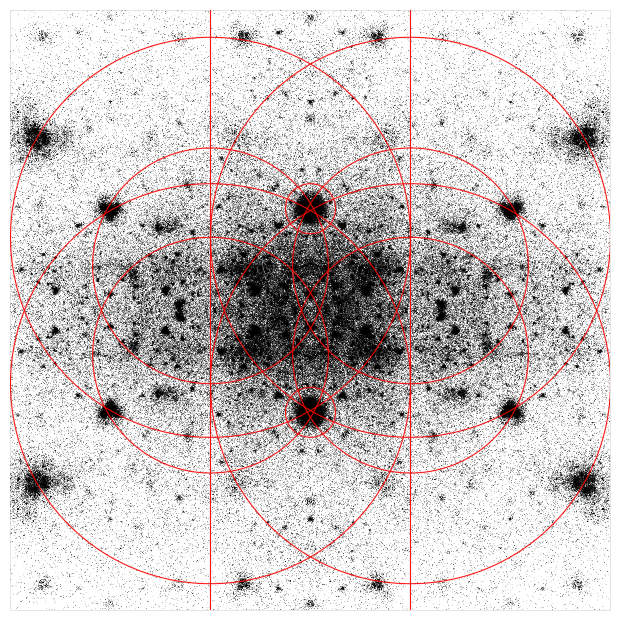}
    \caption*{$\theta = 16\pi/24$}
  \end{subfigure}%
  \begin{subfigure}{0.32\textwidth}
    \centering
    \includegraphics[width=\textwidth]{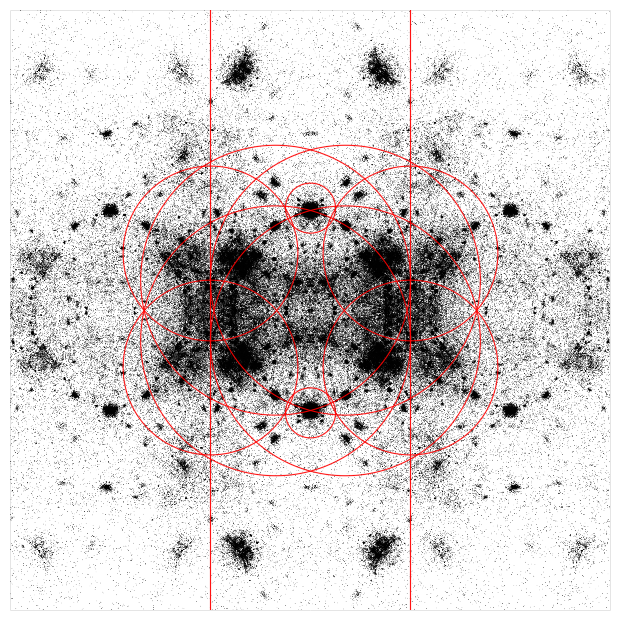}
    \caption*{$\theta = 17\pi/24$}
  \end{subfigure}%
  \begin{subfigure}{0.32\textwidth}
    \centering
    \includegraphics[width=\textwidth]{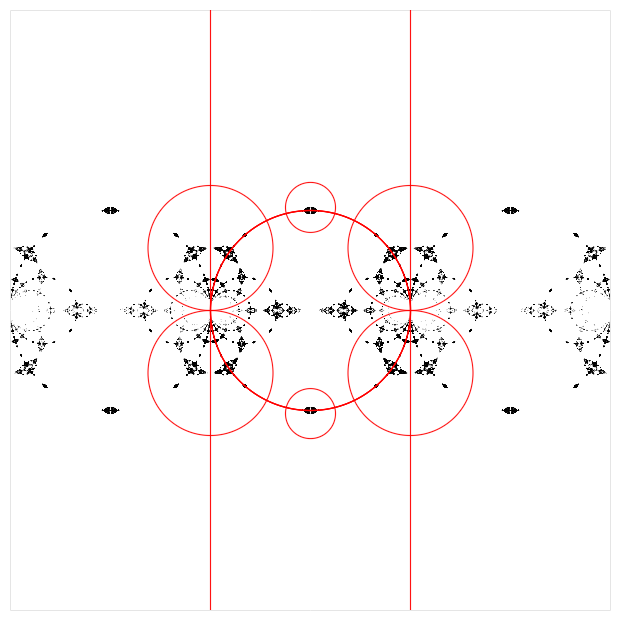}
    \caption*{$\theta = 18\pi/24$}
  \end{subfigure}\\
  \begin{subfigure}{0.32\textwidth}
    \centering
    \includegraphics[width=\textwidth]{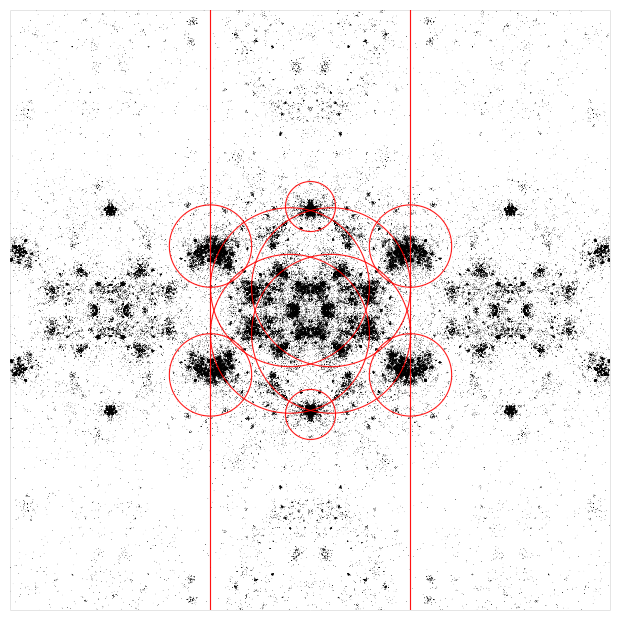}
    \caption*{$\theta = 19\pi/24$}
  \end{subfigure}%
  \begin{subfigure}{0.32\textwidth}
    \centering
    \includegraphics[width=\textwidth]{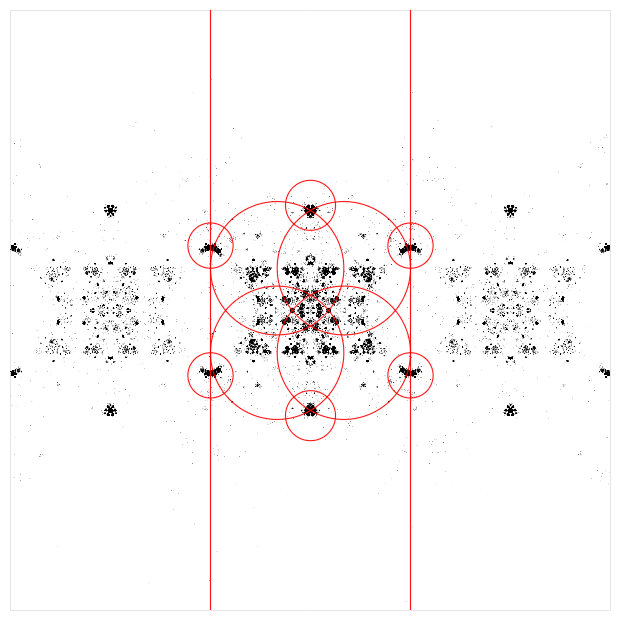}
    \caption*{$\theta = 20\pi/24$}
  \end{subfigure}%
  \begin{subfigure}{0.32\textwidth}
    \centering
    \includegraphics[width=\textwidth]{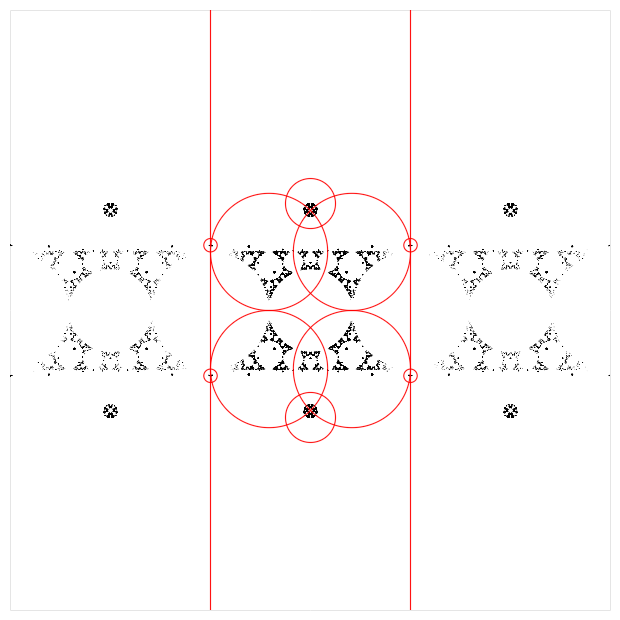}
    \caption*{$\theta = 21\pi/24$}
  \end{subfigure}\\
  \begin{subfigure}{0.32\textwidth}
    \centering
    \includegraphics[width=\textwidth]{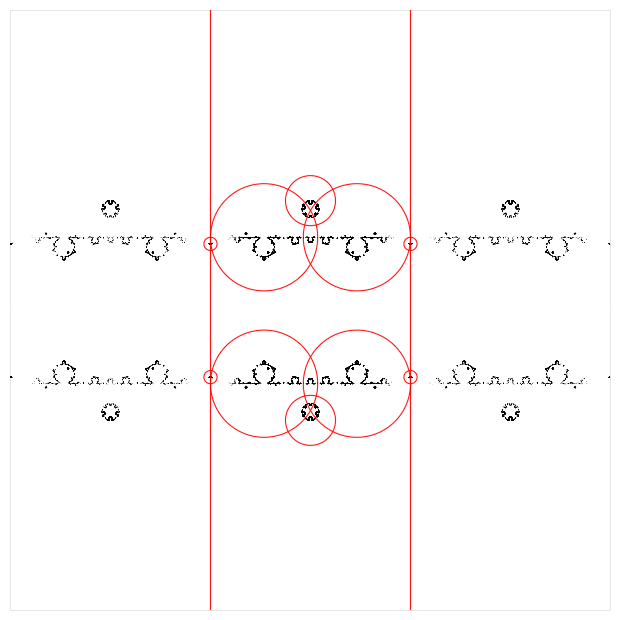}
    \caption*{$\theta = 22\pi/24$}
  \end{subfigure}%
  \begin{subfigure}{0.32\textwidth}
    \centering
    \includegraphics[width=\textwidth]{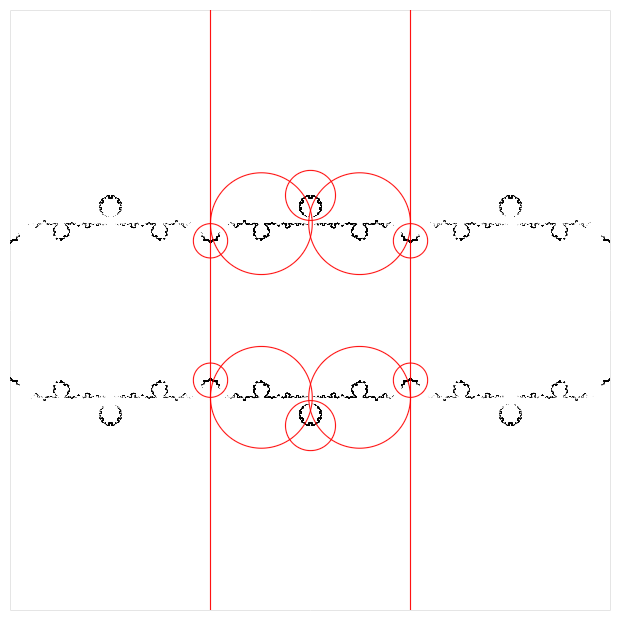}
    \caption*{$\theta = 23\pi/24$}
  \end{subfigure}%
  \begin{subfigure}{0.32\textwidth}
    \centering
    \includegraphics[width=\textwidth]{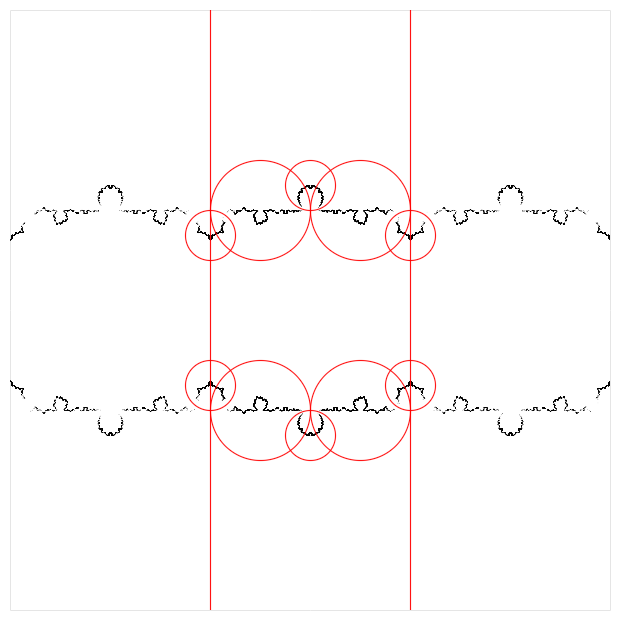}
    \caption*{$\theta = 24\pi/24$}
  \end{subfigure}
  \caption{Explicit deformation beyond the $(1;2)$-compression body; the cone angle in the quotient is $4\theta$. All axes are $ [-3,3] $.\label{fig:deformation_animation2}}
\end{figure}

The full animation may be viewed at
\begin{center} \url{https://aelzenaar.github.io/cones/cbdeformation.html}. \end{center}

\sloppy
\printbibliography

@article{akiyoshi18,
 author = {Akiyoshi, Hirotaka},
 title = {Thin representations for the one-cone torus group},
 fjournal = {Topology and its Applications},
 journal = {Topology Appl.},
 issn = {0166-8641},
 volume = {264},
 pages = {115--144},
 year = {2019},
 language = {English},
 doi = {10.1016/j.topol.2019.06.025},
 zbMATH = {7088362},
 Zbl = {1422.51006}
}

@book{beardon,
 author = {Beardon, Alan F.},
 title = {The geometry of discrete groups},
 fseries = {Graduate Texts in Mathematics},
 series = {Grad. Texts Math.},
 issn = {0072-5285},
 volume = {91},
 year = {1983},
 publisher = {Springer, Cham},
 language = {English},
 zbMATH = {3838373},
 Zbl = {0528.30001}
}

@incollection{bonahon02,
 author = {Bonahon, Francis},
 title = {Geometric structures on $3$-manifolds},
 booktitle = {Handbook of geometric topology},
 isbn = {0-444-82432-4},
 pages = {93--164},
 year = {2002},
 publisher = {Amsterdam: Elsevier},
 language = {English},
 zbMATH = {1714501},
 bookeditor = {Daverman, R. J. and Sher, R. B.},
 Zbl = {0997.57032}
}

@book{bridson_haefliger,
 author = {Bridson, Martin R. and Haefliger, Andr{\'e}},
 title = {Metric spaces of non-positive curvature},
 fseries = {Grundlehren der Mathematischen Wissenschaften},
 series = {Grundlehren Math. Wiss.},
 issn = {0072-7830},
 volume = {319},
 isbn = {3-540-64324-9},
 year = {1999},
 publisher = {Berlin: Springer},
 language = {English},
 zbMATH = {1385418},
 Zbl = {0988.53001}
}

@book{chk,
 author = {Cooper, Daryl and Hodgson, Craig D. and Kerckhoff, Steven P.},
 title = {Three-dimensional orbifolds and cone-manifolds},
 fseries = {MSJ Memoirs},
 series = {MSJ Mem.},
 issn = {2189-1494},
 volume = {5},
 isbn = {4-931469-05-1},
 year = {2000},
 publisher = {Tokyo: Mathematical Society of Japan (MSJ)},
 language = {English},
 doi = {10.2969/msjmemoirs/005010000},
 zbMATH = {1516794},
 Zbl = {0955.57014}
}

@article{gallo00,
 author = {Gallo, Daniel and Kapovich, Michael and Marden, Albert},
 title = {The monodromy groups of {Schwarzian} equations on closed {Riemann} surfaces},
 fjournal = {Annals of Mathematics. Second Series},
 journal = {Ann. Math. (2)},
 issn = {0003-486X},
 volume = {151},
 number = {2},
 pages = {625--704},
 year = {2000},
 language = {English},
 doi = {10.2307/121044},
 zbMATH = {1485628},
 Zbl = {0977.30028},
}

@article{lackenby14,
 author = {Lackenby, Marc and Purcell, Jessica S.},
 title = {Geodesics and compression bodies},
 fjournal = {Experimental Mathematics},
 journal = {Exp. Math.},
 issn = {1058-6458},
 volume = {23},
 number = {2},
 pages = {218--240},
 year = {2014},
 language = {English},
 doi = {10.1080/10586458.2013.870503},
 zbMATH = {6328245},
 Zbl = {1300.57015},
}

@book{maclachlan_reid,
 author = {Maclachlan, Colin and Reid, Alan W.},
 title = {The arithmetic of hyperbolic $3$-manifolds},
 fseries = {Graduate Texts in Mathematics},
 series = {Grad. Texts Math.},
 issn = {0072-5285},
 volume = {219},
 isbn = {0-387-98386-4; 978-1-4419-3122-1},
 year = {2003},
 publisher = {New York, NY: Springer},
 language = {English},
 zbMATH = {1859039},
 Zbl = {1025.57001}
}

@book{marden,
 author = {Marden, Albert},
 title = {Hyperbolic manifolds},
 subtitle = {An introduction in $2$ and $3$ dimensions},
 edition = {2nd edition of the book previously published under the title `{Outer} circles. {An} introduction to hyperbolic $3$-manifolds'},
 isbn = {978-1-107-11674-0; 978-1-316-33777-6},
 year = {2016},
 publisher = {Cambridge: Cambridge University Press},
 language = {English},
 doi = {10.1017/CBO9781316337776},
 zbMATH = {6459788},
 Zbl = {1355.57002}
}

@article{marden74g,
 author = {Marden, Albert},
 title = {The geometry of finitely generated {Kleinian} groups},
 fjournal = {Annals of Mathematics. Second Series},
 journal = {Ann. Math. (2)},
 issn = {0003-486X},
 volume = {99},
 pages = {383--462},
 year = {1974},
 language = {English},
 doi = {10.2307/1971059},
 zbMATH = {3442351},
 Zbl = {0282.30014}
}

@book{maskit,
 author = {Maskit, Bernard},
 title = {Kleinian groups},
 fseries = {Grundlehren der Mathematischen Wissenschaften},
 series = {Grundlehren Math. Wiss.},
 issn = {0072-7830},
 volume = {287},
 isbn = {3-540-17746-9},
 year = {1988},
 publisher = {Berlin etc.: Springer-Verlag},
 language = {English},
 zbMATH = {192943},
 Zbl = {0627.30039}
}

@article{mathews12,
 author = {Mathews, Daniel V.},
 title = {Hyperbolic cone-manifold structures with prescribed holonomy. {II}: {Higher} genus},
 fjournal = {Geometriae Dedicata},
 journal = {Geom. Dedicata},
 issn = {0046-5755},
 volume = {160},
 pages = {15--45},
 year = {2012},
 language = {English},
 doi = {10.1007/s10711-011-9668-y},
 zbMATH = {6101484},
 Zbl = {1256.57014},
}

@article{mathews11,
 author = {Mathews, Daniel V.},
 title = {Hyperbolic cone-manifold structures with prescribed holonomy. {I}: {Punctured} tori},
 fjournal = {Geometriae Dedicata},
 journal = {Geom. Dedicata},
 issn = {0046-5755},
 volume = {152},
 pages = {85--128},
 year = {2011},
 language = {English},
 doi = {10.1007/s10711-010-9547-y},
 zbMATH = {5903733},
 Zbl = {1252.57007},
}

@book{matsuzakitaniguchi,
 author = {Matsuzaki, Katsuhiko and Taniguchi, Masahiko},
 title = {Hyperbolic manifolds and {Kleinian} groups},
 fseries = {Oxford Mathematical Monographs},
 series = {Oxford Math. Monogr.},
 isbn = {0-19-850062-9},
 year = {1998},
 publisher = {Oxford: Clarendon Press},
 language = {English},
 zbMATH = {1145173},
 Zbl = {0892.30035}
}

@book{mostow,
 author = {Mostow, G. D.},
 title = {Strong rigidity of locally symmetric spaces},
 fseries = {Annals of Mathematics Studies},
 series = {Ann. Math. Stud.},
 volume = {78},
 year = {1973},
 publisher = {Princeton University Press, Princeton, NJ},
 language = {English},
 doi = {10.1515/9781400881833},
 zbMATH = {3418289},
 Zbl = {0265.53039}
}

@article{ohshika20,
 author = {Ohshika, Ken'ichi},
 title = {Divergence, exotic convergence and self-bumping in quasi-{Fuchsian} spaces},
 fjournal = {Annales de la Facult{\'e} des Sciences de Toulouse. Math{\'e}matiques. S{\'e}rie VI},
 journal = {Ann. Fac. Sci. Toulouse, Math. (6)},
 issn = {0240-2963},
 volume = {29},
 number = {4},
 pages = {805--895},
 year = {2020},
 language = {English},
 doi = {10.5802/afst.1647},
 zbMATH = {7283621},
 Zbl = {1461.30102},
}

@article{przytycki00,
 author = {Przytycki, J{\'o}zef H. and Sikora, Adam S.},
 title = {On skein algebras and {$ \SL_2(\C)$}-character varieties},
 fjournal = {Topology},
 journal = {Topology},
 issn = {0040-9383},
 volume = {39},
 number = {1},
 pages = {115--148},
 year = {2000},
 language = {English},
 doi = {10.1016/S0040-9383(98)00062-7},
 zbMATH = {1356314},
 Zbl = {0958.57011},
}

@article{thurston82,
 author = {Thurston, William P.},
 title = {Three dimensional manifolds, {Kleinian} groups and hyperbolic geometry},
 fjournal = {Bulletin of the American Mathematical Society. New Series},
 journal = {Bull. Am. Math. Soc., New Ser.},
 issn = {0273-0979},
 volume = {6},
 pages = {357--379},
 year = {1982},
 language = {English},
 doi = {10.1090/S0273-0979-1982-15003-0},
 zbMATH = {3782052},
 Zbl = {0496.57005}
}

@article{yoshida22,
 author = {Yoshida, Ken'ichi},
 title = {Degeneration of $3$-dimensional hyperbolic cone structures with decreasing cone angles},
 fjournal = {Conformal Geometry and Dynamics},
 journal = {Conform. Geom. Dyn.},
 issn = {1088-4173},
 volume = {26},
 pages = {182--193},
 year = {2022},
 language = {English},
 doi = {10.1090/ecgd/375},
 zbMATH = {7615749},
 Zbl = {1514.57031},
}

@article{prasad73,
 author = {Prasad, Gopal},
 title = {Strong rigidity of {$\Q$}-rank $1$ lattices},
 fjournal = {Inventiones Mathematicae},
 journal = {Invent. Math.},
 issn = {0020-9910},
 volume = {21},
 pages = {255--286},
 year = {1973},
 language = {English},
 doi = {10.1007/BF01418789},
 zbMATH = {3416116},
 Zbl = {0264.22009}
}

@article{hodgson98,
 author = {Hodgson, Craig D. and Kerckhoff, Steven P.},
 title = {Rigidity of hyperbolic cone-manifolds and hyperbolic {Dehn} surgery},
 fjournal = {Journal of Differential Geometry},
 journal = {J. Differ. Geom.},
 issn = {0022-040X},
 volume = {48},
 number = {1},
 pages = {1--59},
 year = {1998},
 language = {English},
 doi = {10.4310/jdg/1214460606},
 zbMATH = {1187443},
 Zbl = {0919.57009}
}

@article{weiss13,
 author = {Wei{\ss}, Hartmut},
 title = {The deformation theory of hyperbolic cone-$3$-manifolds with cone-angles less than {{\(2\pi\)}}},
 fjournal = {Geometry \& Topology},
 journal = {Geom. Topol.},
 issn = {1465-3060},
 volume = {17},
 number = {1},
 pages = {329--367},
 year = {2013},
 language = {English},
 doi = {10.2140/gt.2013.17.329},
 zbMATH = {6152265},
 Zbl = {1262.53032},
}

@article{weiss05,
 author = {Wei\ss, Hartmut},
 title = {Local rigidity of $3$-dimensional cone-manifolds},
 fjournal = {Journal of Differential Geometry},
 journal = {J. Differ. Geom.},
 issn = {0022-040X},
 volume = {71},
 number = {3},
 pages = {437--506},
 year = {2005},
 language = {English},
 doi = {10.4310/jdg/1143571990},
 zbMATH = {5033805},
 Zbl = {1098.53038},
}

@article{weiss07,
 author = {Wei\ss, Hartmut},
 title = {Global rigidity of $3$-dimensional cone-manifolds},
 fjournal = {Journal of Differential Geometry},
 journal = {J. Differ. Geom.},
 issn = {0022-040X},
 volume = {76},
 number = {3},
 pages = {495--523},
 year = {2007},
 language = {English},
 doi = {10.4310/jdg/1180135696},
 zbMATH = {5174441},
 Zbl = {1184.53049},
}

@article{kojima98,
 author = {Kojima, Sadayoshi},
 title = {Deformations of hyperbolic {{\(3\)}}-cone-manifolds},
 fjournal = {Journal of Differential Geometry},
 journal = {J. Differ. Geom.},
 issn = {0022-040X},
 volume = {49},
 number = {3},
 pages = {469--516},
 year = {1998},
 language = {English},
 doi = {10.4310/jdg/1214461108},
 zbMATH = {1725383},
 Zbl = {0990.57004},
}

@article{mazzeo11,
 author = {Mazzeo, Rafe and Montcouquiol, Gr{\'e}goire},
 title = {Infinitesimal rigidity of cone-manifolds and the {Stoker} problem for hyperbolic and {Euclidean} polyhedra},
 fjournal = {Journal of Differential Geometry},
 journal = {J. Differ. Geom.},
 issn = {0022-040X},
 volume = {87},
 number = {3},
 pages = {525--576},
 year = {2011},
 language = {English},
 doi = {10.4310/jdg/1312998235},
 zbMATH = {5961189},
 Zbl = {1234.53014},
}

@article{montcouquiol13,
 author = {Montcouquiol, Gr{\'e}goire},
 title = {Deformations of hyperbolic convex polyhedra and cone-$3$-manifolds},
 fjournal = {Geometriae Dedicata},
 journal = {Geom. Dedicata},
 issn = {0046-5755},
 volume = {166},
 pages = {163--183},
 year = {2013},
 language = {English},
 doi = {10.1007/s10711-012-9790-5},
 zbMATH = {6222296},
 Zbl = {1279.52015},
}

@article{cooper18,
 author = {Cooper, Daryl and Danciger, Jeffrey and Wienhard, Anna},
 title = {Limits of geometries},
 fjournal = {Transactions of the American Mathematical Society},
 journal = {Trans. Am. Math. Soc.},
 issn = {0002-9947},
 volume = {370},
 number = {9},
 pages = {6585--6627},
 year = {2018},
 language = {English},
 doi = {10.1090/tran/7174},
 zbMATH = {6891720},
 Zbl = {1395.57022},
}

@article{hodgson05,
 author = {Hodgson, Craig D. and Kerckhoff, Steven P.},
 title = {Universal bounds for hyperbolic {Dehn} surgery},
 fjournal = {Annals of Mathematics. Second Series},
 journal = {Ann. Math. (2)},
 issn = {0003-486X},
 volume = {162},
 number = {1},
 pages = {367--421},
 year = {2005},
 language = {English},
 doi = {10.4007/annals.2005.162.367},
 zbMATH = {5011497},
 Zbl = {1087.57011},
}

@article{futer22,
 author = {Futer, David and Purcell, Jessica S. and Schleimer, Saul},
 title = {Effective bilipschitz bounds on drilling and filling},
 fjournal = {Geometry \& Topology},
 journal = {Geom. Topol.},
 issn = {1465-3060},
 volume = {26},
 number = {3},
 pages = {1077--1188},
 year = {2022},
 language = {English},
 doi = {10.2140/gt.2022.26.1077},
 zbMATH = {7584041},
 Zbl = {1502.30126},
}

@article{futer22b,
 author = {Futer, David and Purcell, Jessica S. and Schleimer, Saul},
 title = {Effective drilling and filling of tame hyperbolic $3$-manifolds},
 fjournal = {Commentarii Mathematici Helvetici},
 journal = {Comment. Math. Helv.},
 issn = {0010-2571},
 volume = {97},
 number = {3},
 pages = {457--512},
 year = {2022},
 language = {English},
 doi = {10.4171/CMH/536},
 zbMATH = {7574046},
 Zbl = {1505.57026},
}

@misc{elzenaar25ej,
  author = {Alex Elzenaar},
  title = {Expansion joints in hyperbolic manifolds},
  year = {2025},
  eprint = {2512.00879},
  archivePrefix = {arXiv},
  primaryClass={math.GT}
}

@article{maskit75,
 author = {Maskit, Bernard},
 title = {On the classification of {Kleinian} groups. {I}: {Koebe} groups},
 fjournal = {Acta Mathematica},
 journal = {Acta Math.},
 issn = {0001-5962},
 volume = {135},
 pages = {249--270},
 year = {1975},
 language = {English},
 doi = {10.1007/BF02392021},
 zbMATH = {3510648},
 Zbl = {0326.30014}
}

@article{maskit77,
 author = {Maskit, Bernard},
 title = {On the classification of {Kleinian} groups. {II}: Signatures},
 fjournal = {Acta Mathematica},
 journal = {Acta Math.},
 issn = {0001-5962},
 volume = {138},
 pages = {17--42},
 year = {1977},
 language = {English},
 doi = {10.1007/BF02392312},
 zbMATH = {3558156},
 Zbl = {0358.30011}
}

@article{koebe12,
 author = {Koebe, Paul},
 title = {{\"U}ber die {Uniformisierung} der algebraischen {Kurven}. {III}. [{On} the uniformisation of algebraic curves. {III}.]},
 fjournal = {Mathematische Annalen},
 journal = {Math. Ann.},
 issn = {0025-5831},
 volume = {72},
 pages = {437--516},
 year = {1912},
 language = {German},
 doi = {10.1007/BF01456673},
 zbMATH = {2625974},
 JFM = {43.0522.02}
}

@article{klein82,
 author = {Klein, Felix},
 title = {Neue {B}eiträge zur {R}iemann'schen {F}unctionentheorie. [{New} contributions to {Riemann}'s function theory.]},
 fjournal = {Mathematische Annalen},
 journal = {Math. Ann.},
 issn = {0025-5831},
 volume = {21},
 pages = {141--218},
 year = {1882},
 language = {German},
 doi = {10.1007/BF01442920},
 zbMATH = {2703796},
 JFM = {15.0351.01}
}

@article{namazi12,
 author = {Namazi, Hossein and Souto, Juan},
 title = {Non-realizability and ending laminations: proof of the density conjecture},
 fjournal = {Acta Mathematica},
 journal = {Acta Math.},
 issn = {0001-5962},
 volume = {209},
 number = {2},
 pages = {323--395},
 year = {2012},
 language = {English},
 doi = {10.1007/s11511-012-0088-0},
 zbMATH = {6126651},
 Zbl = {1258.57010}
}

@article{ohshika11,
 author = {Ohshika, Ken'ichi},
 title = {Realising end invariants by limits of minimally parabolic, geometrically finite groups},
 fjournal = {Geometry \& Topology},
 journal = {Geom. Topol.},
 issn = {1465-3060},
 volume = {15},
 number = {2},
 pages = {827--890},
 year = {2011},
 language = {English},
 doi = {10.2140/gt.2011.15.827},
 zbMATH = {5908706},
 Zbl = {1241.30014},
}

@article{adams95,
 author = {Adams, Colin C.},
 title = {Unknotting tunnels in hyperbolic $3$-manifolds},
 fjournal = {Mathematische Annalen},
 journal = {Math. Ann.},
 issn = {0025-5831},
 volume = {302},
 number = {1},
 pages = {177--195},
 year = {1995},
 language = {English},
 doi = {10.1007/BF01444492},
 zbMATH = {769474},
 Zbl = {0830.57009}
}

@article{adams96,
 author = {Adams, Colin C. and Reid, Alan W.},
 title = {Unknotting tunnels in two-bridge knot and link complements},
 fjournal = {Commentarii Mathematici Helvetici},
 journal = {Comment. Math. Helv.},
 issn = {0010-2571},
 volume = {71},
 number = {4},
 pages = {617--627},
 year = {1996},
 language = {English},
 doi = {10.1007/BF02566439},
 zbMATH = {1008572},
 Zbl = {0873.57006}
}

@article{burton14,
 author = {Burton, Stephan D. and Purcell, Jessica S.},
 title = {Geodesic systems of tunnels in hyperbolic $3$-manifolds},
 fjournal = {Algebraic \& Geometric Topology},
 journal = {Algebr. Geom. Topol.},
 issn = {1472-2747},
 volume = {14},
 number = {2},
 pages = {925--952},
 year = {2014},
 language = {English},
 doi = {10.2140/agt.2014.14.925},
 zbMATH = {6272615},
 Zbl = {1286.57014},
}

@incollection{maskit74,
 author = {Maskit, Bernard},
 title = {Good and bad {Kleinian} groups},
 language = {English},
 editor = {Bers, Lipman and Kra, Irwin},
  booktitle = {A crash course on {Kleinian} groups. {Lectures} given at a special session at the {January} 1974 meeting of the {American} {Mathematical} {Society} at {San} {Francisco}},
 fseries = {Lecture Notes in Mathematics},
 series = {Lect. Notes Math.},
 pages = {94--107},
 issn = {0075-8434},
 volume = {400},
 year = {1974},
 publisher = {Springer, Cham},
 zbMATH = {3470766},
 Zbl = {0301.30016}
}

@article{maher21,
 author = {Maher, Joseph and Schleimer, Saul},
 title = {The compression body graph has infinite diameter},
 fjournal = {Algebraic \& Geometric Topology},
 journal = {Algebr. Geom. Topol.},
 issn = {1472-2747},
 volume = {21},
 number = {4},
 pages = {1817--1856},
 year = {2021},
 language = {English},
 doi = {10.2140/agt.2021.21.1817},
 zbMATH = {7394074},
 Zbl = {1480.57027},
}

@article{keen93,
 author = {Keen, Linda and Maskit, Bernard and Series, Caroline},
 title = {Geometric finiteness and uniqueness for {Kleinian} groups with circle packing limit sets},
 fjournal = {Journal f{\"u}r die Reine und Angewandte Mathematik},
 journal = {J. Reine Angew. Math.},
 issn = {0075-4102},
 volume = {436},
 pages = {209--219},
 year = {1993},
 language = {English},
 url = {https://eudml.org/doc/153503},
 zbMATH = {146273},
 Zbl = {0764.30035},
}

@article{bromberg11,
 author = {Bromberg, Kenneth},
 title = {The space of {Kleinian} punctured torus groups is not locally connected},
 fjournal = {Duke Mathematical Journal},
 journal = {Duke Math. J.},
 issn = {0012-7094},
 volume = {156},
 number = {3},
 pages = {387--427},
 year = {2011},
 language = {English},
 doi = {10.1215/00127094-2010-215},
 zbMATH = {5868009},
 Zbl = {1213.30078},
}

@article{magid12,
 author = {Magid, Aaron D.},
 title = {Deformation spaces of {Kleinian} surface groups are not locally connected},
 fjournal = {Geometry \& Topology},
 journal = {Geom. Topol.},
 issn = {1465-3060},
 volume = {16},
 number = {3},
 pages = {1247--1320},
 year = {2012},
 language = {English},
 doi = {10.2140/gt.2012.16.1247},
 zbMATH = {6068625},
 Zbl = {1257.57023},
}

@article{dey25,
 author = {Dey, Subhadip and Kapovich, Michael},
 title = {{Klein}--{Maskit} combination theorem for {Anosov} subgroups: amalgams},
 fjournal = {Journal f{\"u}r die Reine und Angewandte Mathematik},
 journal = {J. Reine Angew. Math.},
 issn = {0075-4102},
 volume = {819},
 pages = {1--43},
 year = {2025},
 language = {English},
 doi = {10.1515/crelle-2024-0080},
 zbMATH = {7982959},
 Zbl = {1560.22029},
}

@book{fenchel,
 author = {Fenchel, Werner and Nielsen, Jakob},
 editor = {Schmidt, Asmus L.},
 title = {Discontinuous groups of isometries in the hyperbolic plane},
 fseries = {De Gruyter Studies in Mathematics},
 series = {De Gruyter Stud. Math.},
 issn = {0179-0986},
 volume = {29},
 isbn = {3-11-017526-6},
 year = {2003},
 publisher = {Berlin: Walter de Gruyter},
 language = {English},
 zbMATH = {1873123},
 Zbl = {1022.51016}
}

@article{maskit93,
 author = {Maskit, Bernard},
 title = {On {Klein}'s combination theorem. {IV}},
 fjournal = {Transactions of the American Mathematical Society},
 journal = {Trans. Am. Math. Soc.},
 issn = {0002-9947},
 volume = {336},
 number = {1},
 pages = {265--294},
 year = {1993},
 language = {English},
 doi = {10.2307/2154347},
 zbMATH = {169775},
 Zbl = {0777.30028}
}

@misc{danciger24,
  author = {Jeffrey Danciger and François Guéritaud and Fanny Kassel},
  title = {Combination theorems in convex projective geometry},
  year = {2024},
  eprint = {2407.09439},
  archivePrefix = {arXiv},
  primaryClass={math.GR}
}

@misc{chen25,
  author = {Qiyu Chen and Youliang Zhong},
  title = {Circular foliations and shear-radius coordinates on {Teichmüller} spaces of hyperbolic cone surfaces},
  year = {2025},
  eprint = {2512.21068},
  archivePrefix = {arXiv},
  primaryClass={math.GT}
}

@article{mcowen88,
 author = {McOwen, Robert C.},
 title = {Point singularities and conformal metrics on {Riemann} surfaces},
 fjournal = {Proceedings of the American Mathematical Society},
 journal = {Proc. Am. Math. Soc.},
 issn = {0002-9939},
 volume = {103},
 number = {1},
 pages = {222--224},
 year = {1988},
 language = {English},
 doi = {10.2307/2047555},
 zbMATH = {4073390},
 Zbl = {0657.30033}
}

@article{troyanov91,
 author = {Troyanov, Marc},
 title = {Prescribing curvature on compact surfaces with conical singularities},
 fjournal = {Transactions of the American Mathematical Society},
 journal = {Trans. Am. Math. Soc.},
 issn = {0002-9947},
 volume = {324},
 number = {2},
 pages = {793--821},
 year = {1991},
 language = {English},
 doi = {10.2307/2001742},
 zbMATH = {4194621},
 Zbl = {0724.53023}
}

@article{heins62,
 author = {Heins, Maurice},
 title = {On a class of conformal metrics},
 fjournal = {Nagoya Mathematical Journal},
 journal = {Nagoya Math. J.},
 issn = {0027-7630},
 volume = {21},
 pages = {1--60},
 year = {1962},
 language = {English},
 doi = {10.1017/S002776300002376X},
 zbMATH = {3183221},
 Zbl = {0113.05603}
}

@article{feng25,
 author = {Feng, Yu and Xu, Bin},
 title = {Stable parabolic {Higgs} bundles of rank two and singular hyperbolic metrics},
 fjournal = {Journal of Geometry and Physics},
 journal = {J. Geom. Phys.},
 issn = {0393-0440},
 volume = {218},
 eid = {105679},
 year = {2025},
 language = {English},
 doi = {10.1016/j.geomphys.2025.105679},
 zbMATH = {8120211},
}

@article{maskit65p,
 author = {Maskit, Bernard},
 title = {A theorem on planar covering surfaces with applications to 3-manifolds},
 fjournal = {Annals of Mathematics. Second Series},
 journal = {Ann. Math. (2)},
 issn = {0003-486X},
 volume = {81},
 pages = {341--355},
 year = {1965},
 language = {English},
 doi = {10.2307/1970619},
 zbMATH = {3243252},
 Zbl = {0151.33003}
}

@article{bowditch22,
 author = {Bowditch, Brian H.},
 title = {Notes on {Maskit}'s planarity theorem},
 fjournal = {L'Enseignement Math{\'e}matique. 2e S{\'e}rie},
 journal = {Enseign. Math. (2)},
 issn = {0013-8584},
 volume = {68},
 number = {1-2},
 pages = {1--24},
 year = {2022},
 language = {English},
 doi = {10.4171/LEM/1019},
 zbMATH = {7556764},
 Zbl = {1495.57015}
}

\end{document}